\documentclass[onefignum,onetabnum]{siamart220329}
\usepackage{subfigure}
\usepackage{ulem}
\usepackage{cancel}

\newcommand{\R}{{\mat R}}

\newcommand{\be}{\begin{eqnarray}}
\newcommand{\ben}{\begin{eqnarray*}}
\newcommand{\en}{\end{eqnarray}}
\newcommand{\enn}{\end{eqnarray*}}

\newcommand{\ov}{\overline}

\newcommand{\sig}{\sigma}

\newcommand{\mat}{\mathbb}

\definecolor{TJL}{rgb}{0,0,0}
\definecolor{lxl}{rgb}{0,0,0}
\usepackage{lipsum}
\usepackage{amsfonts}
\usepackage{graphicx}
\usepackage{epstopdf}
\usepackage{algorithmic}
\ifpdf
  \DeclareGraphicsExtensions{.eps,.pdf,.png,.jpg}
\else
  \DeclareGraphicsExtensions{.eps}
\fi

\newsiamremark{remark}{Remark}
\newsiamremark{hypothesis}{Hypothesis}
\crefname{hypothesis}{Hypothesis}{Hypotheses}
\newsiamthm{claim}{Claim}
\definecolor{TJL}{rgb}{0,0,0}
\headers{Time domain inverse problem with {\color{TJL}passive data}}{X. Liu, S. Meng, J. Tian and B. Zhang}

\title{An inverse obstacle scattering problem with {\color{TJL}passive data} in the time domain\thanks{Submitted to the editors on December 31, 2024; revised May 31, 2025.}}

\author{Xiaoli Liu\thanks{School of Mathematical Sciences, Beihang University, Beijing 100191, China. The work of Xiaoli Liu was partially supported by
the NNSF of China (No. 12201023), the Fundamental Research Funds for the Central Universities (No. YWF-23-Q-1026 and YWF-22-T-204).(\email{xiaoli\_liu@buaa.edu.cn}).}
\and Shixu Meng\thanks{Department of Mathematics, Virginia Tech,
Blacksburg, VA 24061 USA
  (\email{sgl22@vt.edu}).}
\and Jialu Tian\thanks{Corresponding author. Academy of Mathematics and Systems Science, Chinese Academy of Sciences, Beijing, 100190, China and School of Mathematical Sciences, University of Chinese Academy of Sciences, Beijing 100049, China
(\email{tianjialu@amss.ac.cn}).}
\and Bo Zhang
 \thanks{Academy of Mathematics and Systems Science, Chinese Academy of Sciences, Beijing, 100190, China and School of Mathematical Sciences, University of Chinese Academy of Sciences, Beijing 100049, China. The work of Bo Zhang was supported by the NNSF of China (No. 12431016).(\email{b.zhang@amt.ac.cn})}}

\usepackage{amsopn}

\begin{document}

\maketitle
\begin{abstract}
This work considers a time domain inverse acoustic obstacle scattering problem due to {passive data}. Motivated by the Helmholtz-Kirchhoff identity in the frequency domain, we propose to relate the time domain measurement data {in passive imaging} to an approximate data set given by the subtraction of two scattered wave fields. We propose a time domain linear sampling method for the approximate data set and show how to tackle the measurement data {in passive imaging}. An imaging functional is built based on the linear sampling method, which reconstructs the support of the unknown scattering object using directly the time domain measurements. The functional framework is based on the Laplace transform, which relates the mapping properties of Laplace domain factorized operators to their counterparts in the time domain. Numerical examples are provided to illustrate the capability of the proposed method.
\end{abstract}

\begin{keywords}
time domain, inverse acoustic scattering, passive imaging, linear sampling method
\end{keywords}

\begin{MSCcodes}
35P25, 35R30, 35J05, 65M32
\end{MSCcodes}

\section{Introduction}\label{sec:1}
\numberwithin{equation}{section}
\setcounter{equation}{0}
Inverse scattering merits important applications in medical imaging, non-destructive testing, marine seismic imaging and many other areas.
{\color{lxl}In \textit{active imaging}, both sources and receivers are controlled, while in \textit{passive imaging}, only receivers are employed and the illumination comes from unknown and uncontrolled sources \cite{GarnierPapanicolaou2016}.} Passive imaging finds important applications in {\color{lxl}oceanography \cite{Siderius2010}, structural health monitoring \cite{Sabra2011}, exploration \cite{Hollis2018} and elastography \cite{Gallot2011}} and others.  

To formally introduce the time domain scattering problem, let $D$ be the bounded domain occupied by the obstacles in $\mathbb{R}^3$ and the exterior domain $\mathbb{R}^3\backslash \overline{D}$  is supposed to be connected. The boundary, denote by $\partial D$ is supposed to be of class $C^{2, \alpha}$. 
$B_R$ is a ball with a sufficiently large radius $R$ that contains $D$ and some bounded volume $B\subset \mathbb{R}^3\backslash \overline{D}$. Denote by  $d:=\text{dist}\{\overline{B}, \overline{D}\}$ as the distance between   $D$ and  $B$ and it is assumed that $d>0$.
Some {\color{lxl}unknown} sources are distributed randomly on the surface $\partial B_R$ and transmit a time signal $\chi(t)$. For a sketch of the geometry setting we refer to Figure \ref{geo}. 
\begin{figure}[htbp]  %
    \centering  %
    \includegraphics[width=0.4\textwidth]{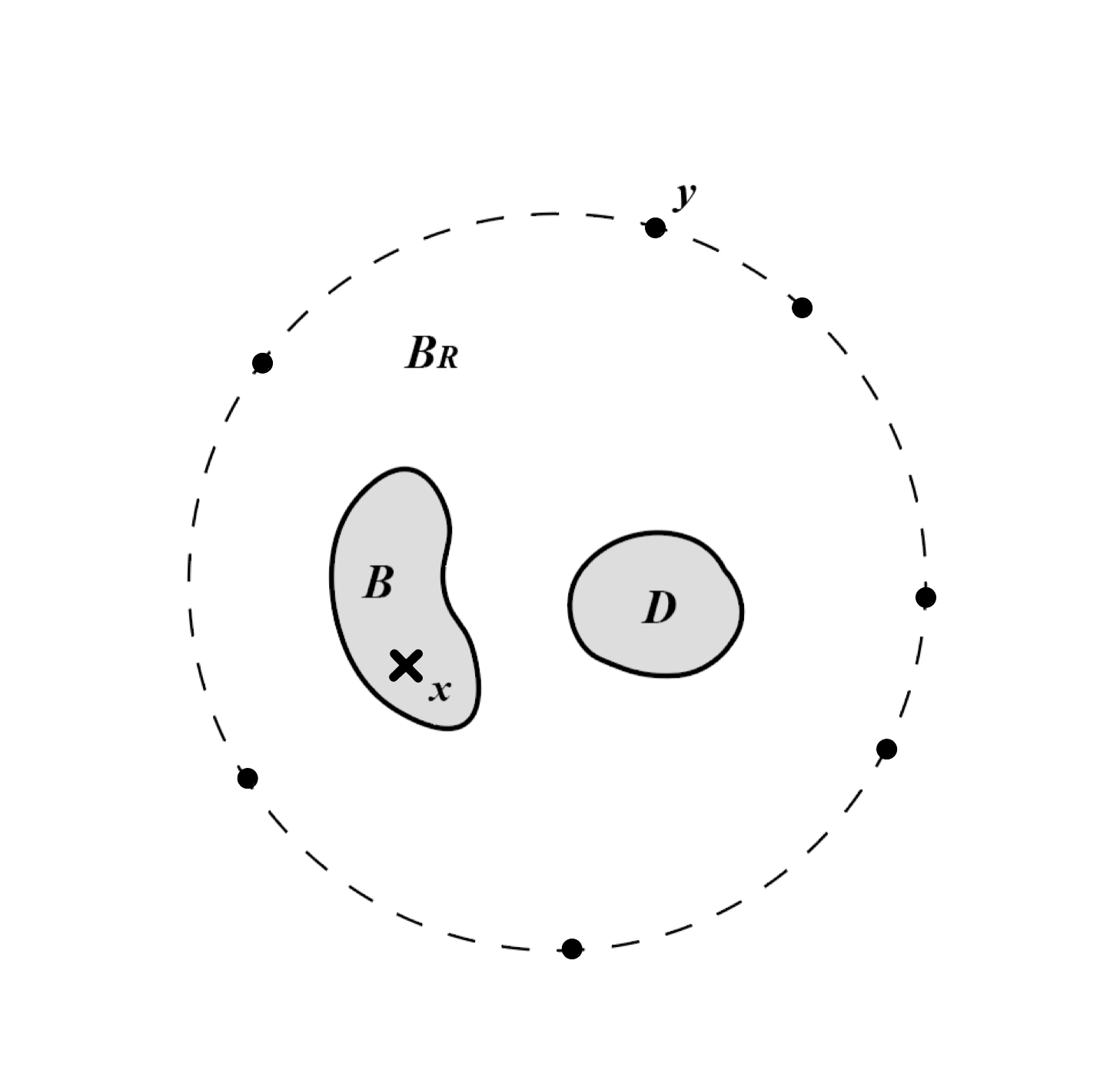}  %
    \caption{The sound-soft obstacle $D$ is illuminated by uncontrolled sources located randomly at $y\in\partial B_R$. The corresponding total fields are measured at $x \in B$.}  %
    \label{geo}  %
\end{figure}
Let the fundamental solution be given by
\[ 
\Phi(t,x;y):=\frac{\delta(t-|x-y|)}{4\pi|x-y|},\quad (t,x)\in\R\times\left\{\R^3\backslash\{y\}\right\}.
\]
{\color{lxl}Denote by $*$ the time convolution of two functions}
and consider an incident wave field 
\be\label{ui}
u^{\rm inc}(t,x;y):=\left[\chi(\cdot)*\Phi(\cdot,x;y)\right](t),\quad (t,x)\in\R\times\left\{\R^3\backslash\{y\}\right\},\quad y\in\partial B_R,
\en 
which vanishes for $t\leq 0$. 
Then the total field $u$ corresponding to an incident wave field $u^{\rm inc}(t,x;y)$ also vanishes for $t\leq 0$ and satisfies  
\be
\left\{
\begin{array}{ll}
\label{waveeq}\partial_t^2 u(t,x;y)-\Delta u(t,x;y) = \delta(x-y)\chi(t),  &(t,x)\in \mathbb{R}\times \left\{\mathbb{R}^3\backslash \overline{D}\right\},\\
u(t,x;y)=0,  &(t,x)\in{\mathbb{R}}\times \partial D.
\end{array}
\right.
\en
Equivalently, the scattered field $u^{\rm scat}=u-u^{\rm inc}$ vanishes for $t\leq 0$ and satisfies 
\be
\left\{
\begin{array}{ll}
\label{waveeq2}\partial_t^2 u^{\rm scat}(t,x;y)-\Delta u^{\rm scat}(t,x;y) = 0,  &(t,x)\in \mathbb{R}\times \left\{\mathbb{R}^3\backslash \overline{D}\right\},\\
u^{\rm scat}(t,x;y)=f(t,x),  &(t,x)\in{\mathbb{R}}\times \partial D,
\end{array}
\right.
\en
where the boundary data $f(t,x;y)$ is related to the incident waves by
\ben
f(t,x;y)=-u^{\rm inc}(t,x;y),\quad (t,x)\in{\mathbb{R}}\times \partial D.
\enn

The {\it forward scattering problem} is to find a solution $u$ satisfying (\ref{waveeq}) corresponding to an incident wave $u^{\rm inc}$ in \eqref{ui}. The {\it inverse scattering problem}   is to reconstruct $D$ from the knowledge of the {\color{lxl}\textit{passive data}}
\begin{equation} \label{data time domain total}
 \left\{u(t,x;y) : (t, x, y)\in \R\times B\times \partial B_R\right\},   
\end{equation} 
{\color{lxl}where $u(t,x;y)$ represents the total field measured at $x\in B$ corresponding to an incident point source $u^{\rm inc}(t,x;y)$ in \eqref{ui} located at $y \in \partial B_R$. Since we consider \textit{passive data} in this paper, meaning that the exact location $y$ of the incident point source in \eqref{data time domain total} is unknown. Here we still keep $y$ explicitly in the formulation for the convenience of subsequent theoretical analysis.

Our aim in this work is to deal with this inverse scattering problem with passive data \eqref{data time domain total}. Motivated by \cite{Garnier2023}, we will establish a Helmholtz-Kirchhoff-type identity in the time domain \eqref{eq:2} to relate the passive data \eqref{data time domain total} with some \textit{active data} 
\begin{equation} \label{data time domain scattered}
\left\{{u}_{\tilde{\chi}}^{\rm scat}(t,p;q)-{\Breve{u}}_{\tilde{\chi}}^{\rm scat}(t,p;q): (t, p, q)\in \R\times B\times B \right\},
\end{equation}
where ${u}_{\tilde{\chi}}^{\rm scat}(t,p;q)$ and ${\Breve{u}}_{\tilde{\chi}}^{\rm scat}(t,p;q)$ represent the scattered field measured at $p\in B$ corresponding to some incident point sources located at $q\in B$. We refer to \eqref{data time domain scattered} as \textit{active data} since here the location $q$ of the incident point source is known. In order to clearly distinguish between passive and active data in the subsequent analysis, we use $p$ and $q$ in the active data \eqref{data time domain scattered} to represent the receiver and source positions instead of $x$ and $y$ in the passive data \eqref{data time domain total}. From this perspective, it suffices to propose a time domain linear sampling method to solve the inverse scattering problem using with the active data \eqref{data time domain scattered} that can be deduced by the passive data \eqref{data time domain total}.}

{\color{lxl}The linear sampling method (LSM) was first established by Colton and Kirsch in \cite{ColtonKirsch1996} for solving inverse acoustic scattering problems in the frequency domain. Then it has been extended to electromagnetic cases in \cite{HaddarMonk2002}. After years of development, the linear sampling method in the frequency domain with active data has been widely studied and systematically expounded in several monographs \cite{CakoniColton2014,CakoniColtonHaddar2016,ColtonKress2019}. Among these, monograph \cite{ColtonKress2019} provides a thorough analysis on LSM for inverse scattering problems with Dirichlet boundary conditions in three dimensional case, while monograph \cite{CakoniColton2014} focuses on impedance boundary conditions in two dimensional case, presenting detailed theoretical analysis and numerical examples. As one of the sampling-type method, the LSM demonstrates notable advantages such as requiring little prior information and being independent of boundary conditions in both two and three dimensions. Some other sampling-type methods, such as the factorization method \cite{KirschGrinberg2007}, the generalized linear sampling method \cite{AudibertHaddar2014} and other approaches \cite{AudibertMeng2014,Ikehata1998,Potthast2000}, have also been developed with active data. Recently, the  work \cite{Garnier2023} studied the reconstruction of the obstacle using linear sampling method with random sources in the frequency domain, see \cite{Garnier2024} for random scatterers.

Solving inverse scattering problems with time domain active data has recently been widely studied and gained fruitful results. Chen et al. \cite{ChenHaddar2010} proposed the time domain linear sampling method (TDLSM) for solving inverse scattering problems by sound-soft bounded obstacles. The works \cite{HaddarLechleiter2014,CakoniMonkSelgas2021} extended this approach to other boundary conditions and inverse medium scattering problems. Factorization method has also been developed to time domain inverse scattering problems, see \cite{CakoniHaddarLechleiter2019} for Dirichlet boundary conditions and \cite{HaddarLiu2020} for Robin and Neumann boundary conditions. The most recent advancement by \cite{GuoLiWang2024} introduced a direct imaging method for bounded obstacle reconstruction using time domain measurements. Beyond these approaches, other sampling-type methods, including the point source method \cite{LukePotthast2006} and the enclosure method \cite{Ikehata2015}, have also been successfully adapted to time domain inverse scattering problems.}  

 In contrast to the above work in the time domain {\color{lxl}with active data},  this work aims to study a linear sampling method using time domain {\color{lxl}passive data that corresponding to unknown and} uncontrolled sources. The work \cite{Garnier2023} applied the Helmholtz-Kirchhoff identity to relate the measurements due to random sources to another set of measurements that is suitable for imaging, however difficulty still remains for the time domain measurements. Motivated by the work \cite{Garnier2023}, we relate the measurement data set  \eqref{data time domain total} due to {\color{lxl}unknown sources}  to an approximate data set  \eqref{data time domain scattered} given by the subtraction of two scattered wave fields. With the approximate data set  \eqref{data time domain scattered},  we use similar techniques as in \cite{CakoniHaddarLechleiter2019} to relate the mapping properties of the Laplace domain factorized operators to their counterparts in the time domain, enabling an imaging indicator using directly the time domain measurements.  One potential advantage of using time domain measurements is to avoid  certain difficulties in using frequency domain measurements at multiple frequencies \cite{guzina2010a}.

The remaining of the paper is as follows.
We first provide the functional framework in Section \ref{sec: functional framework}, including the Laplace transform and functional spaces in the time domain. To study the time domain linear sampling method, we analyze relevant Laplace domain operators in Section  \ref{sec:4}. We then relate the mapping properties of Laplace domain factorized operators to their counterparts in the time domain and develop a time domain linear sampling method  in Section \ref{sec:5}. Numerical examples are provided in Section \ref{sec:6} to test the proposed method, for both data set given by \eqref{data time domain scattered} and \eqref{data time domain total}.

\section{Functional framework} \label{sec: functional framework}
\numberwithin{equation}{section}
\setcounter{equation}{0}
To study the time domain scattering and inverse scattering, we denote by $\mathcal{D}( \R,X)=C^\infty_0(\R,X)$ the set of smooth and compactly supported $X$-valued functions for a Hilbert space $X$. Further, $\mathcal{D}'(\R;X)$ are $X$-valued distributions on the real line, and $\mathcal{S}'(\R;X)$ are the corresponding tempered distributions. For $\sigma\in\R$, we set
\begin{equation*}
    { \mathcal{L}'_\sigma(\R;X)}=\left\{f\in \mathcal{D}'(\R;X)~|~e^{-\sigma t}f(t)\in \mathcal{S}'(\R;X)\right\}.
\end{equation*}
For $f\in \mathcal{L}'_\sigma(\R;X)$, define the {Fourier-Laplace transform} with respect to the time variable as 
\begin{equation*}
    \mathcal{L}[f](s)=\hat{f}(s)=\int_{-\infty}^{\infty}e^{ist}f(t)dt,\quad s= k +i\sig.
\end{equation*}
We introduce the Hilbert space
\begin{equation*}
    H^p_{\sigma}(\R,X):=\left\{f\in\mathcal{L}'_\sigma(\R;X),\,\int_{-\infty+i\sigma}^{\infty+i\sigma}|s|^{2p}\Vert\mathcal{L}[f](s)\Vert^2_{X}ds< \infty\right\},\ p \in\mathbb{R},
\end{equation*}
endowed with the norm
\begin{equation*}
    \Vert f\Vert_{H^p_{\sigma}(\mathbb{R},\, X)}=\left(\int_{-\infty+i\sigma}^{\infty+i\sigma}|s|^{2p}\Vert\mathcal{L}[f](s)\Vert^2_{X}ds\right)^{1/2}
\end{equation*}
and the corresponding inner product. For $f \in H^p_{\sigma}(\R,X)$ and $g \in H^{-p}_{\sigma}(\R,X^*)$, we introduce  the following duality 
$$
{ \langle g, f\rangle_\sigma}= \int_{-\infty + i \sigma}^{\infty + i \sigma} \langle \mathcal{L}[g](s), \mathcal{L}[f](s) \rangle_{X^*,X} ds = \int_{-\infty}^\infty e^{-2\sigma t}\langle g(t), f(t)\rangle_{X^*,X} dt.
$$
More details can be found in \cite{Bamberger1986,Sayas2016}.

Formally applying the Laplace transform to \eqref{waveeq2}, the corresponding scattered field $\hat{u}^{\rm scat}(s,x;y)$ in the Laplace domain satisfy the following equation
\be
\left\{
\begin{array}{ll}
\label{helmeq2}\Delta  \hat{u}^{\rm scat}(s,x;y) +s^2  \hat{u}^{\rm scat}(s,x,y)=0,  &x\in \mathbb{R}^3\backslash\ov{D},\\
\hat{u}^{\rm scat}(s,x;y)=\hat{f}(s,x;y),  &x\in\partial D,
\end{array}
\right.
\en
subject to the  boundary condition
\ben
\hat{f}(s,x;y)=-\hat{u}^{\rm inc}(s,x;y)=-\hat{\chi}(s)\hat{\Phi}_s(x;y).
\enn
Here the fundamental solution in the Laplace domain is given by
$$
\hat\Phi_s(x;y):=\frac{e^{is|x-y|}}{4\pi|x-y|},\quad x\neq y.
$$
To have a physically meaningful solution, we require that $\hat{u}^{\rm scat}(s,x;y) \in H^1_s(\R^3\backslash\ov{D})$, where 
$$
H^1_s(\Omega) = \left\{ w:  \int_\Omega \left(|\nabla w|^2+|sw|^2\right)dx < \infty \right\}
$$
for any Lipschitz domain $\Omega$, the corresponding norm is given by
\ben
\|w\|_{H^1_s(\Omega)}:=\left(\int_\Omega\left(|\nabla w|^2+|sw|^2\right)dx\right)^{1/2}.
\enn
The forward scattering problem is well-posed; we state the following Lemma \cite[Proposition 1]{Bamberger1986}.
\begin{lemma}\label{thm4}
    Let $s=k+i\sigma, \sigma>\sigma_0>0$. For any $\hat{f}\in{H^{{1}/{2}}(\partial D)}$, the scattering problem \eqref{helmeq2} has a unique solution in $H^1_s(\R^3\backslash\ov{D})$. Moreover, 
    \begin{equation}\label{eq22}
        \Vert \hat{u}^{\rm scat}\Vert^2_{H^1_s(\R^3\backslash\ov{D})}\leq C\frac{1}{{\sigma}^2}\mathrm{max}\left(\frac{1}{\sigma_0}, 1\right)|s|^3\Vert \hat{f}\Vert^2_{H^{{1}/{2}}(\partial D)},
    \end{equation}
    where the constant $C$ only depends on $\partial D$. 
\end{lemma}
For later purposes, we discuss how to transform operators in the Laplace domain to the time domain following \cite{Bamberger1986,CakoniHaddarLechleiter2019}.
    Assume that $\mathrm{Im}(s)=\sigma> 0$ and $\hat{A}(s)$ is a function with values in the space of bounded operators between the Hilbert spaces $X$ and $Y$. Set
    \begin{equation*}
        a(t):=\frac{1}{2\pi}\int_{-\infty+i\sigma}^{\infty+i\sigma}\hat{A}(s)e^{-ist}ds,
    \end{equation*}
    and let $Ag=\int_{-\infty}^{\infty}a(t)g(\cdot-t)dt$ be the associated convolution operator. Let $p, r\in\mathbb{R}$. If 
    \begin{equation*}
        \Vert\hat{A}(s)\Vert_{X\rightarrow Y}\leq C|s|^r, 
    \end{equation*}
    then $A$ is a bounded operator from $H^p_{\sigma}(\mathbb{R},X)$ to $H^{p-r}_{\sigma}(\mathbb{R},Y)$.
\section{The modified  near-field operator in the Laplace domain}\label{sec:4}
\numberwithin{equation}{section}
\setcounter{equation}{0}
To study a linear sampling method for {the mathematical measurement data \eqref{data time domain scattered} in the time domain}, we first establish a functional framework in the Laplace domain with measurement data
\begin{equation*} %
\left\{\hat{u}^{\rm scat}(s,p;q)  - \hat{u}^{\rm scat}(-\overline{s},p;q) : (p, q)\in   B\times B \right\}
\end{equation*}
for a given $s=k+i\sigma$ with $\sigma>0$. The motivation to work with such a measurement data set will be seen clearly in the later context. For a clear presentation, in the following we establish the functional frameworks for $\hat{u}^{\rm scat}(s,p;q)$, $\hat{u}^{\rm scat}(-\overline{s},p;q)$, and $\hat{u}^{\rm scat}(s,p;q)  - \hat{u}^{\rm scat}(-\overline{s},p;q)$, respectively.
\subsection{The volume near-field operator $N$ in the Laplace domain}

Let $B \subset \mathbb{R}^3\backslash \overline{D}$ be a bounded domain and  $s=k+i\sigma$ be some complex number with $\sigma>0$, define the Hilbert spaces
\ben
W(B) =\left\{g \in L^2(B): \Delta g+s^2 g=0\quad\text{in}\quad B\right\}
\enn
and
\ben
W^*(B) =\left\{g \in L^2(B): \Delta g+\overline{s}^2 g=0\quad\text{in}\quad B\right\},
\enn
both of which are equipped with the $L^2(B)$-scalar product. We introduce the following volume potentials in the Laplace domain, motivated by the study in the frequency domain \cite{Garnier2023}. Given a density distribution $g \in W^*(B)$, consider an incident wave $v^{\rm inc}$ that can be written as a volume potential 
{ 
\be\label{vp}
v^{\rm inc}(x)=\int_B \hat\Phi_s(x;y)g(y)dy,\quad x\in \mathbb{R}^3.
\en}
It satisfies the inhomogeneous Helmholtz equation 
\ben
\Delta v^{\rm inc} +s^2 v^{\rm inc}=-g\chi_B\quad\text{in}\quad\mathbb{R}^3, 
\enn
where $\chi_B$ denotes the characteristic function of the domain $B$. From the linearity of the scattering problem with respect to the incident field, the solution $v^{\rm scat} \in H^1_s(\mathbb{R}^3\backslash D)$ to
\begin{equation*}
\left\{
\begin{aligned}
&\Delta v^{\rm scat}+s^2v^{\rm scat}=0 \quad\text{in}\quad\mathbb{R}^3\backslash \overline{D}\\
&v^{\rm inc}+v^{\rm scat}=0\quad\text{on}\quad\partial D\\
&\lim_{r\rightarrow \infty} r\left(\frac{\partial v^{\rm scat}}{\partial r}-is v^{\rm scat}\right)=0 
\end{aligned}
\right.
\end{equation*}
can be written as 
\begin{equation*}
v^{\rm scat}(x)=\int_B \hat{u}^{\rm scat}_*(s,x; y)g(y)dy, \ \ x\in \mathbb{R}^3\backslash \overline{D}.
\end{equation*}
Here $\hat{u}^{\rm scat}_*(s,x; y)\in H^1(\mathbb{R}^3\backslash \overline{D})$ is the scattered wave to \eqref{helmeq2} with boundary data $ -\hat{\Phi}_s(\cdot;y)$. Denote the corresponding total wave field by $\hat{u}_*(s,x; y) = \hat{u}^{\rm scat}_*(s,x; y) + \hat{\Phi}_s(x;y)$.

Now we can introduce the volume near-field operator $N$ in the Laplace domain, the idea is similar to the one in the Laplace domain that defined on a surface $\partial B$ (see \cite{ChenHaddar2010}) and the volume near-field operator in the Fourier domain (see \cite{Garnier2023}). The volume near-field operator  $N: W^*(B)\rightarrow W(B)$ is defined by
\begin{align}
 \label{N1}   (Ng)(x)=&\int_B \widehat{u}^{\rm scat}_*(s,x; y)g(y)dy,\quad x\in B.
\end{align}
The volume operator $V: W^*(B)\rightarrow H^{1/2}(\partial D)$ is defined by
\begin{align*} %
   (Vg)(x)=\int_B&\hat\Phi_s(x;y)g(y)dy,\quad x\in \partial D.
\end{align*}
The operator $A: H^{1/2}(\partial D)\rightarrow W(B)$ is define by
\be\label{opA}
 A \hat{f}=\hat{u}^{\rm scat}|_B,
\en
where $\hat{u}^{\rm scat}\in H^1_s(\R^3\backslash\ov{D})$ is the unique solution to \eqref{helmeq2} with boundary data $\hat{f}$. 

It follows by superposition principle that, for any $g\in W^*(B)$, the volume near-field operator $N$ in \eqref{N1} can be factorized as 
\begin{align}\label{N}
    N=-AV.
\end{align}
For later purposes, we can derive that, for any $\hat{f}\in H^{1/2}(\partial D)$,  
\be\label{eqA}
(A\hat{f})(x)=\int_{\partial D} \frac{\partial \hat{u}_*(s, y; x)}{\partial \nu(y)}\hat{f}(y)ds(y),\quad x\in B,
\en
since
\begin{align*}
    \quad&\hat{u}^{\rm scat}(x)\\
    =&\int_{\partial D}\left\{\hat{u}^{\rm scat}(y)\frac{\partial \hat\Phi_s(x;y)}{\partial \nu(y)}-\frac{\partial \hat{u}^{\rm scat}}{\partial \nu}(y)\hat\Phi_s(x;y)\right\}ds(y)\\
    =&\int_{\partial D}\left\{\hat{u}^{\rm scat}(y)\left(\frac{\partial \hat u_*(s,y; x)}{\partial \nu(y)}-\frac{\partial \hat u^{\rm scat}_*(s,y; x)}{\partial \nu(y)}\right)+\frac{\partial \hat{u}^{\rm scat}}{\partial \nu}(y)\hat u^{\rm scat}_*(s,y; x)\right\}ds(y)\\
    =&\int_{\partial D}\left\{\frac{\partial \hat{u}^{\rm scat}}{\partial \nu}(y)\hat u^{\rm scat}_*(s,y; x)-\hat{u}^{\rm scat}(y)\frac{\partial \hat{u}_*^{\rm scat}(s,y; x)}{\partial \nu(y)}\right\}ds(y)\\
    &+\int_{\partial D}\hat{u}^{\rm scat}(y)\frac{\partial \hat{u}_*(s,y; x)}{\partial \nu(y)}ds(y)\\   
    =&\int_{\partial D} \frac{\partial \hat{u}_*(s, y; x)}{\partial \nu(y)}\hat{f}(y)ds(y), \quad x\in \mathbb{R}^3\backslash \overline{D}.
\end{align*}
Now we are ready to prove the following lemmas.
\begin{lemma}\label{thm7}
The volume operator $V: W^*(B)\rightarrow H^{1/2}(\partial D)$ is injective and has dense range.
\end{lemma}
{ 
\begin{proof}
To prove the injectivity, suppose that $Vg=0$ for $g\in W^*(B)$, consider the volume potential \eqref{vp}. 
It is a solution to the Helmholtz equation $\Delta v^{\rm inc}+s^2v^{\rm inc}=0$ in $D$ with boundary data $v^{\rm inc}|_{\partial D}=Vg$. Then $Vg=0$ implies $v^{\rm inc}=0 \ \mathrm{in} \ D.$ Furthermore, since $v^{\rm inc}$ is a solution to the Helmholtz equation in ${\mathbb{R}^3}\backslash \overline{B}$, we have $v^{\rm inc}=0$ in ${\mathbb{R}^3}\backslash \overline{B}$ (by the unique continuation principle). Note that $v^{\rm inc}\in  H^2_{loc}(\mathbb{R}^3)$, then
$v^{\rm inc}\in H^2_0(B)$. Taking the $L^2(B)$-scalar product of $\Delta v^{\rm inc}+s^2v^{\rm inc}=-g$ 
with $\overline{g}$, we obtain that
\begin{equation*}
-\Vert g\Vert^2_{L^2(B)} = \int_B\left(\Delta v^{\rm inc}(x)+s^2v^{\rm inc}(x)\right)\overline{g}(x)dx= \int_B v^{\rm inc}(x)  \left(\Delta  \overline{g}+s^2  \overline{g} \right) dx,
\end{equation*}
where one applies integration by parts two times and $v^{\rm inc}\in H^2_0(B)$. Since $\Delta \overline{g}+s^2\overline{g}=0$ in $B$, the right hand side vanishes, i.e., $g=0$ in $B$.
To show the denseness of the range of $V$, it is sufficient to show  that $V^*$ is injective where $V^*:H^{-1/2}(\partial D)\rightarrow W^*(B)$ is given by
\begin{align*}
   (V^*\phi)(x)=\int_{\partial D}\hat\Phi_{-\overline{s}}(x, y)\phi(y)ds(y),\quad x\in B.
\end{align*}
Suppose that $V^*$ is not injective, then there exits $\phi \in H^{-1/2}(\partial D)$ such that $V^*\phi$ vanishes in $B$.
Since ${v}^{\rm inc}_{\phi}(x):=\int_{\partial D}\hat\Phi_{-\overline{s}}(x, y)\phi(y)ds(y)$
satisfies $\Delta v^{\rm inc}_\phi+\overline{s}^2 v^{\rm inc}_\phi=0$ in $\mathbb{R}^3\backslash \overline{D}$ and $D$. By unique continuation principle and the jump relations, it follows that $\phi=0$ on $\partial D$. This completes the proof.
\end{proof}
\begin{lemma}\label{thm8}
The operator $A: H^{1/2}(\partial D)\rightarrow W(B)$ is injective and has dense range.
\end{lemma}
{ 
\begin{proof}
To show that $A$ is injective, let $\hat{f}\in H^{1/2}(\partial D)$ and suppose that $A\hat{f}=0$. Let $w$ be the unique radiating solution to \eqref{helmeq2} with boundary data $w|_{\partial D}=\hat{f}$.  $A\hat{f}=0$ together with unique continuation principle implies that  $w=0$ in $H^1(\mathbb{R}^3\backslash \overline{D})$. This shows that $\hat{f}=w|_{\partial D}=0$.

To show the dense range of $A$, let $g\in W(B)$ and suppose that 
\begin{equation*}
    \int_B(A\hat{f})(x) \overline{g(x)} dx=0,\ \forall \hat{f}\in H^{1/2}(\partial D).
\end{equation*}
Together with \eqref{eqA}, it follows that
\begin{equation*}
    \int_{\partial D}\int_B\frac{\partial \hat{u}_*(s, y; x)}{\partial \nu(y)}\overline{g(x)}dx\hat{f}(y)ds(y)=0,\ \forall \hat{f}\in H^{1/2}(\partial D).
\end{equation*}
Consider the volume potential
$v^{\rm inc}(y)=\int_B \hat{\Phi}_s(y; x)\overline{g(x)}dx$,
then the corresponding scattered field is $v^{\rm scat}(y)=\int_B \hat{u}^{\rm scat}_*(s,y; x)\overline{g(x)}dx$ and the corresponding total field is $v(y)=\int_B \hat{u}_*(s,y; x)\overline{g(x)}dx$.
The above equation implies that $\frac{\partial v}{\partial \nu}=0$, together with $v=0$ on $\partial D$,  one has from Holmgren's theorem that $v=0$ in $\mathbb{R^3}\backslash \{\overline{D}\cup\overline{B}\}$. This yields that $v=\frac{\partial v}{\partial \nu}=0$ on $\partial B$. Moreover, because $v^{\rm scat}$ satisfies the Helmholtz equation in $B$, it follows that 
$$\Delta v+s^2 v=\Delta v^{\rm inc}+s^2v^{\rm inc}=-\overline{g} \ \mathrm{in} \ B.$$
Multiplying both sides by $g$ yields that
\begin{equation*}
-\Vert g\Vert^2_{L^2(B)} = \int_B\left(\Delta v(x)+s^2v(x)\right)g(x)dx= \int_B v(x)  \left(\Delta  g+s^2  g \right) dx,
\end{equation*}
where one applies integration by parts two times and $v\in H^2_0(B)$. Since $\Delta g+s^2g=0$ in $B$, the right hand side vanishes, i.e., $g=0$ in $B$.
It follows that the operator $A$ has dense range.
\end{proof}
}
Now we have the following theorem.
\begin{theorem}\label{thm6}
    The volume near-field operator $N: W^*(B)\rightarrow W(B)$
is injective and has dense range.
\end{theorem}
\begin{proof}
    This follows from the factorization \eqref{N}, Lemma \ref{thm7}, and Lemma \ref{thm8}.
\end{proof}

Another key ingredient in the analysis of the LSM is the characterization of the range of $A$. 
\begin{lemma}\label{thm10}
    $\hat{\Phi}_s|_B(\cdot;z)\in \mathrm{range}(A)$ if and only if $z\in D$.    
\end{lemma}
\begin{proof}
    If $z\in D$, then $\hat{\Phi}_s|_{\partial D}(\cdot; z)\in H^{1/2}(\partial D)$ and $\hat{\Phi}_s|_B(\cdot;z)=A\hat{\Phi}_s|_{\partial D}(\cdot; z).$ If $z\notin D,$ assume that there exists $\hat{f}\in H^{1/2}(\partial D)$ such that $A\hat{f}=\hat{\Phi}_s|_B(\cdot;z).$ Therefore, by the unique continuation principle in $\mathbb{R}^3\backslash\left(\overline{D}\cup\{z\}\right)$, the solution $u\in H^1(\mathbb{R}^3\backslash D)$ to the exterior Dirichlet problem with $u|_{\partial D}=\hat{f}$ must coincide with $\hat{\Phi}_s(\cdot;z)$ in $\mathbb{R}^3\backslash\left(\overline{D}\cup\{z\}\right)$. If $z\in \mathbb{R}^3\backslash \overline{D},$ this contradicts the regularity of $u$. If $z\in\partial D,$ from the boundary condition one has that $\hat{\Phi}_s|_{\partial D}(\cdot;z)=\hat{f}\in H^{1/2}(\partial D),$ which is a contradiction to $\hat{\Phi}_s(\cdot; z)\notin H^1(D)$ when $z\in \partial D$. 
\end{proof}
\subsection{The modified volume near-field operator $N_\sigma$ in the Laplace domain}

In this subsection, we introduce the modified near-field operator $N_\sigma$ and study its properties. The motivation to introduce such a modified operator is in similar spirit to \cite{CakoniHaddarLechleiter2019}, which allows to develop the mathematical theory of certain Laplace domain and time domain operators; moreover setting $\sigma=0$ directly allows to implement the numerical algorithm. To begin with, for $s= k + i\sigma$ with sufficiently small $\sigma>0$, we introduce the following modified volume operator $V_{\sigma}: W(B)\rightarrow H^{1/2}(\partial D)$ by  
\begin{align}
\label{vsigma}   (V_{\sigma}\,g)(x)=\int_B&\hat\Phi_{s-2i\sigma}(x;y)g(y)dy,\quad x\in \partial D.
\end{align}
Note that $s - 2i \sigma = \overline{s}$. We have the following properties for $V_\sigma$.
\begin{lemma}\label{thm11}
    The modified volume operator $V_{\sigma}: W(B)\rightarrow H^{1/2}(\partial D)$ is injective and has dense range.
\end{lemma}
\begin{proof}
We start with the injectivity. Let $g\in W(B)$ and consider the ``incident wave'' 
\begin{equation*}
    v^{\rm inc}(x)=\int_B\hat\Phi_{\overline{s}}(x;y)g(y)dy, \quad x\in \mathbb{R}^3.
\end{equation*}
It is a solution to the Helmholtz equation $\Delta v+\overline{s}^2v=0$ in $D$ with boundary data $v^{\rm inc}|_{\partial D}=V_{\sigma}g$. Suppose that $V_{\sigma}g=0$ on $\partial D$. This implies that $v^{\rm inc}=0$ in $D$ since $\ov{s}^2$ is not a Dirichlet eigenvalue of $-\Delta$ in $D$. Furthermore, since $v^{\rm inc}$ is a solution to the Helmholtz equation in ${\mathbb{R}^3}\backslash \overline{B}$, we have $v^{\rm inc}=0$ in ${\mathbb{R}^3}\backslash \overline{B}$ (by the unique continuation principle).  Note that $v^{\rm inc}\in  H^2_{loc}(\mathbb{R}^3)$, then
$v^{\rm inc}\in H^2_0(B)$. Taking the $L^2(B)$-scalar product of $\Delta v^{\rm inc}+\ov{s}^2v^{\rm inc}=-g$ 
with $g\in W(B)$, we obtain that
\begin{equation*}
-\Vert g\Vert^2_{L^2(B)} = \int_B\left(\Delta v^{\rm inc}(x)+\ov{s}^2v^{\rm inc}(x)\right)\overline{g}(x)dx= \int_B v^{\rm inc}(x)  \left(\Delta  \overline{g}+\ov{s}^2  \overline{g} \right) dx,
\end{equation*}
where one applies integration by parts two times and $v^{\rm inc}\in H^2_0(B)$. Since $\Delta \overline{g}+\overline{s}^2\overline{g}=0$ in $B$, the right hand side vanishes, i.e., $g=0$ in $B$.

We now prove the dense range by showing that the adjoint operator
\begin{align*}
V_{\sigma}^*&: H^{-1/2}(\partial D) \rightarrow W(B)\\ 
(V_{\sigma}^*\phi)(x)=&\int_{\partial D}\hat{\Phi}_{-s}(x; y)\phi(y)ds(y),\quad x\in B,
\end{align*}
is injective. Let $\phi\in H^{-1/2}(\partial D)$ and consider $v_1^{\rm inc}(x)=\int_{\partial D}\hat\Phi_{-s}(x; y)\phi(y)ds(y)$,
which satisfies the Helmholtz equation $\Delta v_1^{\rm inc}+s^2v_1^{\rm inc}=0$ in both $\mathbb{R}^3\backslash \overline{D}$ and $D$; moreover  $v^{\rm inc}_1|_B=V_{\sigma}^*\phi$. Suppose that $V_{\sigma}^*\phi=0$, i.e., $v^{\rm inc}_1|_B=0$. By the unique continuation principle, one has $v_1^{\rm inc}=0$ in $\mathbb{R}^3\backslash \overline{D}$. Jump relation across $\partial D$ yields that $v^{\rm inc}_1=0$ on $\partial D$, then $v^i_1=0$ in $D$ since $s^2$ is not a Dirichlet eigenvalue of $-\Delta$ in D. Now we have that $v^{\rm inc}_1=0$ in $\mathbb{R}^3\backslash \overline{D}$ and in $D$, therefore $\phi=0$   follows from the jump relation across $\partial D$.
\end{proof}

Now we can introduce the modified near-field operator $N_{\sigma} : W(B)\rightarrow W(B)$ by
\ben 
N_{\sigma}:=-AV_{\sigma}.
\enn 
The following result can be obtained directly from the properties on $A$ and $V_\sigma$.
\begin{lemma}\label{thm13}
    The modified volume near-field operator $N_{\sigma}:W(B)\rightarrow W(B)$ is injective and has dense range.
\end{lemma}   
\begin{proof}
From Lemma \ref{thm8} and Lemma \ref{thm11}, both $A$ and $V_\sigma$ are injective and have dense range, this concludes the proof.
\end{proof}
\subsection{ The modified imaginary near-field operator $I$ in the Laplace domain}
In this subsection we introduce the modified imaginary near-field operator $I$ in the Laplace domain. To begin with, we first introduce the conjugate operator $\overline{N}: W(B)\rightarrow W^*(B)$ where
\begin{align*}
    (\overline{N}g)(x)=\int_B&\hat{u}^{\rm scat}(-\overline{s}, x; y)g(y)dy,\quad x\in B.
\end{align*}
Similarly, to factorize the operator $\overline{N}$, we introduce  $\overline{V}: W(B)\rightarrow H^{1/2}(\partial D)$ and $\overline{A}: H^{1/2}(\partial{D})\rightarrow W^*(B)$ where
\begin{align}
 \label{VV}  (\overline{V}g)(x)=\int_B&\hat\Phi_{-\overline{s}}(x; y)g(y)dy \quad\mbox{for}\quad x\in \partial D \qquad\mbox{and}\qquad \overline{A}g=w_g|_{B},
\end{align}
here $w_g \in H^1(\mathbb{R}^3\backslash \overline{D})$ is the unique solution to
$$
\left\{
\begin{array}{ll}
\Delta  w_{g}  + (-\overline{s})^2 w_{g}  =0,  &x\in \mathbb{R}^3\backslash\ov{D},\\
w_{g} =g,  &x\in\partial D.
\end{array}
\right.
$$
Similar to the deduction of \eqref{eqA}, we have  for any $g\in H^{1/2}(\partial D)$ that
\be\label{eqA2}
(\ov{A}g)(x)=\int_{\partial D} \frac{\partial \hat{u}_*(-\ov{s}, y; x)}{\partial \nu(y)}g(y)ds(y),\quad x\in B.
\en
 Similar to Theorem \ref{thm6},  $\overline{N}$ can be factored as $\overline{N}=-\overline{A}\overline{V}$ and it is injective and has dense range.

Now we are ready to study the imaginary near-field operator.
\begin{theorem}\label{thm14}
    The modified imaginary near-field operator 
    \begin{align*}
        I: W(B)&\rightarrow W(B)\bigoplus W^*(B)\\
        Ig=(&N_{\sigma}-\overline{N})g,\quad x\in B,
    \end{align*}
is injective and has dense range.
\end{theorem}
\begin{proof}
If $\hat{u}\in W(B)\cap W^*(B)$, then $u$ satisfies both 
\ben
\Delta \hat{u}+s^2\hat{u}=0\quad\text{in} \quad B \quad
\text{and}\quad
\Delta \hat{u}+\overline{s}^2\hat{u}=0\quad\text{in}\quad B,
\enn
respectively. Thus $\hat{u}=0$ since $s^2\neq\overline{s}^2$, then $W(B)\bigoplus W^*(B)$ is a direct sum.

To prove $I$ is injective, let $Ig=0$ for some $g\in W(B)$. Since $W(B)\bigoplus W^*(B)$ is a direct sum and $(N_{\sigma}-\overline{N})g = 0$, we have $N_{\sigma}g=0$ and $\overline{N}g=0$.  Theorem \ref{thm6} and Theorem \ref{thm13} yields that $g=0$ in $B$.

{ To show $I$ has dense range, { let $f\in W(B)$ and $h\in W^*(B)$ }and suppose that { $\langle Ig,f+h\rangle=0$,} $\forall g\in W(B)$, i.e., { $\langle(-AV_\sigma+\overline{A}\overline{V})g, f+h\rangle=0$.} Using \eqref{eqA} -- \eqref{vsigma} and \eqref{VV} -- \eqref{eqA2}, one can write
\begin{align*}
    &\int_B\int_{\partial D}\frac{\partial \hat{u}_*(s, x; y)}{\partial \nu(y)}\int_B \hat\Phi_{\overline{s}}(y;z)g(z)dzds(y)\overline{f(x)}dx \\
    =&\int_B\int_{\partial D}\frac{\partial \hat{u}_*(-\overline{s}, x; y)}{\partial \nu(y)}\int_B \hat\Phi_{-\overline{s}}(z; y)g(z)dzds(y)\overline{h(x)}dx.
\end{align*}
Since the above equation holds for all $g\in W(B)$, then the following equation holds in $L^2(B)$ where
\begin{align*}
&\int_{\partial D}\hat\Phi_{\overline{s}}(z;y)\int_B\frac{\partial \hat{u}_*(s, x; y)}{\partial \nu(y)}\overline{f(x)}dx ds(y)\\
=&\int_{\partial D}\hat\Phi_{-\overline{s}}(z; y)\int_B\frac{\partial \hat{u}_*(-\overline{s}, x; y)}{\partial \nu(y)}\overline{h(x)}dx ds(y),\quad z \in B.
\end{align*}
Using the unique continuation principle, this implies the above equation  still holds in $z\in \mathbb{R}^3 \backslash \overline{D}$; moreover the above equation also holds in $D$, this is because: continuity of the single layer potential yields that the above equation holds  on $\partial D$, note in addition that both the left hand side and right hand side satisfy the Helmholtz equation with the same (complex) wavenumber, one can get that the above equation holds in $D$ as well. Now jump relation of the conormal derivative leads to{ 
$$\int_B\frac{\partial \hat{u}_*(s, x; y)}{\partial \nu(y)}\overline{f(x)}dx=\int_B\frac{\partial \hat{u}_*(-\overline{s}, x; y)}{\partial \nu(y)}\overline{h(x)}dx,\quad y\in \partial D.
$$}%
The function on the left represents a physical wave that is radiating, but the function on the right hand represents a nonphysical wave, therefore it follows that
{ 
$$
\int_B\frac{\partial \hat{u}_*(s, x; y)}{\partial \nu(y)}\overline{f(x)}dx=0\quad \mbox{and}\quad \int_B\frac{\partial \hat{u}_*(-\overline{s}, x; y)}{\partial \nu(y)}\overline{h(x)}dx=0\quad \mbox{for}\quad y\in\partial D.
$$}%
It remains to show that both $f$ and $h$ vanish. In the following we show that $f$ vanishes and the case for $h$ follows exactly the same. Specifically, consider the the volume potential
{ 
$$v^{\rm inc}(x)=\int_B \hat{\Phi}_s(x; y)\overline{f(x)}dx,\quad y\in \mathbb{R}^3,$$}
then the corresponding scattered wave and total wave are
{ 
$$v^{\rm scat}(y)=\int_B \hat{u}^{\rm scat}_*(s, x; y)\overline{f(x)}dx\quad \mbox{and} \quad v(y)=\int_B \hat{u}_*(s, x; y)\overline{f(x)} dx~~~\mbox{for}~~~y\in \mathbb{R}^3\backslash\overline{D},$$}%
respectively.
This implies that $v={\partial v}/{\partial \nu}=0$ on $\partial D,$ thus one has $v=0$ in $\mathbb{R}^3\backslash \{\overline{D}\cup\overline{B}\}$ via Holmgren's theorem. Since $v$ is as smooth as $v^i\in H^2_{loc}(\mathbb{R}^3)$, one has $v={\partial v}/{\partial \nu}=0$ on $\partial B.$ Moreover, because $v^{\rm scat}$ satisfies the Helmholtz equation in $B$, we have 
{ 
$$\Delta v+s^2 v=\Delta v^{\rm inc}+s^2v^{\rm inc}=-\overline{f} \quad \mathrm{in} \quad B.$$}%
Taking the $L^2(B)$-scalar product of { $\Delta v+ s^2 v=-\overline{f}$ }
with { $f\in W(B)$}, we obtain that
{ 
\begin{equation*}
-\Vert f\Vert^2_{L^2(B)} = \int_B\left(\Delta v(x)+{s}^2 v(x)\right)f(x)dx= \int_B v(x) \left(\Delta  f+s^2f \right) dx,
\end{equation*}}%
where one applies integration by parts two times and $v\in H^2_0(B)$. Since { $\Delta f+s^2f=0$} in $B$, the right hand side vanishes, i.e., $f=0$ in $B$. {Similarly, we obtain that $h=0\ \mathrm{in} \ B$. This completes the proof.}}
\end{proof}

 \section{Time domain linear sampling method for passive imaging} \label{sec:5}
 In this section, we develop the time domain linear sampling method for the mathematical measurement data \eqref{data time domain scattered}; such data \eqref{data time domain scattered} will be approximated by the data \eqref{data time domain total} due to random sources, which will be discussed in Section \ref{sec:6}. We apply the Laplace transform to relate the mapping properties of Laplace domain factorized operators in Section \ref{sec:4} to their counterparts in the time domain, followed by the main result of the linear sampling method.
\subsection{Mapping properties in the time domain} 
To begin with, we introduce
\ben
    W^p_{\sigma}\left(\mathbb{R}, L^2(B)\right)=\left\{g\in  H^p_{\sigma}\left(\mathbb{R}, L^2(B)\right):\Delta g-g''=0\quad\text{in}\quad B\right\} \quad 
\enn
and
\ben
    (W^*)^p_{\sigma}\left(\mathbb{R}, L^2(B)\right)=\left\{g\in H^p_{\sigma}\left(\mathbb{R}, L^2(B)\right):\Delta g-4\sigma^2g+4\sigma g'-g''=0\quad\text{in}\quad B\right\}.
\enn
It can be seen that for $\sigma>0$, the Laplace transforms of $g_1\in W^p_{\sigma}\left(\mathbb{R}, L^2(B)\right)$ and $g_2\in  (W^*)^p_{\sigma}\left(\mathbb{R}, L^2(B)\right)$ belong to $W(B)$ and $W^*(B)$, respectively. We now prove the following lemma for the operators $A, \overline{A}, V_{\sigma}$ and $\overline{V}$, which will be used to study the corresponding time domain operators $\mathcal{A}, \overline{\mathcal{A}}, \mathcal{V}_{\sigma}$ and $\overline{\mathcal{A}}$. 
 In the following, the constants $C$ may be different but they only depend on $D$, $B$ and $\sigma_0$.
\begin{lemma}\label{thm17}
    For $s=k+i\sigma$ with some fixed imaginary part $\sigma>\sigma_0>0$, we have
    \ben
       \Vert A\Vert_{H^{1/2}(\partial D)\rightarrow W(B)}\leq C|s|^{1/2}\quad&\text{and}&\quad
       \Vert \overline{A}\Vert_{H^{1/2}(\partial D)\rightarrow W^*(B)}\leq C|s|^{1/2}, \\
       \Vert V_{\sigma}\Vert_{W(B)\rightarrow H^{1/2}(\partial D)}\leq C|s|\quad&\text{and}&\quad
       \Vert \overline{V}\Vert_{W(B)\rightarrow H^{1/2}(\partial D)}\leq C|s|,\\
       \Vert N_{\sigma}\Vert_{W(B)\rightarrow W(B)}\leq C|s|^{3/2}\quad&\text{and}&\quad
       \Vert \overline{N}\Vert_{W(B)\rightarrow W^*(B)}\leq C|s|^{3/2},
    \enn
   where the constant $C$ depends on $D$, $B$, and $\sigma_0$.
\end{lemma}
\begin{proof}
For $\hat{g}\in H^{1/2}(\partial D)$,   estimate \eqref{eq22} and the definition of $A$   \eqref{opA} yields that
    \begin{equation*}
       \Vert A\hat{g}\Vert^2_{H^1_s(B)}\leq C\frac{1}{{\sigma}^2}\mathrm{max}\left(\frac{1}{\sigma_0}, 1\right)|s|^3\Vert \hat{g}\Vert^2_{H^{{1}/{2}}(\partial D)}.
    \end{equation*}
Note that $\Vert A\hat{g}\Vert^2_{H^1_s(B)} \ge |s|^2\Vert A\hat{g}\Vert^2_{L^2(B)}$, then it can be seen that
    \begin{equation*}
       \Vert A\Vert_{H^{1/2}(\partial D)\rightarrow W(B)}\leq C|s|^{1/2},
    \end{equation*}
    where the constant $C$ depends on $D$ and $\sigma_0$.
    The estimate for $\overline{A}$ follows similarly.
    
For operator $\overline{V}$, by the definition of $\overline{V}$ in \eqref{VV}, we have
{ 
    \begin{equation*}
       \Vert \overline{V}\hat f\Vert_{H^{1/2}(\partial D)}=\left\| \int_B\hat\Phi_{-\overline{s}}(\cdot;y)\hat f(y)dy\right\|_{H^{1/2}(\partial D)}.
    \end{equation*}}
For $\hat{f} \in W(B)$, we consider 
    \begin{equation*}
       v^{\rm inc}(x):=\int_B\hat\Phi_{-\overline{s}}(x;y)\hat{f}(y)dy.
    \end{equation*}
    From trace theorem, it is sufficient to estimate
    $\Vert v^{\rm inc}\Vert_{H^1(D)}$. Using  Minkowski's integral inequality and that $d=\text{dist}\{\overline{B}, \overline{D}\}>0$, 
    \begin{align}\label{eqvi}    
    \Vert v^{\rm inc} \Vert^2_{L^2(D)}
    &=\int_D\left|\int_B\hat\Phi_{-\overline{s}}(x; y)\hat{f}(y)dy\right|^2dx\\
    &\leq\left(\int_B|\hat{f}(y)|^2dy\right) \cdot \left(\int_D \frac{1}{(4\pi)^2d^2}dx\right) 
        \leq C\Vert \hat{f}\Vert^2_{L^2(B)}  \nonumber 
    \end{align}
    for some constant $C$ depending on $B$ and $D$. To estimate $\nabla v^{\rm inc}$, from
    \begin{align*}
    \left(v^{\rm inc}(x)\right)'_{x_i} =\int_B\left(\frac{-i\overline{s}e^{-i\overline{s}|x-y|}\cdot (x_i-y_i)}{4\pi|x-y|^2}-\frac{e^{-i\overline{s}|x-y|}\cdot (x_i-y_i)}{4\pi|x-y|^3}\right)\hat f(y)dy,
    \end{align*}
and the Minkowski's integral inequality, it follows that
    \begin{align} \label{eqvii}
       &\left\| \left(v^{\rm inc}(x)\right)'_{x_i}\right\|_{L^2(D)} 
      \leq\int_B\left(\int_D \left(\frac{|s|}{4\pi|x-y|}+\frac{1}{4\pi|x-y|^2}\right)^2|\hat f(y)|^2dx\right)^{1/2}dy\\
      \nonumber &\quad\quad\quad\leq\left(\int_B|\hat f(y)|^2dy\right)^{1/2}\cdot \left(\int_D\left(\frac{|s|^2}{16\pi^2d^2}+\frac{|s|}{8\pi^2d^3}+\frac{1}{16\pi^2d^4}\right)dx\right)^{1/2} \\
      \nonumber&\quad\quad\quad\leq C|s|\Vert \hat f\Vert_{L^2(B)}.  
    \end{align}
    From \eqref{eqvi} and \eqref{eqvii} we prove that $$\Vert \overline{V}\Vert_{W(B)\rightarrow H^{1/2}(\partial D)}\leq C|s|.$$
The estimate for $V_{\sigma}$ follows similarly. 

Finally it follows from $N_\sigma = -AV_\sigma$ and $\ov{N}=-\ov{A}\ov{V}$ that
    \begin{align*}
       \Vert \overline{N}\Vert_{W(B)\rightarrow W^*(B)}\leq C|s|^{3/2}\quad\text{and}\quad \Vert N_{\sigma}\Vert_{W(B)\rightarrow W(B)}\leq C|s|^{3/2},
    \end{align*}
    which concludes the proof.
\end{proof}

Using the above lemma, we can prove the properties of the following time domain operators.  Recall that the time domain operators are related by their Laplace counterparts, i.e, for any Laplace domain operator $\hat{A}(s)$ satisfying
   \begin{equation*}
        \Vert\hat{A}(s)\Vert_{X\rightarrow Y}\leq C|s|^r, 
    \end{equation*}
the time domain operator $A$ given by
$$
Ag=\int_{-\infty}^{\infty}a(t)g(\cdot-t)dt \qquad \mbox{ where } \qquad a(t)=\frac{1}{2\pi}\int_{-\infty+i\sigma}^{\infty+i\sigma}\hat{A}(s)e^{-ist}dt,
$$
defines a bounded operator from $H^p_{\sigma}(\mathbb{R},X)$ to $H^{p-r}_{\sigma}(\mathbb{R},Y)$.
\begin{theorem}\label{thm18}
For some fixed $\sigma>0$, the operators in the time domain corresponding to $V_{\sigma}, \overline{V}, A, \overline{A}$ and $N_\sigma, \ov{N}$ are
\begin{align*}
        &\mathcal{V_{\sigma}}: W^p_{\sigma}(\mathbb{R}, L^2(B))\rightarrow H^{p-1}_{\sigma}(\mathbb{R}, H^{1/2}(\partial D)), \\
        &\mathcal{\overline{V}}: W^p_{\sigma}(\mathbb{R}, L^2(B))\rightarrow H^{p-1}_{\sigma}(\mathbb{R}, H^{1/2}(\partial D)),\\
        &\mathcal{A}: H^{p}_{\sigma}(\mathbb{R}, H^{1/2}(\partial D))\rightarrow W^{p-1/2}_{\sigma}(\mathbb{R}, L^2(B)), \\
        &\mathcal{\overline{A}}: H^{p}_{\sigma}(\mathbb{R}, H^{1/2}(\partial D))\rightarrow \left(W^*\right)^{p-1/2}_{\sigma}(\mathbb{R}, L^2(B)),    \\
        &\mathcal{N_{\sigma}}: W^p_{\sigma}(\mathbb{R}, L^2(B))\rightarrow W^{p-3/2}_{\sigma}(\mathbb{R}, L^2(B)),\\  
        &\overline{\mathcal{N}}: W^p_{\sigma}(\mathbb{R}, L^2(B))\rightarrow \left(W^*\right)^{p-3/2}_{\sigma}(\mathbb{R}, L^2(B)),
    \end{align*}
    respectively. They are all bounded, injective and have dense range.
\end{theorem}
\begin{proof}

We firstly prove that $\mathcal{V_{\sigma}}$ is injective. Let $\mathcal{V_{\sigma}}g=0\in H^{p-1}_{\sigma}\left(\mathbb{R}, H^{1/2}(\partial D)\right)$, then $\mathcal{L}[\mathcal{V_{\sigma}}g]=0$ for all $s=k+i\sigma, k\in\mathbb{R}$. By the injectivity of $V_\sigma = \mathcal{L}[\mathcal{V_{\sigma}} ]$ in Theorem \ref{thm11}, we have $\mathcal{L}[g]=0$ in $W(B)$ so $g$ vanishes in $W^p_{\sigma}\left(\mathbb{R}, L^2(B)\right)$.
      
     The dense range of $\mathcal{V_{\sigma}}$ can be given by showing the injectivity of its adjoint operator $\mathcal{V_{\sigma}}^*$. For $\hat{g}\in W(B)$, take the inverse Laplace transform of $V_{\sigma}\hat g$ defined in \eqref{vsigma}, we have the explicit form of $\mathcal{V_{\sigma}}$ as
     \begin{align*}
      (\mathcal{V_{\sigma}}g)(t, x)=\int_B\int_{-\infty}^{+\infty}e^{2\sigma(t-\xi)}&\frac{\delta(t-\xi-|x-y|)}{4\pi|x-y|}g(\xi, y)d\xi dy.
    \end{align*}
 Consider $g\in C^{\infty}_0\left(\mathbb{R}, L^2(B)\right)$ and $f\in C^{\infty}_0\left(\mathbb{R}, H^{-1/2}(\partial D)\right)$, we have
 {\small
     \begin{align*}
        \langle\mathcal{V_{\sigma}}g, f\rangle_\sigma
        &=\int_{\partial D}\int_{-\infty}^{+\infty}\left(\int_B\int_{-\infty}^{+\infty}e^{2\sigma(t-\xi)}\frac{\delta(t-\xi-|x-y|)}{4\pi|x-y|}g(\xi, y)d\xi dy\right)e^{-2\sigma t}f(t,x)dtds(x)\\
        &=\int_B\int_{-\infty}^{+\infty}e^{-2\sigma\xi}\left(\int_{\partial D}\int_{-\infty}^{+\infty}\frac{\delta(t-\xi-|x-y|)}{4\pi|x-y|}f(t,x)dtds(x)\right)g(\xi, y)d\xi dy.
    \end{align*}}
This leads to the adjoint operator $\mathcal{V_{\sigma}}^*: H^{-p+1}_{\sigma}(\mathbb{R}, H^{-1/2}(\partial D))\rightarrow W^{-p}_{\sigma}(\mathbb{R}, L^2(B))$ where
    \begin{align*}
      &(\mathcal{V_{\sigma}}^*f)(\xi, y)=\int_{\partial D}\int_{-\infty}^{+\infty}\frac{\delta(t-\xi-|x-y|)}{4\pi|x-y|}f(t, x)dt ds(x).
    \end{align*}
    Suppose that $\mathcal{V_{\sigma}}^*f=0$, we take its Laplace transform with $s=k+i\sigma$ to obtain that
    \begin{align*}
    \mathcal{L}[\mathcal{V^*_{\sigma}}f](y)=
        &\int_{-\infty}^{+\infty}e^{is\xi}\int_{\partial D}\int_{-\infty}^{+\infty}\frac{\delta(t-\xi-|x-y|)}{4\pi|x-y|}f(t, x)dt ds(x)d\xi\\
        =&\int_{\partial D}\int^{\infty}_{-\infty}\frac{e^{is(t-|x-y|)}}{4\pi |x-y|}f(t, x)dtds(x)\\
        =&\int_{\partial D}\frac{e^{-is|x-y|}}{4\pi |x-y|}\int^{\infty}_{-\infty}e^{ist}f(t,x)dtds(x)=0,\quad y\in B.
    \end{align*}
    By unique continuation and jump relations (following exactly  the proof of { Lemma \ref{thm7}}), $\mathcal{L}[\mathcal{V^*_{\sigma}}f]=0$ implies that 
\(
    \int_{-\infty}^{+\infty}e^{ist}f(t, x)dt=0\) for \(x\in\partial D\).
    Therefore $f(t, x)=0$ and $\mathcal{V_{\sigma}}$ has dense range.

 The injectivity and denseness of range for $\overline{\mathcal{V}}$ can be proved  similarly. The boundedness of operators $\mathcal{A}$ and $\overline{\mathcal{A}}$ follows directly from Lemma \ref{thm4}. The injectivities of $A$ and $\ov{A}$ imply that $\mathcal{A}$ and $\overline{\mathcal{A}}$ are injective. Now we show that the range of $\mathcal{A}$ is dense by showing the injectivity of its adjoint operator $\mathcal{A}^*$. For $\hat{g}\in H^{1/2}(\partial D)$, take the inverse Laplace transform of $A\hat g$ defined in \eqref{eqA}, we have the explicit form of $\mathcal{A}$ by
     \begin{align*}
            (\mathcal{A}g)&(t, x)=\int_{-\infty}^{+\infty}\int_{\partial D}\frac{\partial u_*(t-\xi, y; x)}{\partial \nu(y)}g(\xi, y)ds(y)d\xi,
        \end{align*}
        where $u_*$ is the time domain counterpart of $\hat{u}_*$.
 Again to derive $\mathcal{A}^*$, consider $g\in C^{\infty}_0\left(\mathbb{R}, H^{1/2}(\partial D)\right)$ and $f\in C^{\infty}_0\left(\mathbb{R}, L^2(B)\right)$ where
       \begin{align*}
        \langle\mathcal{A}g, f\rangle_\sigma
        &=\int_B\int_{-\infty}^{+\infty}e^{-2\sigma t}\int_{-\infty}^{+\infty}\int_{\partial D}\frac{\partial u_*(t-\xi,y;x)}{\partial \nu(y)}g(\xi,y)ds(y)d\xi f(t,x)dtdx\\
        =\int_{-\infty}^{+\infty}&\int_{\partial D}\left(\int_B\int_{-\infty}^{+\infty}\frac{\partial u_*(t-\xi,y;x)}{\partial \nu(y)}f(t,x)e^{2\sigma(\xi-t)}dtdx\right)e^{-2\sigma\xi}g(\xi,y)ds(y)d\xi.
        \end{align*}
 This leads to the adjoint operator $\mathcal{A}^*: W^{1/2-p}_{\sigma}\left(\mathbb{R}, L^2(B)\right)\rightarrow H^{-p}_{\sigma}\left(\mathbb{R}, H^{-1/2}(\partial D)\right)$ where 
        \begin{align*}
            &(\mathcal{A}^*f)(\xi, y)=\int_{-\infty}^{+\infty}\int_Be^{2\sigma(\xi-t)}\frac{\partial u_*(t-\xi, y; x)}{\partial \nu(y)}f(t, x)dxdt.
        \end{align*}
        Suppose that $\mathcal{A}^*f=0$, we take its Laplace transform with $s=k+i\sigma$ to obtain that
        \begin{align*}
            \mathcal{L}[\mathcal{A}^*f](y)
            =&\int^{+\infty}_{-\infty}e^{is\xi}\int_B\int^{+\infty}_{-\infty}e^{-2\sigma(t-\xi)}f(t,x)\frac{\partial u_*(t-\xi,y;x)}{\partial \nu(y)}dtdxd\xi\\
            =&\int^{+\infty}_{-\infty}e^{i\overline{s}\xi}\left(\int_B\int^{+\infty}_{-\infty}e^{-2\sigma t}f(t,x)\frac{\partial u_*(t-\xi,y;x)}{\partial \nu(y)}dtdx\right)d\xi\\
            =&\int_B\left(\int^{+\infty}_{-\infty}e^{ist}f(t,x)dt\right)\cdot\left(\int^{+\infty}_{-\infty}e^{i\overline{s}t}\frac{\partial u_*(-t,y;x)}{\partial \nu(y)}dt\right)dx\\
            =&\int_B\left(\int^{+\infty}_{-\infty}e^{ist}f(t,x)dt\right)\cdot\left(\int^{+\infty}_{-\infty}e^{-i\overline{s}t}\frac{\partial u_*(t,y;x)}{\partial \nu(y)}dt\right)dx \\
            =& \int_B\hat{f}(s,x)\frac{\partial \hat{u}_*(-\overline{s},y;x)}{\partial \nu(y)}dx=0,\quad y\in\partial D.
        \end{align*}
        Similar to the proof of Theorem \ref{thm14}, we have that $\hat{f}=0$ and thus $f=0$. The denseness of range for $\overline{\mathcal{A}}$ can be proved in a similar way. From Theorem \ref{thm4} and Theorem \ref{thm17}, we obtain that the above operators are bounded. Finally, $\mathcal{N_{\sigma}}$ and $\overline{\mathcal{N}}$ are injective, bounded and have dense range since $\mathcal{N_{\sigma}}=-\mathcal{AV_{\sigma}}$ and $\overline{\mathcal{N}}=\mathcal{\overline{A}\overline{V}}$.
\end{proof}

Now we prove the following property of the time domain modified imaginary near-field operator $\mathcal{I}:=\mathcal{N_{\sigma}}-\overline{\mathcal{N}}$. 
\begin{theorem}\label{thm19}
    The time domain modified imaginary near-field operator
    \begin{align}
       \nonumber \mathcal{I}: W^{p+3/2}_{\sigma}\left(\mathbb{R}, L^2(B)\right)&\rightarrow W^p_{\sigma}\left(\mathbb{R}, L^2(B)\right)\bigoplus \left(W^*\right)^p_{\sigma}\left(\mathbb{R}, L^2(B)\right)\\
       \label{operator I} \mathcal{I}g&=\left(\mathcal{N_{\sigma}}-\overline{\mathcal{N}}\right)g
    \end{align}
    is injective and has dense range.
\end{theorem}
\begin{proof}
     We firstly prove that $\mathcal{I}$ is injective. Let $\mathcal{I}g=0$, then $\mathcal{L}[\mathcal{I}g]=0$ for all $s=k+i\sigma, k\in\mathbb{R}$. By the injective result  on $I$ in Theorem \ref{thm14}, we have $\mathcal{L}[g]=0$ in $W(B)$. Hence, $g=0$ in $W(B)\bigoplus W^*(B)$.
    To show the denseness of range, we rewrite $\mathcal{I}$ as
    {\footnotesize
    \begin{align*}
        &[\mathcal{I}g](t,x)=-[\mathcal{AV_{\sigma}}g](t,x)+[\overline{\mathcal{A}}\overline{\mathcal{V}}g](t,x)\\
        =&-\int^{+\infty}_{-\infty}\int_{\partial D}\frac{\partial u_*(t-\tau, x; y)}{\partial \nu(y)}\int^{+\infty}_{-\infty}\int_B \frac{\delta(\tau-\xi-|y-z|)}{4\pi|y-z|}e^{2\sigma(\tau-\xi)}g(\xi, z) dzd\xi ds(y)d\tau\\
        +&\int^{+\infty}_{-\infty}\int_{\partial D}\frac{\partial u_*(-(t-\tau), x; y)}{\partial \nu(y)}e^{2\sigma(t-\tau)}\int^{+\infty}_{-\infty}\int_B\frac{\delta(\tau-\xi+|y-z|)}{4\pi|y-z|}e^{2\sigma(\tau-\xi)}g(\xi, z)dzd\xi ds(y)d\tau,
    \end{align*}}
and its adjoint operator $\mathcal{I}^*$ as
    {\footnotesize
    \begin{align*}
    &~~\mathcal{L}[\mathcal{I^*}(f, h)](\xi,z)\\
    =&-\int^{+\infty}_{-\infty}\int_B\int^{+\infty}_{-\infty}\int_{\partial D}e^{-2\sigma t}\frac{\partial u_*(t-\tau,x;y)}{\partial \nu(y)}\frac{\delta(\tau-\xi-|y-z|)}{4\pi|y-z|}e^{2\sigma\tau}f(t,x)ds(y)d\tau dxdt\\
    &+\int^{+\infty}_{-\infty}\int_B\int^{+\infty}_{-\infty}\int_{\partial D}\frac{\partial u_*(\tau-t,x;y)}{\partial \nu(y)}\frac{\delta(\tau-\xi+|y-z|)}{4\pi|y-z|}h(t,x)ds(y)d\tau dxdt.
    \end{align*}}
Here $\mathcal{I}^*$ is from $W^{-p}_{\sigma}\left(\mathbb{R}, L^2(B)\right)\bigoplus \left(W^*\right)^{-p}_{\sigma}\left(\mathbb{R}, L^2(B)\right)$ to $W^{-p-3/2}_{\sigma}\left(\mathbb{R}, L^2(B)\right)$.
Suppose that $\mathcal{I^*}(f, h)=0$ and take its Laplace transform with $s=k+i\sigma$, we have that
    {\footnotesize
    \begin{align*}
        \mathcal{L}[\mathcal{I^*}(f, h)](z)=&-\int^{+\infty}_{-\infty}\int_B\int^{+\infty}_{-\infty}\int_{\partial D}\frac{e^{is(\tau-|y-z|)}}{4\pi|y-z|}\frac{\partial u_*(t-\tau,x;y)}{\partial \nu(y)}f(t,x)e^{-2\sigma(t-\tau)}ds(y)d\tau dxdt\\
        &+\int^{+\infty}_{-\infty}\int_B\int^{+\infty}_{-\infty}\int_{\partial D}\frac{e^{is(\tau+|y-z|)}}{4\pi|y-z|}\frac{\partial u_*(\tau-t,x;y)}{\partial \nu(y)}h(t,x)ds(y)d\tau dxdt\\
        =&-\int_{\partial D}\frac{e^{-is|y-z|}}{4\pi|y-z|}\left(\int^{+\infty}_{-\infty}e^{i\overline{s}\tau}\int^{+\infty}_{-\infty}\int_B\frac{\partial u_*(t-\tau,x;y)}{\partial \nu(y)}f(t,x)e^{-2\sigma t}dxdtd\tau\right)ds(y)\\
        &+\int_{\partial D}\frac{e^{is|y-z|}}{4\pi|y-z|}\left(\int^{+\infty}_{-\infty}e^{is\tau}\int^{+\infty}_{-\infty}\frac{\partial u_*(\tau-t,x;y)}{\partial \nu(y)}h(t,x)dxdtd\tau\right)ds(y)\\
        =&-\int_{\partial D}\frac{e^{-is|y-z|}}{4\pi|y-z|}\int_B\hat{f}(s,x)\frac{\partial \hat{u}_*(-\overline{s},y;x)}{\partial \nu(y)}dxds(y)\\
        &+\int_{\partial D}\frac{e^{is|y-z|}}{4\pi|y-z|}\int_B\hat{h}(s,x)\frac{\partial \hat{u}_*(s,x;y)}{\partial \nu(y)}dxds(y)\\
        =&0,
    \end{align*}}
    for $z\in B$.
    Similar to the proof in Theorem \ref{thm14}, we have that $\mathcal{I^*}$ is injective and thus $\mathcal{I}$ has dense range.   
\end{proof}

\subsection{Time domain linear sampling method for passive imaging}
In this subsection, we introduce the time domain linear sampling method for passive imaging. We first define the modified product volume operator $\mathbb{V}$ and show its properties in the following theorem.
\begin{theorem}\label{thm20}
The modified product volume operator
\begin{align*}
    \mathbb{V}: W^p_{\sigma}\left(\mathbb{R}, L^2(B)\right)\rightarrow& H^{p-1}_{\sigma}\left(\mathbb{R}, H^{1/2}(\partial D)\right)\times H^{p-1}_{\sigma}\left(\mathbb{R}, H^{1/2}(\partial D)\right)\\
    \mathbb{V} f&=\left( \overline{\mathcal{V}}f, \mathcal{V_{\sigma}}f  \right)
\end{align*}
is injective and has dense range.
\end{theorem}
\begin{proof}
The injectivity can be proved from the definition of $\mathbb{V}$ and Theorem \ref{thm18} directly.  The dense range of $\mathbb{V}$ can be given by showing the injectivity of its adjoint operator $\mathbb{V}^*$. For $\hat{f}\in W(B)$, take the inverse Laplace transform of $\overline{V}\hat f$ defined in \eqref{VV} and
$V_{\sigma}\hat f$ defined in \eqref{vsigma}, we have
\ben 
\left(\overline{\mathcal{V}}f\right)(t,x)=\int_{-\infty}^{+\infty}\int_Be^{2\sigma(t-\xi)}\frac{\delta (\xi-t-|y-x|)}{4\pi|y-x|}f(\xi, y)dyd\xi,
\enn
and
\ben
\left(\mathcal{V_{\sigma}}f \right)(t,x)=\int_{-\infty}^{+\infty}\int_Be^{2\sigma(t-\xi)}\frac{\delta(t-\xi-|x-y|)}{4\pi|x-y|}f(\xi, y)dyd\xi.
\enn
Thus, we get the explicit form of $\mathbb{V}$. 
Consider its adjoint operator
    \begin{align*}
         \mathbb{V^*}&: H^{1-p}_{\sigma}\left(\mathbb{R}, H^{-1/2}(\partial D)\right)\times H^{1-p}_{\sigma}\left(\mathbb{R}, H^{-1/2}(\partial D)\right)\rightarrow W^{-p}_{\sigma}\left(\mathbb{R}, L^2(B)\right)\\
         &[\mathbb{V^*}(g, h)](\xi,y)=\int_{\partial D}\int_{-\infty}^{+\infty}\frac{\delta(\xi-t-|x-y|)}{4\pi|x-y|}g(t,x)dt ds(x)\\
         &~~~~~~~~~~~~~~~~~~~+\int_{\partial D}\int_{-\infty}^{+\infty}\frac{\delta(t-\xi-|x-y|)}{4\pi|x-y|}h(t,x)dt ds(x)
    \end{align*}
with $g\in H^{1-p}_{\sigma}\left(\mathbb{R}, H^{-1/2}(\partial D)\right)$ and $h\in H^{1-p}_{\sigma}\left(\mathbb{R}, H^{-1/2}(\partial D)\right)$. Assume that \\$\mathbb{V^*}(g, h)=0$ and take its Laplace transform with $s=k+i\sigma,\ k\neq 0$, we have
\begin{align*}
0&=\mathcal{L}[\mathbb{V^*}(g, h)](s,y)\\
        &=\int_{\partial D}\frac{e^{is|x-y|}}{4\pi|x-y|}\int_{-\infty}^{+\infty}g(t, x)e^{ist}dtds(x)+\int_{\partial D}\frac{e^{-is|x-y|}}{4\pi|x-y|}\int_{-\infty}^{+\infty}h(t, x)e^{ist}dtds(x)\\
        &=\int_{\partial D}\frac{e^{is|x-y|}}{4\pi|x-y|}\hat{g}(s,x)ds(x)+\int_{\partial D}\frac{e^{-is|x-y|}}{4\pi|x-y|}\hat{h}(s,x)ds(x),\quad y\in B.
    \end{align*}
    Following the proof of Theorem \ref{thm14}, it follows that both $g(t,x)$ and $h(t,x)$ vanish in $H^{1-p}_\sigma(\mathbb{R}, H^{-1/2}(\partial D))$. Therefore, the injectivity of $\mathbb{V^*}$ implies that $\mathbb{V}$ has dense range. This completes the proof.
\end{proof}

Now we present the main result in this paper, the time domain linear sampling method that determines the support $D$ using directly the time domain measurements. For any sampling point $z\in\R^3$,  let $\zeta\in C_c^{\infty}(\mathbb{R})$ be a smooth function and $\tau\in\mathbb{R}$, we define a family of monopole test functions $\phi_{z, \tau}$ by
\begin{equation}\label{test function}
    \phi_{z, \tau}(t, x):=
    \left[\zeta(\cdot)*{\color{lxl}\Phi}(\cdot,x;z)\right](t-\tau),\quad(t, x)\in\R\times\left\{\mathbb{R}^3\backslash\{z\}\right\}.
\end{equation}
This test function is a time domain point source at $z$; here $\tau$ is a time shift. We have the following result.
\begin{theorem}\label{thm21}
Let $\sigma>0$, $\epsilon>0$ and $\tau\in\mathbb{R}$.
\begin{itemize}
\item[(a)] 
    If $z\in D$, then there exists   $g_{z, \tau, \epsilon}\in W^{3/2}_{\sigma}\left(\mathbb{R}, L^2(B)\right)$ such that
    \ben\label{eqI1}
    \lim_{\epsilon\rightarrow 0}\Vert \mathcal{I}g_{z, \tau, \epsilon}-\phi_{z, \tau}\Vert_{H^0_{\sigma}(\mathbb{R}, L^2(B))}= 0.
    \enn
    Moreover, it holds that
    \begin{align*}
    \lim_{\epsilon\rightarrow 0}\Vert \mathcal{V_{\sigma}}g_{z, \tau, \epsilon}\Vert_{H^{1/2}_{\sigma}(\mathbb{R}, H^{1/2}(\partial D))}<\infty \quad \text{and}\quad 
        \lim_{\epsilon\rightarrow0}\Vert \overline{\mathcal{V}}g_{z, \tau, \epsilon}\Vert_{H^{1/2}_{\sigma}(\mathbb{R}, H^{1/2}(\partial D))}\textless\infty.
    \end{align*}
\item[(b)] 
    If $z\notin D$, then for every $g_{z, \tau, \epsilon}\in W^{3/2}_{\sigma}(\mathbb{R}, L^2(B))$ with 
    \be\label{eq:I}
    \lim_{\epsilon\rightarrow 0}\Vert \mathcal{I}g_{z, \tau, \epsilon}-\phi_{z, \tau}\Vert_{H^0_{\sigma}(\mathbb{R}, L^2(B))}=0,
    \en
    it holds that $\lim_{\epsilon\rightarrow0}\Vert g_{z, \tau, \epsilon}\Vert_{W^{3/2}_{\sigma}(\mathbb{R}, L^2(B))}=\infty$ and
{        \begin{align*}
            \lim_{\epsilon\rightarrow 0}\Vert \mathcal{V_{\sigma}}g_{z, \tau, \epsilon}\Vert_{H^{1/2}_{\sigma}(\mathbb{R}, H^{1/2}(\partial D))}=\infty\quad &\text{or}\quad \lim_{\epsilon\rightarrow 0}\Vert \overline{\mathcal{V}}g_{z, \tau, \epsilon}\Vert_{H^{1/2}_{\sigma}(\mathbb{R}, H^{1/2}(\partial D))}=\infty.
        \end{align*}}
\end{itemize}
\end{theorem}
\begin{proof}
    (a) Let $z\in D$, it follows that $\phi_{z, \tau}|_{\partial D}(t, \cdot)\in H^{1/2}_{\sigma}(\mathbb{R}, H^{1/2}(\partial D))$. Then we introduce $\tilde{\phi}:=\left(0,-\phi_{z, \tau}|_{\partial D}(t, \cdot)\right)\in { H^{1/2}_{\sigma}\left(\mathbb{R}, H^{1/2}(\partial D)\right)\times H^{1/2}_{\sigma}\left(\mathbb{R}, H^{1/2}(\partial D)\right)}$. Given any $\epsilon>0$, from Theorem \ref{thm20}, there exists a sequence $g_{z, \tau, \epsilon}\in W^{3/2}_{\sigma}\left(\mathbb{R}, L^2(B)\right)$ such that
    \ben
     \left\|\mathbb{V} g_{z, \tau, \epsilon}-\tilde{\phi}\right\|_{{ H^{1/2}_{\sigma}\left(\mathbb{R}, H^{1/2}(\partial D)\right)\times H^{1/2}_{\sigma}\left(\mathbb{R}, H^{1/2}(\partial D)\right)}}<
    \frac{\epsilon}{2\ \mathrm{max}\{\Vert \mathcal{A}\Vert, \Vert \overline{\mathcal{A}}\Vert\}+1}. 
    \enn
Then it follows that
    \begin{equation*}
        \Vert \overline{\mathcal{V}}g_{z, \tau, \epsilon}+0\Vert_{H^{1/2}_{\sigma}(\mathbb{R}, H^{1/2}(\partial D))}\textless \ \frac{\epsilon}{2\ \mathrm{max}\{\Vert \mathcal{A}\Vert, \Vert \overline{\mathcal{A}}\Vert\}+1},
    \end{equation*}
    and
    \begin{equation*}
        \Vert \mathcal{V_{\sigma}}g_{z, \tau, \epsilon}+\phi_{z, \tau}|_{\partial D}(t, \cdot)\Vert_{H^{1/2}_{\sigma}(\mathbb{R}, H^{1/2}(\partial D))}\textless \ \frac{\epsilon}{2\ \mathrm{max}\{\Vert \mathcal{A}\Vert, \Vert \overline{\mathcal{A}}\Vert\}+1}.
    \end{equation*}
  Due to $\mathcal{I}=-\mathcal{AV_{\sigma}}+\ov{\mathcal{A}}\ov{\mathcal{V}}$ and the properties of $\mathcal{A}$ and $\ov{\mathcal{A}}$ in Theorem \ref{thm18}, it follows that
    \begin{align*}
        \Vert \mathcal{I}g_{z, \tau, \epsilon}-&\phi_{z, \tau}\Vert_{H^0_{\sigma}(\mathbb{R}, L^2(B))}\textless \epsilon;
    \end{align*}
moreover, $\mathcal{V_{\sigma}}g_{z, \tau, \epsilon}$ and $\overline{\mathcal{V}}g_{z, \tau, \epsilon}$ remain bounded in ${H^{1/2}_{\sigma}(\mathbb{R}, H^{1/2}(\partial D))}$ as $\epsilon\rightarrow0$. 
 
(b) Let $z\notin D$, for any $g_{z, \tau, \epsilon}\in W^{3/2}_{\sigma}\left(\mathbb{R}, L^2(B)\right)$ that satisfies \eqref{eq:I}, suppose that 
\ben
\Vert \mathcal{V_{\sigma}}g_{z, \tau, \epsilon}\Vert_{H^{1/2}_{\sigma}(\mathbb{R}, H^{1/2}(\partial D))}< M\quad\text{and}\quad\Vert \overline{\mathcal{V}}g_{z, \tau, \epsilon}\Vert_{H^{1/2}_{\sigma}(\mathbb{R}, H^{1/2}(\partial D))}< M\quad\text{as}\quad\epsilon\rightarrow 0.
\enn
Choose a sequence $\{\epsilon_n\}$ such that $\epsilon_n\rightarrow 0$ as $n\rightarrow \infty$, then $\left\{\mathcal{V_{\sigma}}g_{z, \tau, \epsilon_n}\right\}$ is bounded. Thus, there exists a subsequence $\{\epsilon_{n_m}\}$ such that $\mathcal{V_{\sigma}}g_{z, \tau, \epsilon_{n_m}}$ is weakly convergent to a function in $H^{1/2}_{\sigma}(\mathbb{R}, H^{1/2}(\partial D))$ as $m\rightarrow\infty$, denote by $f$. 
Similarly for $\{\epsilon_{n_m}\}$ and $\left\{\overline{\mathcal{V}}g_{z, \tau, \epsilon_{n_m}}\right\}$, there exists a subsequence $\{\epsilon_{n_{m_l}}\}$ such that $\overline{\mathcal{V}}g_{z, \tau, \epsilon_{n_{m_l}}}$ is weakly convergent to a function in $H^{1/2}_{\sigma}(\mathbb{R}, H^{1/2}(\partial D))$ as $l\rightarrow\infty$, denote by $h$.
Together with the definition of $\mathcal{I}$, it follows that $\mathcal{-A}f+\overline{\mathcal{A}}h=\phi_{z, \tau}$ in $H^0_{\sigma}(\mathbb{R}, L^2(B))$. Taking the Laplace transform yields that
\ben
-A\hat{f}+\overline{A}\hat{h}=\mathrm{exp}(-\sigma\tau)\hat{\zeta}(s)\hat{\Phi}_s(\cdot; z)\quad\text{in}\quad B
\enn
for almost every $s=k+i\sigma, k\in\mathbb{R}$. The different radiating behavior of $A\hat{f}$ and $\overline{A}\hat{h}$ leads to that $-A\hat{f}=\mathrm{exp}(-\sigma\tau)\hat{\zeta}(s)\hat{\Phi}_s(\cdot; z)$ in $B$; this contradicts Lemma \ref{thm10} since $z\notin D$. This allows to show that either $\Vert \mathcal{V_{\sigma}}g_{z, \tau, \epsilon}\Vert_{H^{1/2}_{\sigma}(\mathbb{R}, H^{1/2}(\partial D))}$ or $\Vert \overline{\mathcal{V}}g_{z, \tau, \epsilon}\Vert_{H^{1/2}_{\sigma}(\mathbb{R}, H^{1/2}(\partial D))}$ approaches to infinity. From Theorem \ref{thm18}, the operators $\mathcal{V_{\sigma}}$ and $\overline{\mathcal{V}}$ are bounded, then it follows that $\Vert g_{z, \tau, \epsilon}\Vert_{W^{3/2}_{\sigma}(\mathbb{R}, L^2(B))}\rightarrow \infty$ as $\epsilon \to 0$. This completes the proof.
\end{proof}

    We  remark that the operator $\mathcal{I}$ is a modified time domain operator in similar spirit to \cite{CakoniHaddarLechleiter2019}, which allows to develop the mathematical theory of the linear sampling method; moreover setting $\sigma=0$ directly allows to implement the numerical algorithm. We obtain the approximation of $\Vert \mathcal{V_\sigma}g_{z,\tau,\epsilon}\Vert_{H^{1/2}_\sigma(\mathbb{R},H^{1/2}(\partial D))}$ or $\Vert \mathcal{\overline{V}}g_{z,\tau,\epsilon}\Vert_{H^{1/2}_\sigma(\mathbb{R},H^{1/2}(\partial D))}$ and it is enough for us to illustrate the linear sampling method in our numerical experiments. This is an inherited method in the linear sampling method.

    We also remark that Theorem \ref{thm21} also holds if we replace $\mathcal{I}$ by the volume near field operator $\mathcal{N}$.  We state the result in Corollary \ref{thm24}. Later on we will provide several numerical experiments to test the linear sampling method using this operator $\mathcal{N}$.

\begin{corollary}\label{thm24}
Let $\sigma>0$, $\epsilon>0$ and $\tau\in\mathbb{R}$.
\begin{itemize}
\item[(a)] 
    If $z\in D$, then there exists  $g_{z, \tau, \epsilon}\in W^{3/2}_{\sigma}\left(\mathbb{R}, L^2(B)\right)$ such that
    \be\label{eq:N1}
    \lim_{\epsilon\rightarrow 0}\Vert \mathcal{N}g_{z, \tau, \epsilon}-\phi_{z, \tau}\Vert_{H^0_{\sigma}(\mathbb{R}, L^2(B))}= 0.
    \en
    Moreover, it holds that
    \begin{align*}
    \lim_{\epsilon\rightarrow 0}\Vert \mathcal{V_{\sigma}}g_{z, \tau, \epsilon}\Vert_{H^{1/2}_{\sigma}(\mathbb{R}, H^{1/2}(\partial D))}<\infty \quad \text{and}\quad 
        \lim_{\epsilon\rightarrow0}\Vert \overline{\mathcal{V}}g_{z, \tau, \epsilon}\Vert_{H^{1/2}_{\sigma}(\mathbb{R}, H^{1/2}(\partial D))}\textless\infty.
    \end{align*}   
\item[(b)] 
    If $z\notin D$, then for every $g_{z, \tau, \epsilon}\in W^{3/2}_{\sigma}(\mathbb{R}, L^2(B))$ with \eqref{eq:N1}, it holds that 
\begin{align*}
            \lim_{\epsilon\rightarrow0}\Vert g_{z, \tau, \epsilon}\Vert&_{W^{3/2}_{\sigma}(\mathbb{R}, L^2(B))}=\infty,\\
            \lim_{\epsilon\rightarrow 0}\Vert \mathcal{V_{\sigma}}g_{z, \tau, \epsilon}\Vert_{H^{1/2}_{\sigma}(\mathbb{R}, H^{1/2}(\partial D))}=\infty\quad &\text{or}\quad \lim_{\epsilon\rightarrow 0}\Vert \overline{\mathcal{V}}g_{z, \tau, \epsilon}\Vert_{H^{1/2}_{\sigma}(\mathbb{R}, H^{1/2}(\partial D))}=\infty.
        \end{align*}
\end{itemize}
\end{corollary}

\section{Numerical experiments}\label{sec:6}
\begin{figure}[htbp]
   \centering     
   \subfigure[]{
   \includegraphics[width=0.4\linewidth]{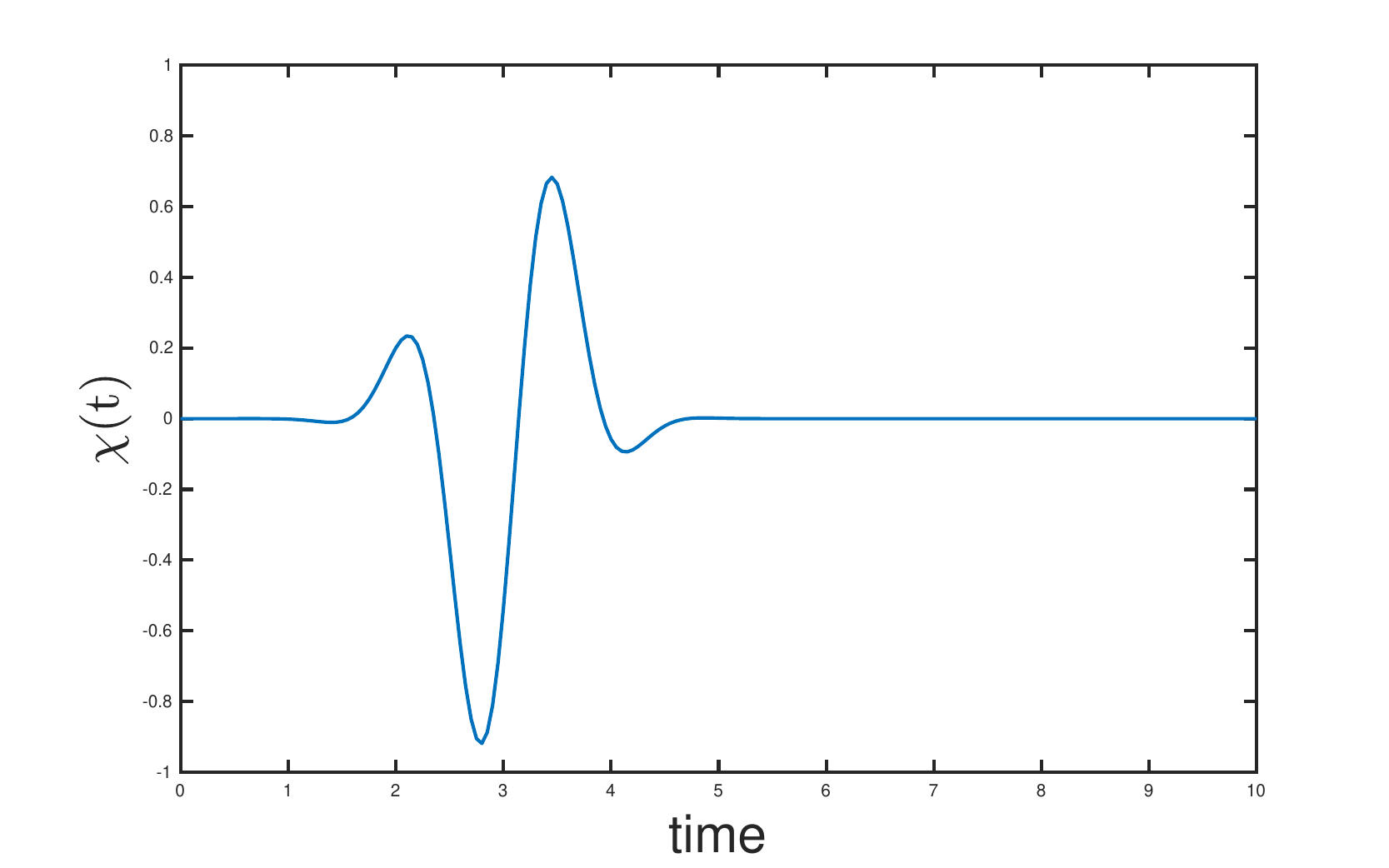}}
   \subfigure[]{
   \includegraphics[width=0.4\linewidth]{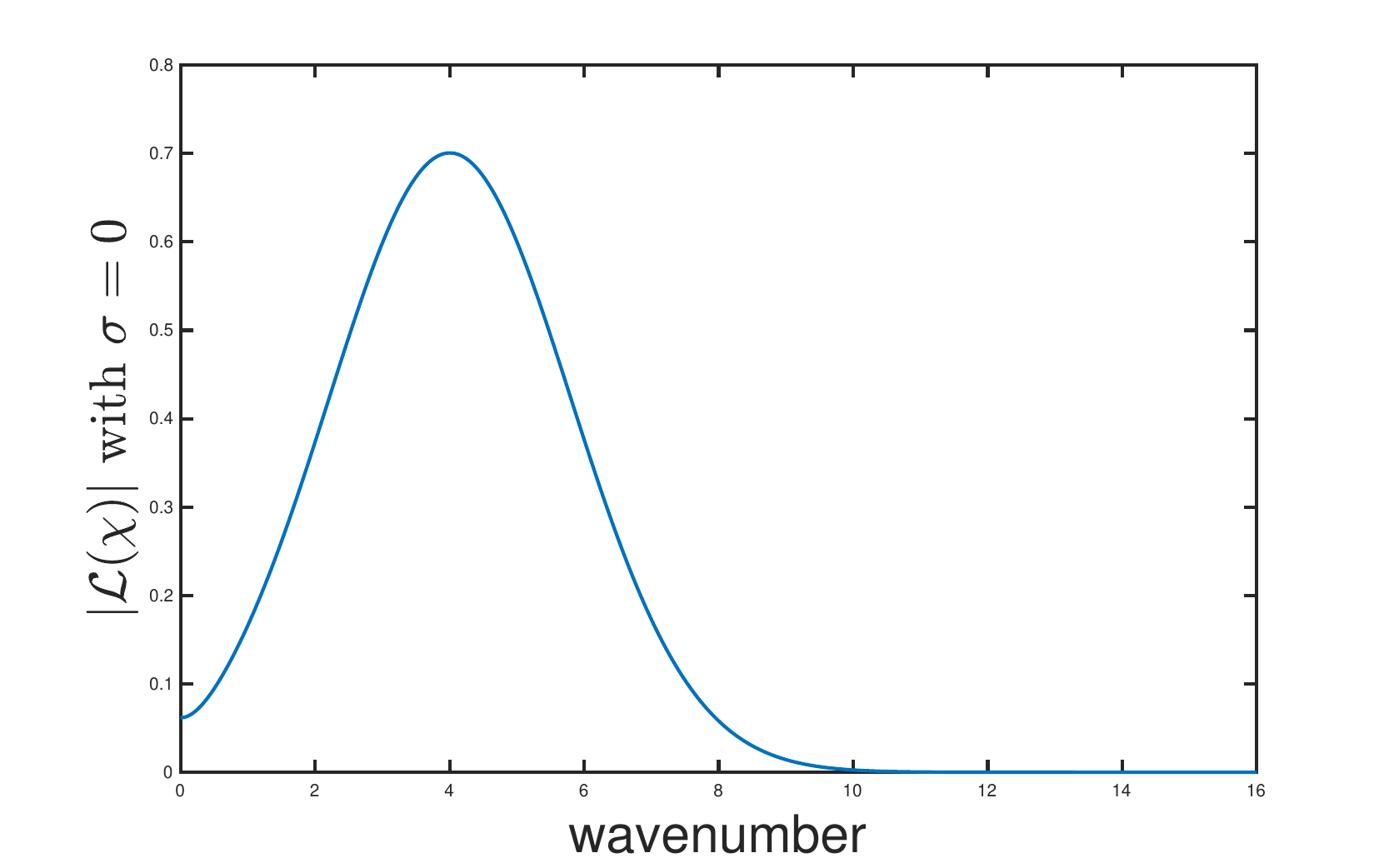}}
   \caption{The pulse $\chi$ (a) and $|\mathcal{L}\chi|$ with $\sigma=0$ (b).}
   \label{Figure 2}
\end{figure}
\subsection{Manipulation of time domain measurements for random sources}\label{sec:2}
\numberwithin{equation}{section}
\setcounter{equation}{0}
To image with random sources, the Helmholtz-Kirchhoff identity and cross-correlation play important roles in the frequency domain, e.g., \cite[Theorem 2.2]{GarnierPapanicolaou2016}. Similarly in the Laplace domain,
for two fixed points $p, q\in \mathbb{R}^3, p\neq q$ and $s=k+i\sigma$ with $0<\sigma\ll1$, we have that
\be\label{hkforphi}
\hat\Phi_s(p;q)-\overline{\hat\Phi_s(p;q)}
 \approx i(s+\overline{s})\int_{\partial B_R}\hat\Phi_s({p;z}) \overline{\hat\Phi_s({q;z})}ds(z),
\en
where $B_R$ is a large ball centered at origin,  and $\partial B_R$ encloses $p$ and $q$ and is far from them; moreover, let $\hat u_*({s,p;q}) = \hat u^{\rm scat}_*({s,p;q}) + {\hat\Phi_s(p;q)}$ where $\hat u^{\rm scat}_*({s,p;q})$ is the unique physical  solution to \eqref{helmeq2} with boundary data $-{\hat\Phi_s(p;q)}$, for $p, q\in \mathbb{R}^3$, $p\neq q$ and $s=k+i\sigma$ with $0<\sigma\ll1$, it holds that
\begin{equation}\label{hkforu}%
\hat u_*({s,p; q})-\overline{\hat{u}_*({s,p; q})} \approx i(s+\overline{s})\int_{\partial B_R}\hat{u}_*({s,p; z})\,\overline{\hat{u}_*({s,q; z})}ds(z).
\end{equation}
 The above two equations and their proofs are in similar spirit to \cite{Garnier2023,GarnierPapanicolaou2016} and we refer to \cite{Garnier2023,GarnierPapanicolaou2016} for more details. A brief discussion is also included in the Appendix.
From \eqref{hkforphi} and \eqref{hkforu}, we get the following relationship for the scattered fields 
\begin{equation}\label{eq:1}
\begin{aligned}
&~~~\hat{u}^{\rm scat}_*({s,p; q})-\overline{\hat{u}^{\rm scat}_*({s,p; q})} \\
&\approx   i(s+\overline{s})\int_{\partial B_R}{\hat{u}_*({s,p; z})}\,\overline{\hat{u}_*({s,q; z})}ds(z) -\left(\hat\Phi_s(p;q)-\ov{\hat\Phi_s(p;q)}\right).
\end{aligned}
\end{equation}
Finally, we use \eqref{eq:1} and the inverse Laplace transform to give the time domain Helmholtz-Kirchhoff identity for the scattered field. 

Before carrying out the identity, we first introduce some notations that will be used later on. 
Given a function $g(t)$, define $\Breve{g}(t)$ by
\be\label{breve:g}
\Breve{g}(t):= g(-t)e^{2\sigma t}\quad\text{with}\quad\sigma>0.
\en
Denote by $g_f$ the convolution in time of $g$ and $f$, that is,
\ben
g_f(t):= \left[g(\cdot)*f(\cdot)\right](t).
\enn
Denote by $\tilde g$ the convolution in time of $g$ and $\Breve{g}$, 
\be\label{tilde:g}
\tilde{g}(t):=\left[g(\cdot)*\Breve{g}(\cdot)\right](t).
\en
It is seen that the Laplace transform of $\Breve{g}$, $g_f$ and   $\tilde g$ with respect to $s=k+i\sigma$ are $\ov{\hat g(s)}$, $\hat g(s)\hat{f}(s)$ and $\hat g(s)\ov{\hat{g}(s)}$, respectively. With the above notation, in the following we denote by $u^{\rm scat}_{f}= u^{\rm scat}_**f$,  $u_{f}= u_**f$,  $\Breve{u}(t)= u_*(-t)e^{2\sigma t}$, and $\Breve{u}^{\rm scat}(t)= u^{\rm scat}_*(-t)e^{2\sigma t}$.

    Let $\sigma>0$ and $\chi:\R\rightarrow\R$ be a square integrable function  such that $\hat{\chi}(s)$ is compactly supported in $\left\{s=k+i\sigma \ |\ k \in [k_1, k_2]\right\}$, multiply $\hat{\chi}(s)\overline{\hat{\chi}(s)}$ on both sides of \eqref{eq:1} and take the inverse Laplace transform, we have that
\begin{equation}\label{eq9}
\begin{aligned}
&\mathcal{L}^{-1}\left[\left(\hat{u}^{\rm scat}_*({s,p;q})-\overline{\hat{u}^{\rm scat}_*({s,p;q})}\right)\cdot\hat{\chi}(s)\overline{\hat{\chi}(s)}\right](t)\\
    =&\left[\left({u^{\rm scat}_*}({\cdot,p;q})-{\Breve{u}^{\rm scat}_*} (\cdot,p;q)\right)*\tilde\chi(\cdot)\right](t)
    ={u^{\rm scat}_{\tilde\chi}}(t,p;q)-{\Breve{u}^{\rm scat}_{\tilde\chi}}(t,p;q),
\end{aligned}    
\end{equation}
and
\begin{equation}\label{eq10}
\begin{aligned}
&\mathcal{L}^{-1}\left[\left(i(s+\overline{s})\int_{\partial B_R}{\hat u_*({s,p;z})}\,\overline{\hat u_*({s,q;z})}ds(z)\right)\cdot\hat{\chi}(s)\overline{\hat{\chi}(s)}\right](t)\\    
=&\mathcal{L}^{-1}\left[\int_{\partial B_R}\overline{\hat{\chi}(s)\hat u_*({s,q;z})}\cdot(is-\overline{is})\hat{\chi}(s)\hat u_*({s,p;z})ds(z)\right](t)\\
=&2\int_{\partial B_R}\left[\Breve{u}_{\chi}({\cdot,q;z})*\left(
        \sigma u_\chi({\cdot,p;z})-u_{\dot\chi}({\cdot,p;z})\right)\right](t)ds(z),
\end{aligned}    
\end{equation}
together with
\begin{equation}\label{eq11}
\begin{aligned}
&\mathcal{L}^{-1}\left[\left(\hat\Phi_s({p;q})-\ov{\hat\Phi_s({p;q})}\right)\cdot\hat{\chi}(s)\overline{\hat{\chi}(s)}\right](t)\\
=&\left[\left(\Phi({\cdot,p;q})-{\Breve{\Phi}({\cdot,p;q})}\right)*\tilde\chi(\cdot)\right](t)
=\Phi_{\tilde\chi}({t,p;q})-\Breve{\Phi}_{\tilde\chi}({t,p;q}).
\end{aligned}    
\end{equation}
From \eqref{eq9}, \eqref{eq10} and \eqref{eq11},
    formally it follows that
    \be\label{eq:2}
    &~~~~~~~~&u^{\rm scat}_{\tilde\chi}({t,p;q}) - \Breve{u}^{\rm scat}_{\tilde\chi}({t,p;q})\\
\nonumber        &\approx &2\int_{\partial B_R}
\left[{\Breve{u}_{\chi}({\cdot,q;z})}*\left(
        \sigma {u_\chi({\cdot,p;z})}-{u_{\dot\chi}({\cdot,p;z})}\right)\right](t)ds(z)
        -\Phi_{\tilde\chi}({t,p;q})+\Breve{\Phi}_{\tilde\chi}({t,p;q}),
    \en
    where $\dot\chi$ is the derivative of $\chi$ with respect to the time variable $t$. 
    
{ The time domain Helmholtz-Kirchhoff identity \eqref{eq:2}   gives a connection between the time domain total fields due to random sources and the subtraction of two  scattered fields. The left hand side of \eqref{eq:2} are the scattered fields measured at $p$ that corresponding to the point source located at $q$, while the right hand side of \eqref{eq:2} are the total fields measured at $p$ and $q$ that corresponding to the point sources located at  $z\in\partial B_R$. For passive imaging, one can only collect the total fields in $B$ due to unknown point sources located randomly on a surface $\partial B_R$ and use such information to approximate the scattered field corresponding to known incident locations with the help of \eqref{eq:2}. These scattered wave fields are used to construct $\mathcal{I}$ in \eqref{operator I} for the time domain linear sampling method  in Theorem \ref{thm21}.  }

\subsection{The numerical inversion scheme}
{{\color{lxl}Although the theoretical analysis requires working with the Laplace parameter $\sigma$ to be positive and sufficiently small.
We  follow the ideas in \cite{ChenHaddar2010,GuoMonkColton2015,HaddarLiu2020} and set $\sigma=0$ in our numerical experiments for simplicity. In this case, the Helmholtz-Kirchhoff identity \eqref{eq:2} in the time
domain is still valid, this is mainly because of the corresponding results in the frequency domain, see references \cite{Garnier2023,Garnier2024}. Then} the right hand side of \eqref{eq:2}, denoted by $c(t,p;q)$, is 
\ben
c(t,p;q)=-2\int_{\partial B_R}
\left[{{u}_{\chi}({-\cdot,q;z})}*{u_{\dot\chi}({\cdot,p;z})}\right](t)ds(z)-\Phi_{\tilde\chi}({t,p;q})+\left[\Phi(-\cdot,p;q)*{\color{lxl}\tilde\chi}(\cdot)\right](t),
\enn
since $\Breve{g}(t)=g(-t)$ in \eqref{breve:g} when $\sigma=0$. We can rewrite $c(t,p;q)$ as
\be\label{ctpq}
c(t,p;q)
&=&-2\int_{\partial B_R}\left[\int_{-\infty}^{+\infty}{u}_{\chi}({\tau,q;z})u_\chi(t+\tau,p;z)d\tau\right]_t^{\prime} ds(z)\\
\nonumber&&-{\color{lxl}\Phi_{\tilde\chi}({t,p;q})+\Phi_{\tilde\chi}(-t,p;q)}.
\en
The equality \eqref{ctpq} allows to construct an operator $\mathcal{C}$  which  can be seen as an approximation to $\mathcal{I}$ by
\be\label{operatorC}
\left[\mathcal{C}g\right](t,p)=\int_B\left[c(\cdot,p;q)*g(\cdot,q)\right](t)ds(q),\quad (t,p)\in\R\times B.
\en

The inverse acoustic scattering problem with the time domain linear sampling method (TDLSM) consists of two steps. First, we use the passive total field data \eqref{data time domain total} to get $c(t,p;q)$, as well as operator $\mathcal{C}$ (direct problem). Second, we solve the linear system $\mathcal{C}g_{z}=\phi_{z}$ and compute $1/\|g_z\|$ for each sampling point $z$ (inverse problem).
For the direct scattering problem, we use the open-source MATLAB toolbox  deltaBEM, e.g., \cite{Dominguez2014}, to get the corresponding scattered field and total field data. For a signal $\chi$ that supported in {$[0,T_0]$}, we use it to generate the {incident wave $u^{\rm inc}$ defined in \eqref{ui}}. In the passive imaging, we consider $L$ incident point sources $z_l$ located randomly on  $\partial B_R$ where
\be\label{source}
z_l=Re^{i\theta_l},\quad \theta_l=\frac{2\pi}{L}(l-1+\beta_l),\quad 1\leq l\leq L,
\en
where $\beta_l$ is drawn from the uniform distribution on $[0, \beta]$. Here $\beta$ can be seen as a {\it random level}. The corresponding total field data is recorded during {\color{lxl}$[0,T]$} and written as
\be\label{data:1}
\left\{u_{{\chi}}({\color{lxl}n\Delta t},x_j;z_l)~:~0\leq n\leq {2N}, 1\leq j\leq J, 1\leq l\leq L\right\}
\en
and
\be\label{data:2}
\left\{u_{{\chi}}({\color{lxl}n\Delta t},y_m;z_l)~:~0\leq n\leq {2N}, 1\leq m\leq M, 1\leq l\leq L\right\}
\en
with {\color{lxl}$\Delta t=T/2N$} for $2N+1$ time steps, $J$ measured locations $x_j$ and $M$ measured locations $y_m$. For $n'=-N,...,0,...,N$, we define
\ben
\varphi(2n'\Delta t,x_j,y_m;z_l)
:=\sum_{n=n_1}^{n_2}u_\chi({\color{lxl}n\Delta t},y_m;z_l)u_\chi\left({\color{lxl}(2n'+n)\Delta t},x_j;z_l\right)\Delta t
\enn
with $n_1=\max\{0,-2n'\}$ and $n_2=\min\{2N,2(N-n')\}$ for $2N+1$ time steps. Using finite difference to obtain the derivative with respect to $t$ in \eqref{ctpq}, discrete form of $c(t,p;q)$ can be represented as
\begin{align*}
    &c(2n'\Delta t, x_j; y_m)=-\frac{\pi R}{L}\sum_{l=1}^L\left\{\varphi\left(2(n'+1)\Delta t,x_j,y_m;z_l\right)-\varphi\left(2(n'-1)\Delta t,x_j,y_m;z_l\right)\right\}\\
    &-{\color{lxl}
    \Phi_{\tilde\chi}(2n'\Delta t,x_j;y_m)+\Phi_{\tilde\chi}(-2n'\Delta t,x_j;y_m)},\quad n'=-(N-1),...,0,...,(N-1).
\end{align*}
We use $\varphi\left(-2(N+1)\Delta t,x_j,y_m;z_l\right)=0$ and  $\varphi\left(2(N+1)\Delta t,x_j,y_m;z_l\right)=0$ to get $c(\cdot, x_j; y_m)$ for $2N+1$ time steps in the numerical implementation. {\color{lxl}Thus the discrete version $C$ of operator $\mathcal{C}$ in \eqref{operatorC} can be represented as
\be
\label{operatorCC}\left[Cg\right]\left(2k\Delta t,x_j\right)
:=2\sum_{m=1}^{M}\sum_{h=k-N}^{k+N} c\left(2(k-h)\Delta t,x_j;y_m\right)g(2h\Delta t,y_m)\Delta t\Delta y 
\en
for $k=-N,...,0,...,N$.

We emphasize that operator $C$ is built by using the {\it passive data}, which is the total field \eqref{data:1} and \eqref{data:2} corresponding to unknown sources $z_l$. To show the effectiveness of our linear sampling method using operator $C$, we would like to compare its reconstructed results with the linear sampling method with {\it active data}, which is the scattered field corresponding to known sources. Consider the pulse $\tilde{\chi}(t)$ in \eqref{tilde:g}, which can be represented as $\tilde\chi(t)=\left[\chi(\cdot)*\chi(-\cdot)\right](t)$ while $\sigma=0$ and supported in $[-T_0,T_0]$, the corresponding  scattered field
\be\label{data:3}
\left\{u^{\rm scat}_{\tilde{\chi}}({\color{lxl}2{n''}\Delta t},x_j;y_m)~:~{\color{lxl}-N\leq{n''}\leq N}, 1\leq j\leq J, 1\leq m\leq M\right\}
\en
is recorded for $2N+1$ time steps during $[-T,T]$, $M$ incident point sources $y_m$, and $J$ measured locations $x_j$. Here $y_m$ and $x_j$ coincide with the ones in \eqref{data:1} and \eqref{data:2}. 
 Thus the discrete version $I$ of the modified imaginary near-field operator  $\mathcal{I}$ in \eqref{operator I} can be represented as
\be\label{operatorII}
\left[Ig\right]\left(2k\Delta t,x_j\right)
:=&&2\sum_{m=1}^{M}\sum_{h=k-N}^{k+N}
\left[u^{\rm scat}_{\tilde{\chi}}(2(k-h)\Delta t,x_j;y_m)\right.\\
\nonumber
&&\left.-u^{\rm scat}_{\tilde{\chi}}(2(h-k)\Delta t,x_j;y_m)\right]g(2h\Delta t,y_m)\Delta t\Delta y
\en
for $k=-N,...,0,...,N$. The discrete version $N$ of the standard near-field operator $\mathcal{N}$ can be represented as 
\be\label{operatorNN}
\left[Ng\right]\left(2k\Delta t,x_j\right)
:=2\sum_{m=1}^{M}\sum_{h=k-N}^{k+N} 
u^{\rm scat}_{\tilde{\chi}}\left(2(k-h)\Delta t,x_j;y_m\right)g(2h\Delta t,y_m)\Delta t\Delta y
\en
for $k=-N,...,0,...N$. The test function we used is $\phi_{\tilde{\chi},z,\tau}\in\R^{(2N+1)\times M}$, which is the discrete form of the test function defined in \eqref{test function} with {$\zeta=\tilde{\chi}$}. That is, for each $k=-N,...,0,...N$,
\ben
\phi_{\tilde{\chi},z,\tau}(2k\Delta t,x_j):=
\frac{\tilde{\chi}\left(2k\Delta t-\tau-|x_j-z|\right)}{4\pi|x_j-z|}.
\enn

For the inverse scattering problem, we solve $Ag_z=\phi_{\tilde{\chi},z,\tau}$ with $Ag_z:=Cg,Ig,Ng$ in \eqref{operatorCC}, \eqref{operatorII} and \eqref{operatorNN}, respectively. The numerical inversion is done along the following scheme.
\begin{itemize}
    \item[(a)]
    Compute $P$ singular values $(\sigma_p)_{p=1,...,P}$ of $A$ with the largest magnitude and the corresponding singular value decomposition is $A=USV^*$. The value of $P$ is determined by the ratio $|\sigma_P/\sigma_1|\geq 0.005$. We denote $U_{P}$, $S_{P}$ and $V_{P}$ the corresponding matrices of $U$, $S$ and $V$ after truncated with the ratio.
    \item[(b)]
    Choose an appropriated shift in time $\tau$ and evaluate the truncated solutions $g_{z,P}$ to $Ag_z=\phi_{\tilde{\chi},z,\tau}$. The explicit formula corresponding to the right hand side $\phi_{\tilde{\chi},z,\tau}$ is
    \ben
    g_{z,P}= V_{P}S^{-1}_{P}U_{P}^*\phi_{\tilde{\chi},z,\tau}.
    \enn
    \item[(c)]
    For each sampling point $z$, we consider the indicator function:
    \ben
    I_A(z):=\frac{1}{\|g_{z,P}\|_F},
    \enn
    Here $\|\cdot\|_F$ represents the Frobenius norm.
    \item[(d)]
    Plot the values of $I_A(z)$ in the sampling region.
\end{itemize}
}%

\subsection{Numerical examples}

{\color{lxl}Our theory is developed in three dimensions, but we shall restrict ourselves to carrying out numerical experiments in a two-dimensional setting. The reason to restrict theoretical work to 3D is that the fundamental solution to the wave equation in the time domain is much more complicated in 2D than in 3D and we want to avoid further technical details. We expect that all the theoretical results are valid in 2D since both the time domain Helmholtz-Kirchhoff identity and analysis for the operators can be similarly done in 2D dealing with  its fundamental solution. See \cite{CakoniColton2014} for a detailed analysis in 2D in the frequency domain.} 

In our numerical experiments, we use the pulse $\chi(t)=\mathrm{sin}(4t)e^{-1.6(t-3)^2}$, which corresponds to a signal at {\color{lxl}a central frequency $4$ (i.e. a central wavelength roughly equals to $1.6$)} and with bandwidth equal to $8$. The function $\chi$ and its Fourier transform   are depicted in Figure \ref{Figure 2}. 

{\color{lxl}In the passive imaging, we consider $L=80$ random sources $z_l$ located on a circle $R=20$ centered at $(0,0)$ with the random level $\beta=0.1$ in \eqref{source}. The corresponding total field $u_\chi$ is computed at $30$ equiangular points $a_i$ on a circle with a radius of $2.5$ and centered at $(1,1)$ from $[-\pi,\pi)$.  The total fields $u_{{\chi}}(\cdot,x_j;z_l)$ in \eqref{data:1} and $u_{{\chi}}(\cdot,y_m;z_l)$ in \eqref{data:2} are obtained by set $x_j=a_{2j}$ and $y_m=a_{2m-1}$ for $J=M=15$. The recording time of the total field $u_{\chi}$ is {$[0, 40]$} with time step $\Delta t = 0.1$. In the sub-figures for $I_C$, we use crosses ''$\times$'' to represent the location these measurement points $\{a_i\}_{i=1,...,30}$.

To compare the reconstructions with active imaging, we use the same locations of $x_j$ and $y_m$ described above and compute the active data in \eqref{data:3}. In the sub-figures for $I_I$ and $I_N$, we use crosses "$\times$" and dots "$\cdot$" to represent $\{x_j\}_{j=1,...,15}$ and $\{y_m\}_{m=1,...,15}$, respectively.} 

The simulated data are perturbed by adding random noise by $u_{\chi,\delta}=u_\chi+\delta(2\mu-1)|u_\chi|$, where the noise level $\delta=5\%$ and $\mu$ is uniform distribution on $[0,1]$. Similarly, we can get the noisy scattered field data $u^{\rm scat}_{\tilde{\chi},\delta}$ in \eqref{data:3} for $Ig$ and $Ng$.  The sampling region is a circle $R=2.2$ centered at $(1,1)$ and values outside of the circle is set to be zero. In each figure, the black solid line represents the boundary of the obstacle.
 
{\bf Example 1.} {\color{lxl}In the first example, we set domain $B$ to be an annulus with inner radius $2.2$ and outer radius $2.8$ centered at $(1,1)$. We compare the reconstructions under the different choices of measured points inside $B$. In the first row of Figure \ref{Figure 3}, we give the reconstructions with measurements $u_\chi$ and $u^{\rm scat}_{\tilde{\chi}}$ on two concentric circles (radii 2.4 and 2.6) within $B$. In the second row of Figure \ref{Figure 3}, the reconstructions are given with these fields measured on a single circle of radius 2.5 as described before. Here the obstacle is an ellipse that can be  parameterized as
\ben
\partial D_{\text{Ellipse}}(\theta)=\left(1 + 0.25\mathrm{cos}\theta, 1 + 0.5\mathrm{sin}\theta\right),\quad \theta\in [0, 2\pi].
\enn 
It can be seen from Figure \ref{Figure 3} that our sampling method for passive imaging using the total field data corresponding from unknown point sources $I_C$ gives a comparable reconstruction to the results from active imaging data that using  $I_N$ and $I_I$. 
Moreover, different choices of measured points inside domain $B$ give similar reconstructed results. Therefore, in the following examples, we choose the measured points located on a circle (seen as a very thin domain) as described in the beginning of this subsection.
  \begin{figure}[htbp]\label{Figure 3}
   \centering     
   \subfigure[$I_C$]{
   \includegraphics[width=0.3\linewidth, height=3.9cm]{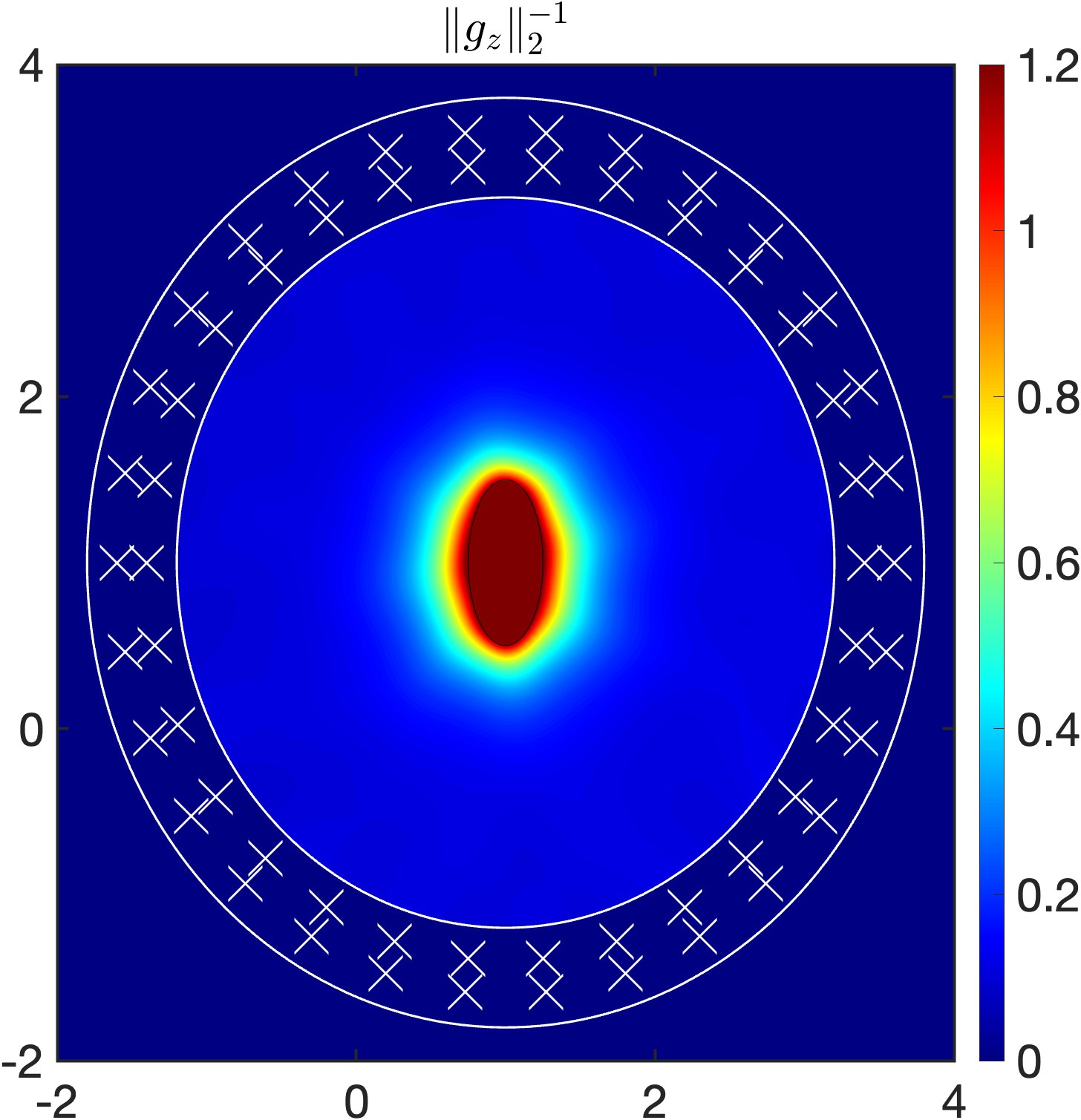}}
   \subfigure[$I_I$]{
   \includegraphics[width=0.3\linewidth, height=3.9cm]{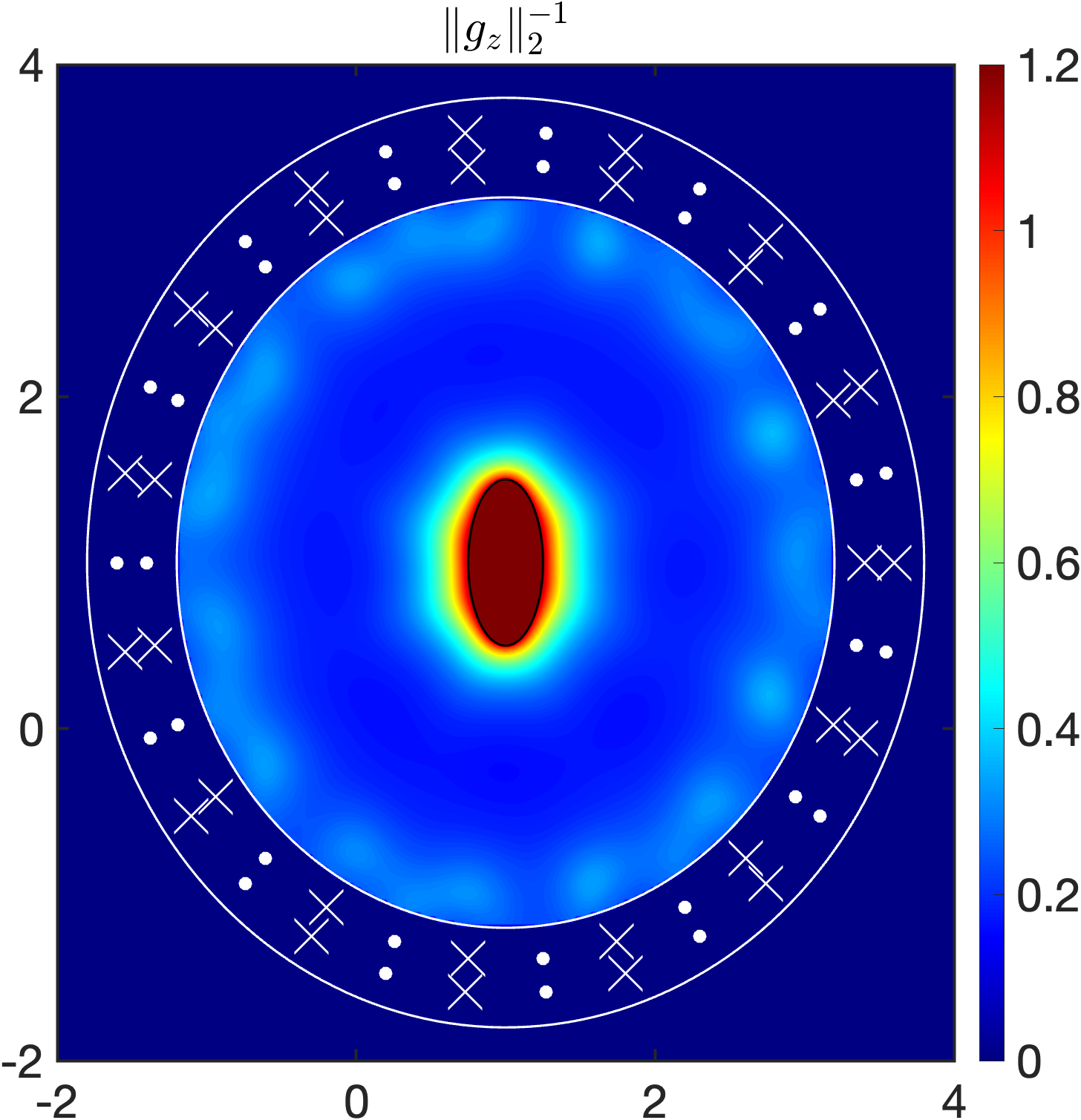}}
   \subfigure[$I_N$]{
   \includegraphics[width=0.3\linewidth, height=3.9cm]{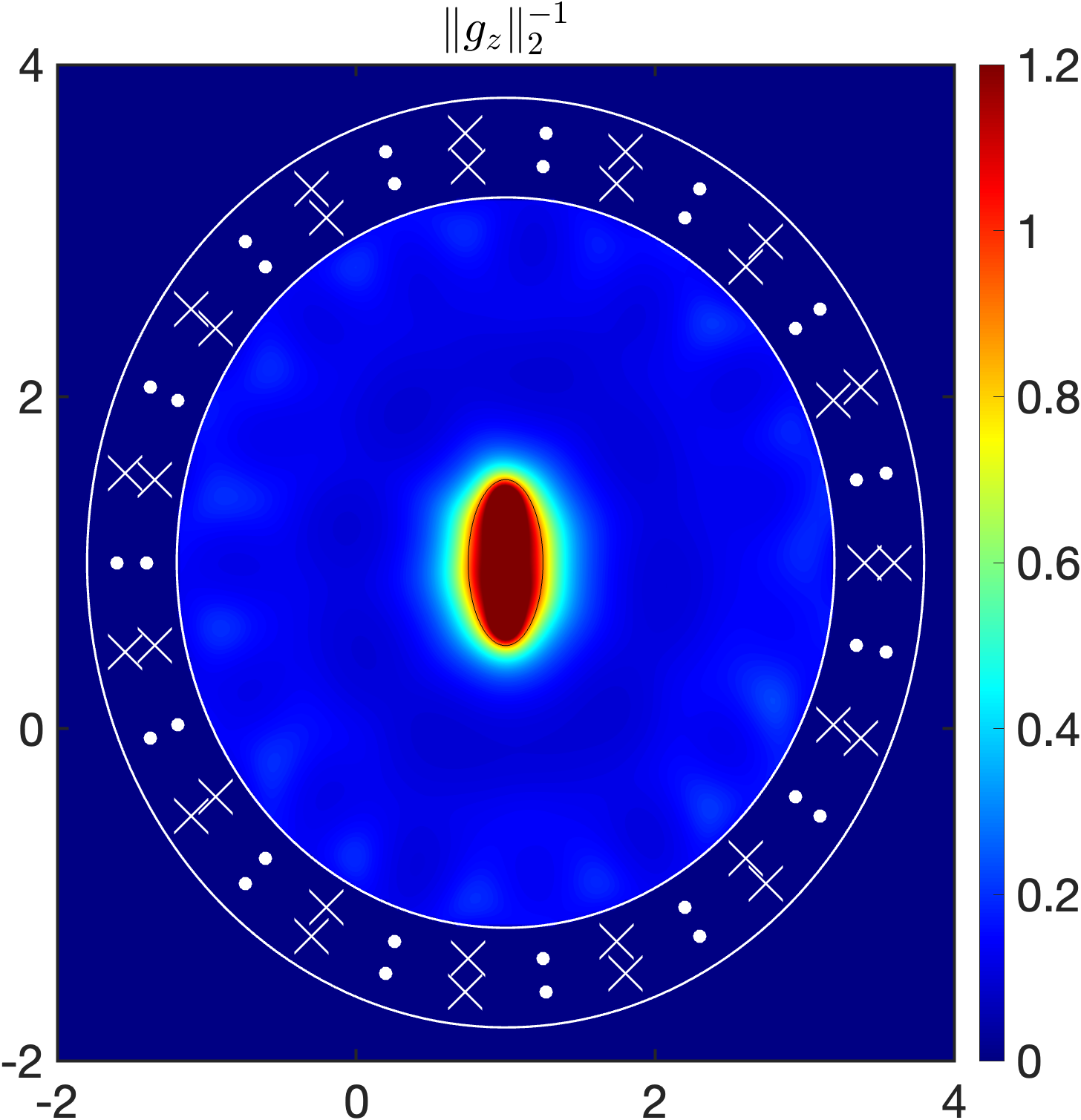}}

   \subfigure[$I_C$]{
   \includegraphics[width=0.3\linewidth, height=3.9cm]{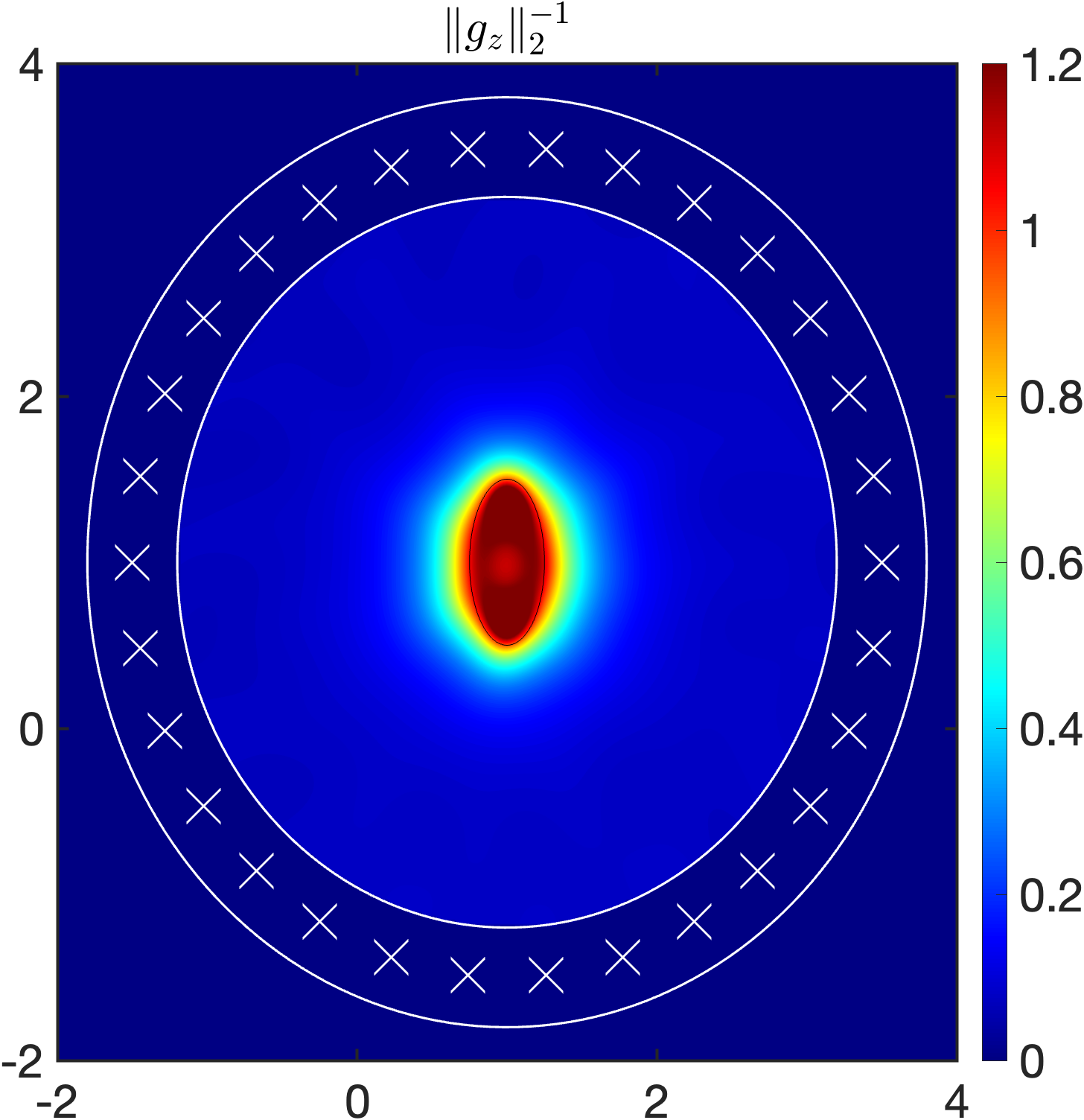}}
   \subfigure[$I_I$]{
   \includegraphics[width=0.3\linewidth, height=3.9cm]{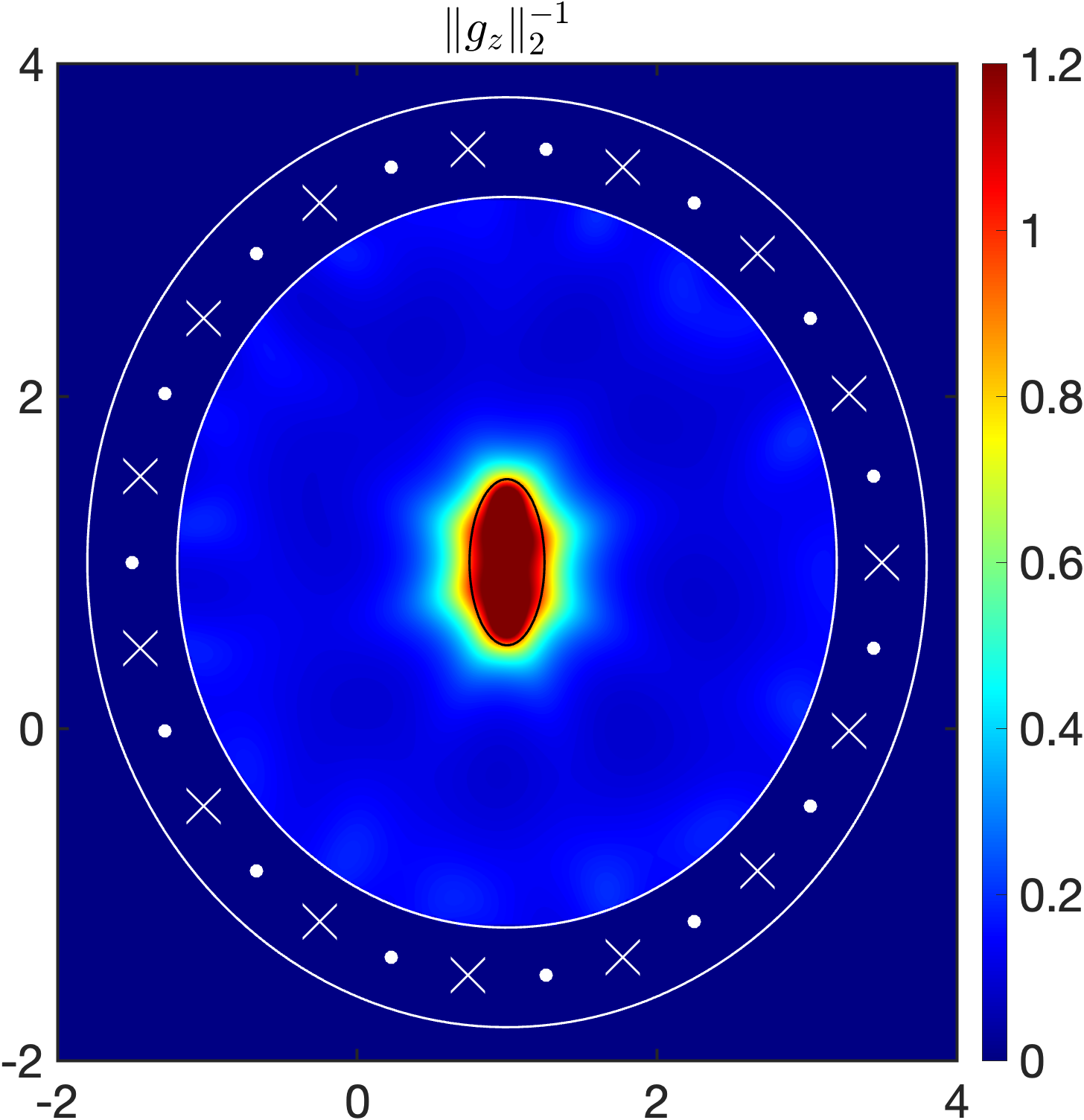}}
   \subfigure[$I_N$]{
   \includegraphics[width=0.3\linewidth, height=3.9cm]{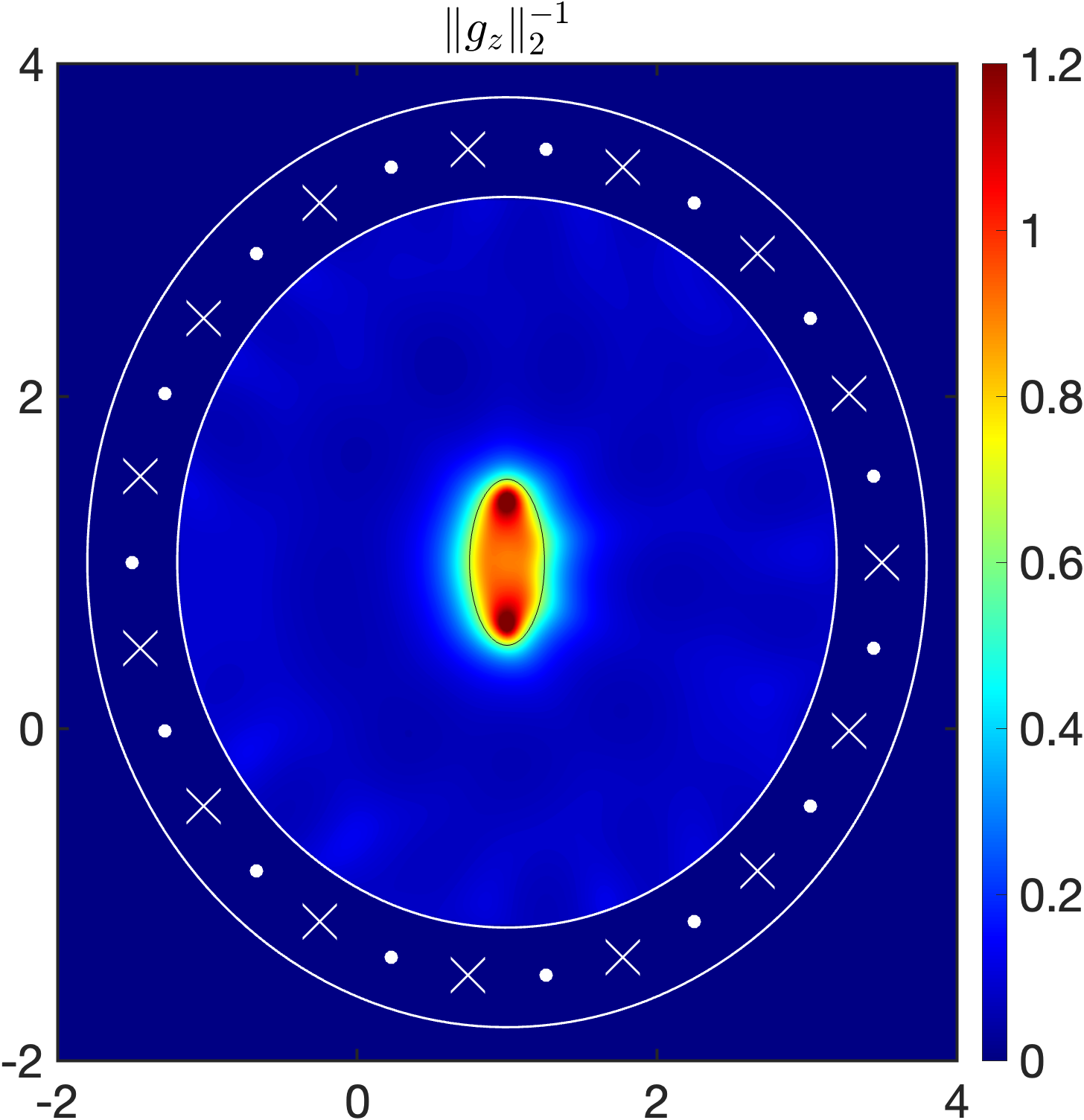}}
   \caption{\color{lxl}Reconstructions for an ellipse with measurements inside an annulus $B$. Figure \ref{Figure 3}(a)-(c) show the results with measurements on two concentric circles (radii $2.4$ and $2.6$); Figure \ref{Figure 3}(d)-(f) show the results with measurements on one circle with radius $2.5$.}
\end{figure}}
{\bf Example 2.} In the second example, we consider the cases of two scatterers using the above three indicators $I_C, I_I$ and $I_N$. These two obstacles are disks centered at $(0.25,1.75)$ and $(1.75,0.25)$ with radius $R=1/3$ and $R=1/5$, respectively. The reconstruction in Figure \ref{Figure 4} shows that our linear sampling method using passive data as well as $I_I$ and $I_N$ with passive data are all well-performed for multiple scatterers.
\begin{figure}[htbp]\label{Figure 4}
   \centering     
   \subfigure[$I_C$]{
   \includegraphics[width=0.3\linewidth, height=3.9cm]{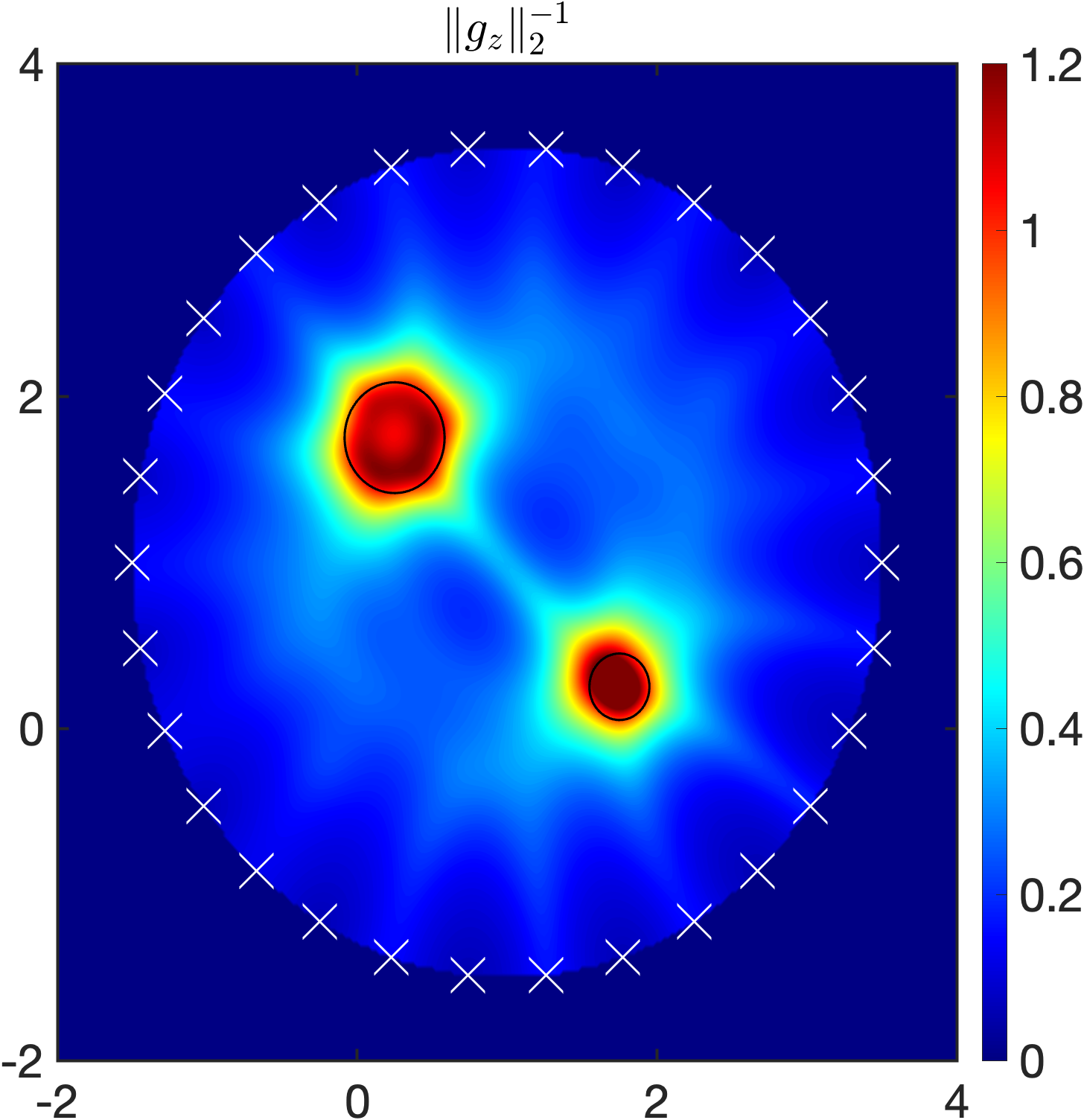}}
   \subfigure[$I_I$]{
   \includegraphics[width=0.3\linewidth, height=3.9cm]{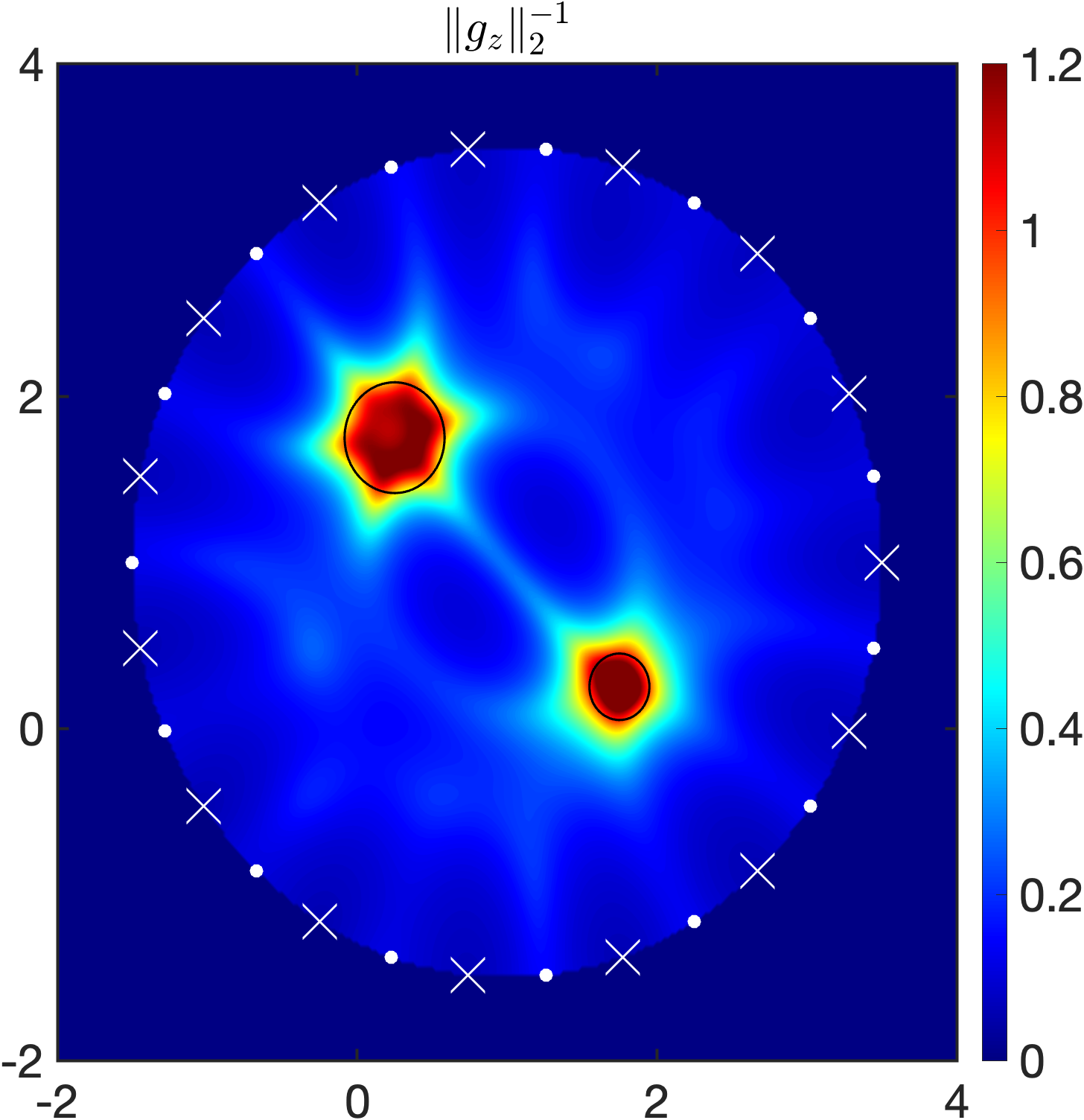}}
   \subfigure[$I_N$]{
   \includegraphics[width=0.3\linewidth, height=3.9cm]{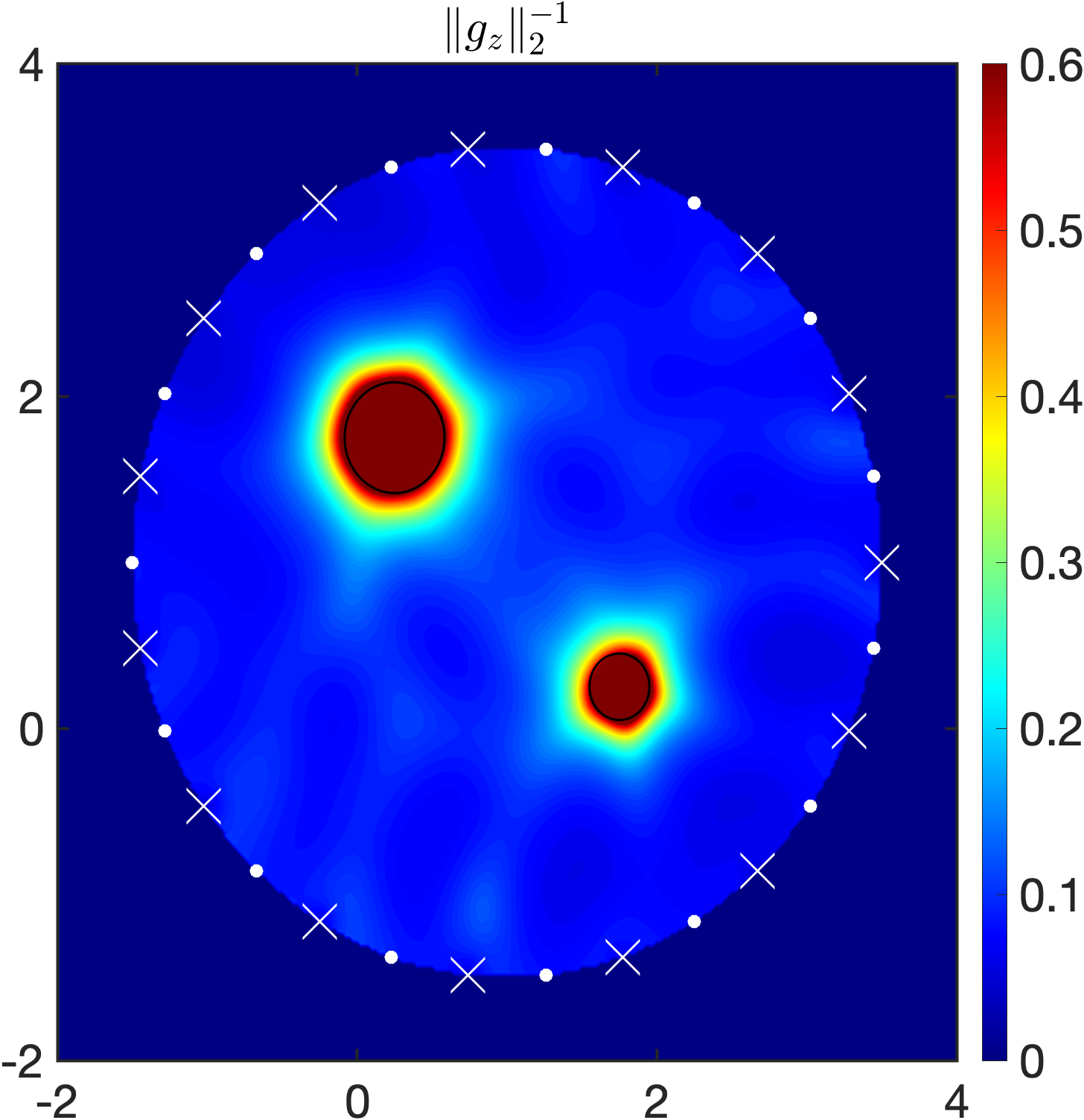}}
   \caption{Reconstruction for two obstacles by using indicators $I_C,I_I,I_N$, respectively.}
\end{figure}

{\bf Example 3.} In the third example, we investigate the influence of the random level $\beta$ in \eqref{source}. We consider a kite-shaped object located at $(1,1)$ and it is described by the parametric representation
\ben
\partial D_{\text{kite}}(\theta)=\left(1+ 0.25\mathrm{cos}\theta +0.25\mathrm{cos}(2\theta), 1+0.5\mathrm{sin}\theta\right),\quad \theta\in [0, 2\pi].
\enn 

{\color{lxl}Figure \ref{Figure 5} shows the result with $\beta=0.1, 0.5, 0.9$, respectively. We observe that when the random level $\beta$ increases (from left to right), the quality of the reconstruction deteriorates. These numerical results align with theoretical expectations. The random level $\beta$ quantifies the degree to which unknown point sources deviate from uniformly distributed locations on $\partial B_R$. In our discretized computation of $\mathcal{C}g$, the integration on $\partial B_R$ was performed under the assumption that point sources are uniformly distributed on $\partial B_R$. As the random level $\beta$ increases (indicating greater deviation from uniform distribution), the integration error grows larger under the original uniform-distribution assumption. This leads to increased computational error in $\mathcal{C}g$ and reduced accuracy. Therefore, the reconstruction performance deteriorates as $\beta$ becomes increases.}
 \begin{figure}[htbp]\label{Figure 5}
    \centering
    \subfigure[$\beta=0.1$]{
    \includegraphics[width=0.3\linewidth, height=3.9cm]{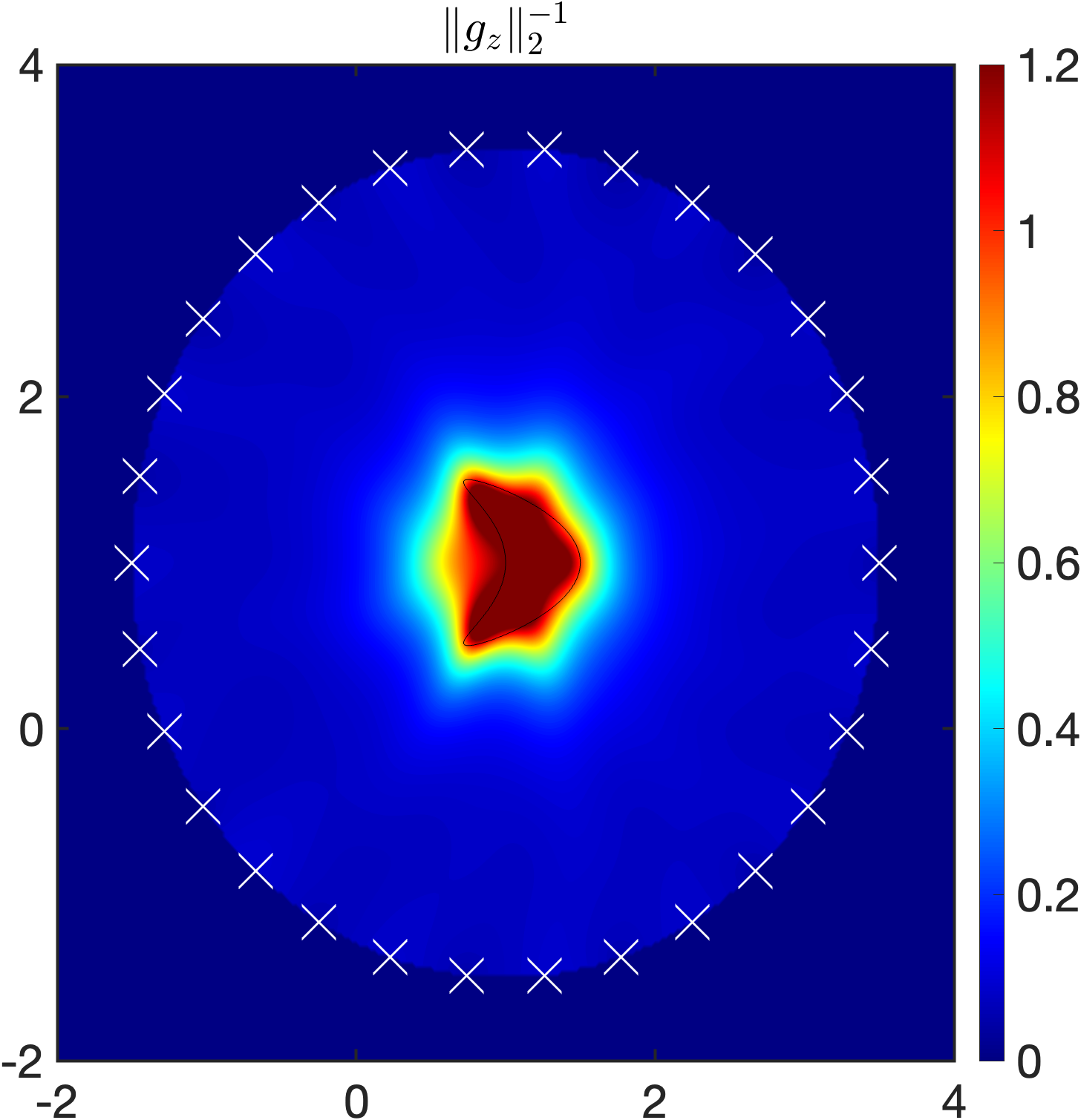}}
    \subfigure[$\beta=0.5$]{
    \includegraphics[width=0.3\linewidth, height=3.9cm]{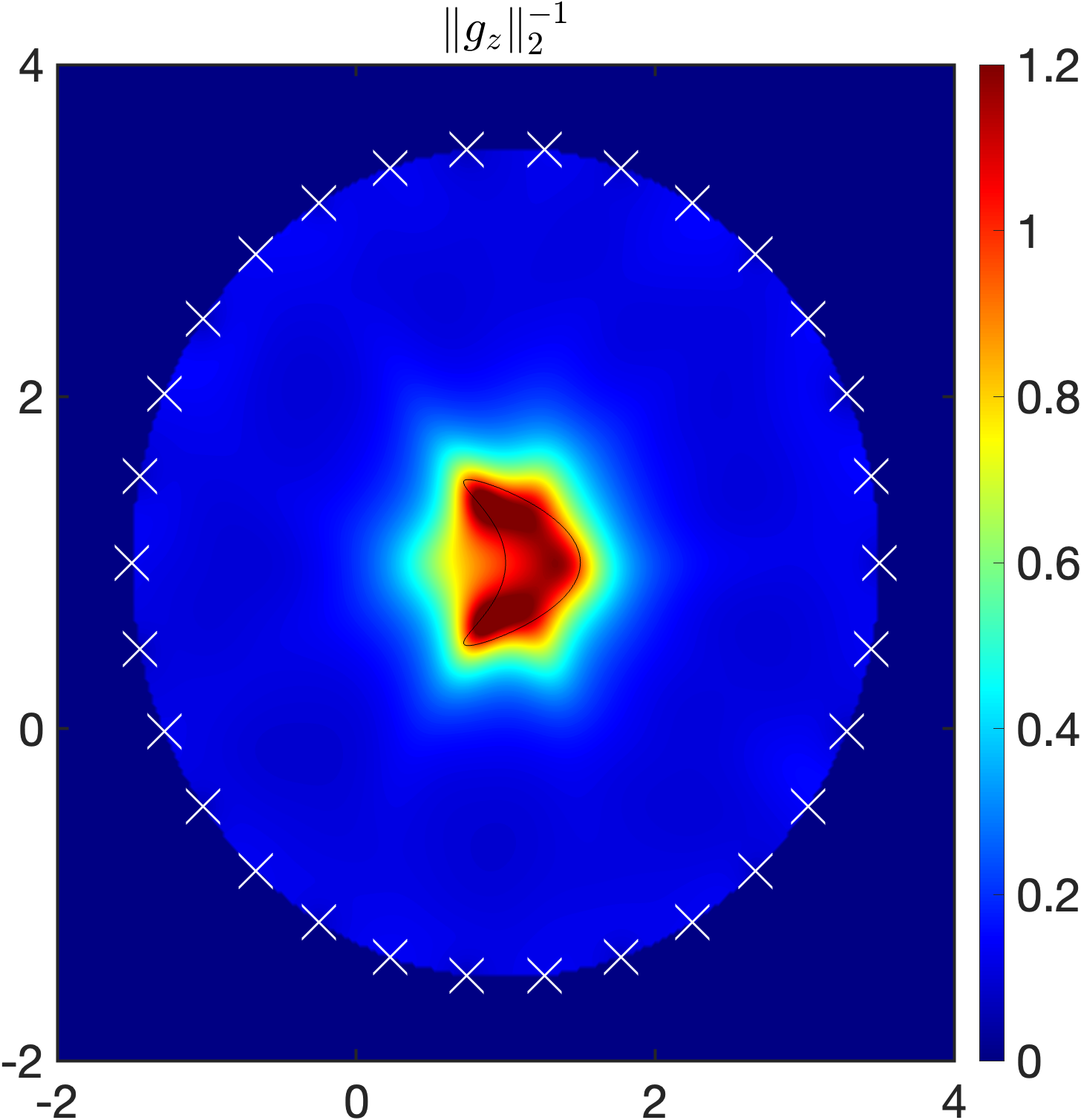}}
    \subfigure[$\beta=0.9$]{
    \includegraphics[width=0.3\linewidth, height=3.9cm]{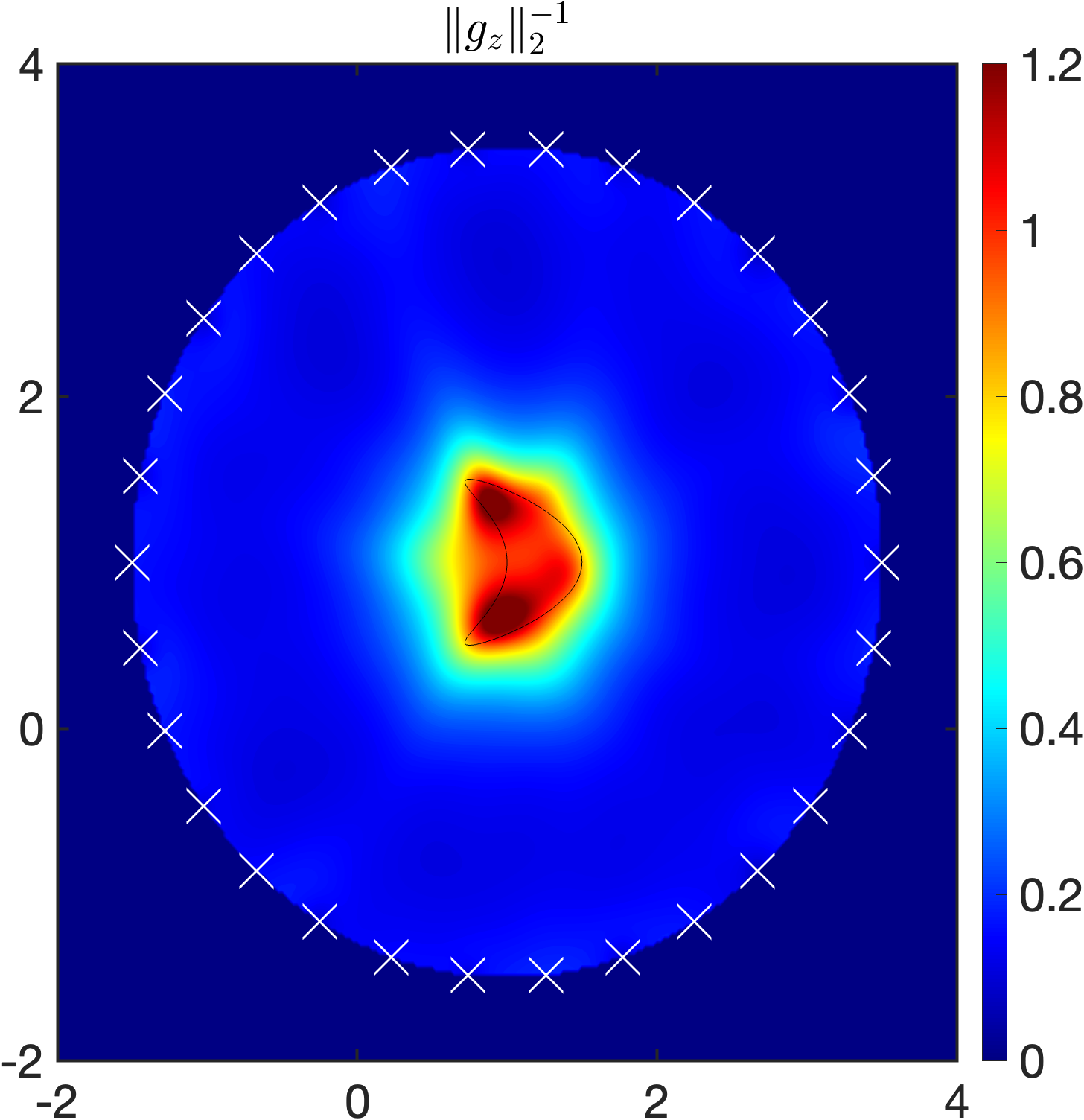}}
    \caption{{\color{lxl}Reconstructions for a kite-shaped obstacle by using indicator $I_C$ under different random levels $\beta=0.1,0.5,0.9$, respectively.}}
\end{figure}

{\color{lxl}{\bf Example 4.} In the fourth example, we consider different numbers of random point sources $z_l$. The obstacle is still assumed to be the kite-shaped one in {\bf Example 3}. Here we fix the random level $\beta=0.9$ and show the reconstructions with the number of random sources $L=80,140,200$ in Figure \ref{Figure 6}, respectively. The reconstruction is improved by using a larger number of random sources $L=140$ and $L=200$, as demonstrated in Figure \ref{Figure 6}.
\begin{figure}[htbp]\label{Figure 6}
    \centering
    \subfigure[$\beta=0.9,L=80$]{
    \includegraphics[width=0.3\linewidth, height=3.9cm]{figures_reconstruction/figure3_09.png}}
    \subfigure[$\beta=0.9,L=140$]{
    \includegraphics[width=0.3\linewidth, height=3.9cm]{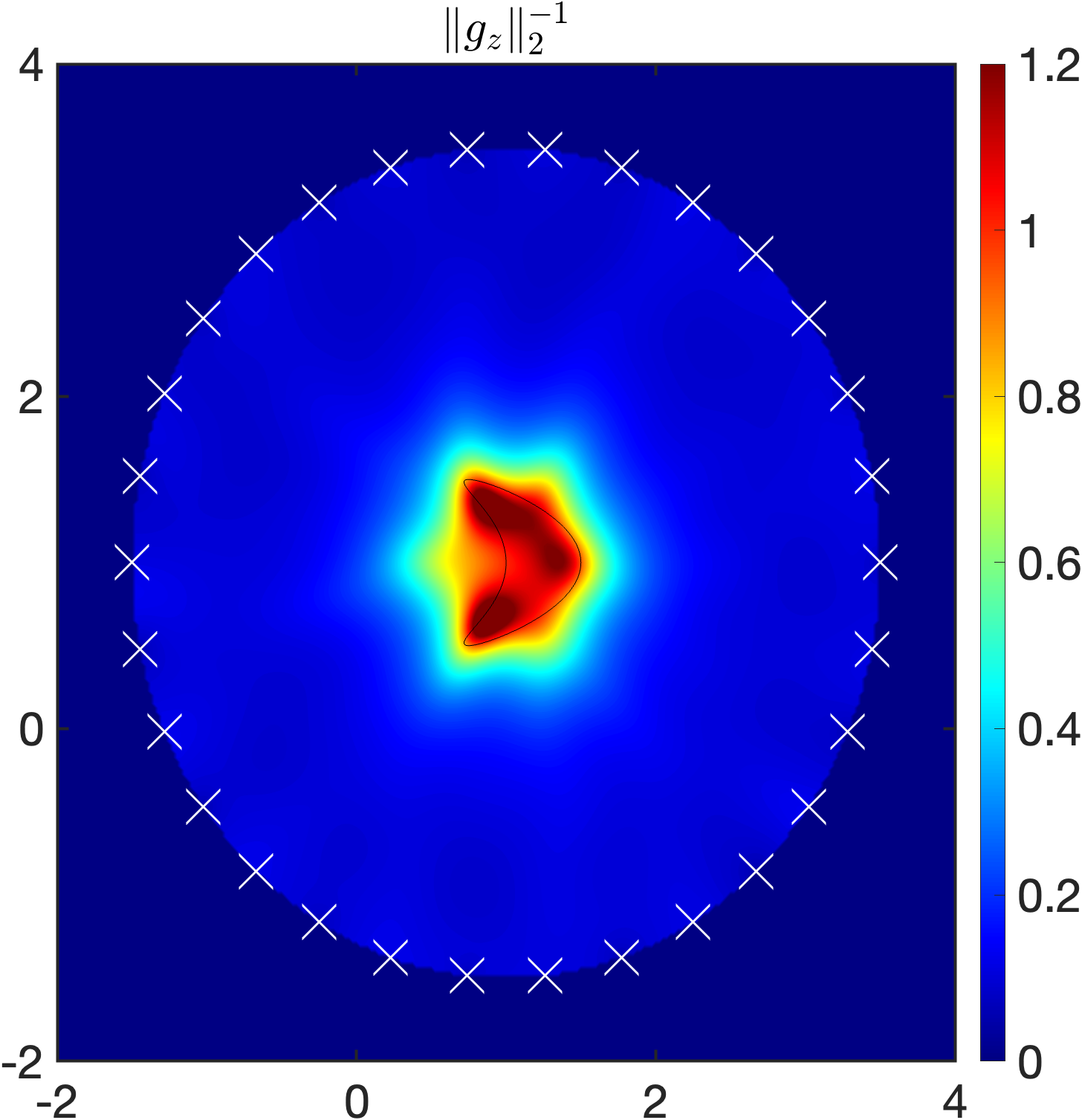}}
    \subfigure[$\beta=0.9,L=200$]{
    \includegraphics[width=0.3\linewidth, height=3.9cm]{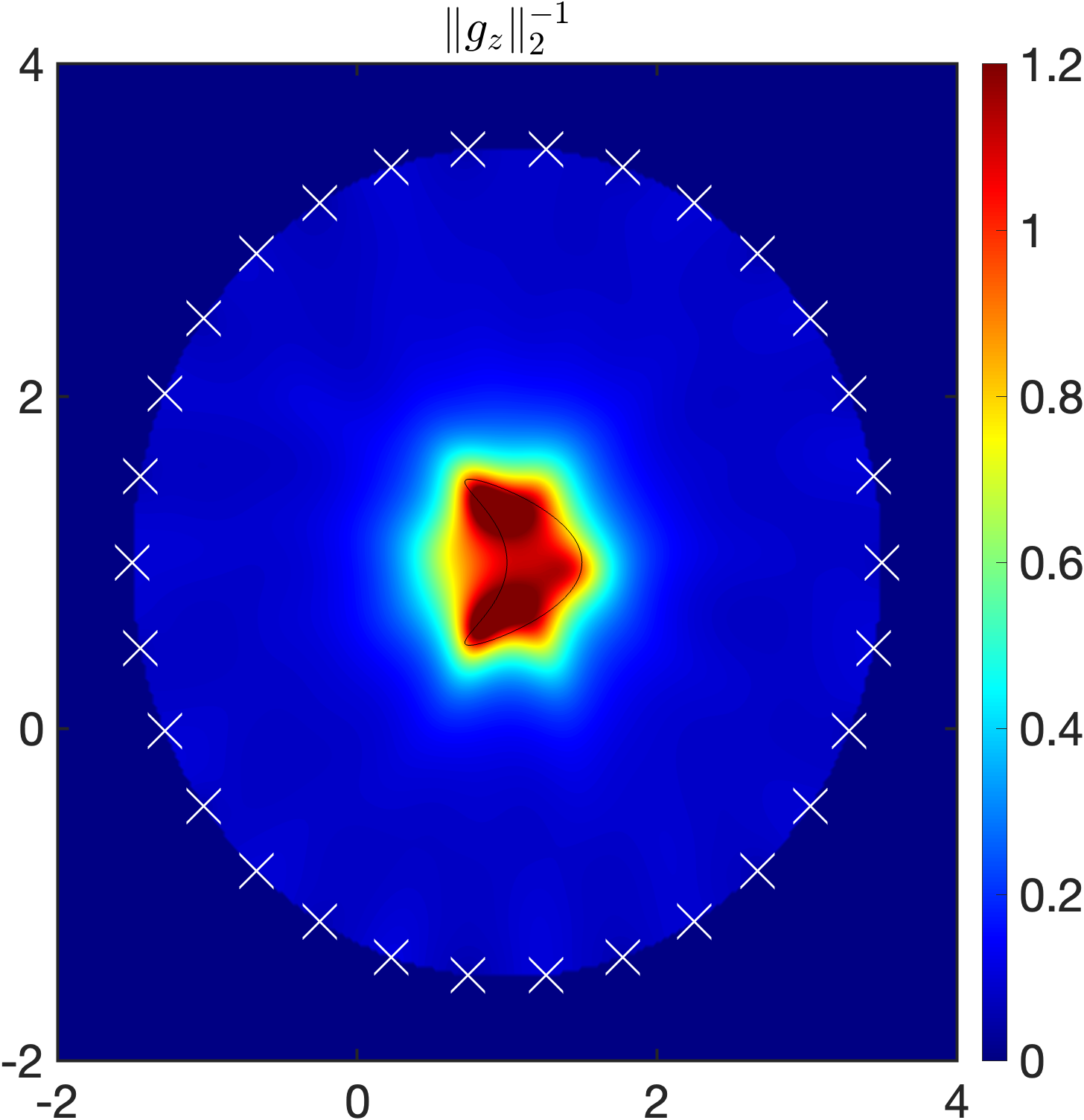}}
    \caption{\color{lxl}Reconstructions for a kite-shaped obstacle by using indicator $I_C$ under a random level $\beta=0.9$ with different numbers of random sources $L=80,140,200$, respectively.}
\end{figure}

{\bf Example 5.} In the last example, we consider limited-aperture measurements. The reconstructions are given in Figure \ref{Figure 7}. Compared with full-aperture case, limited-aperture measurement leads to poorer reconstructed results. This is expected for all numerical methods since it is a much harder inverse scattering problem. However, it can be seen that the reconstructed result using our method with passive data (see Figure \ref{Figure 7})(a)-(c)) gives comparable results to the results with active data (see Figure \ref{Figure 7}(d)-(i)).
\begin{figure}[htbp]\label{Figure 7}
    \centering
    \subfigure[$I_C$, full-aperture]{
    \includegraphics[width=0.3\linewidth, height=3.9cm]{figures_reconstruction/figure5_1.png}}
    \subfigure[$I_I$, full-aperture]{
    \includegraphics[width=0.3\linewidth, height=3.9cm]{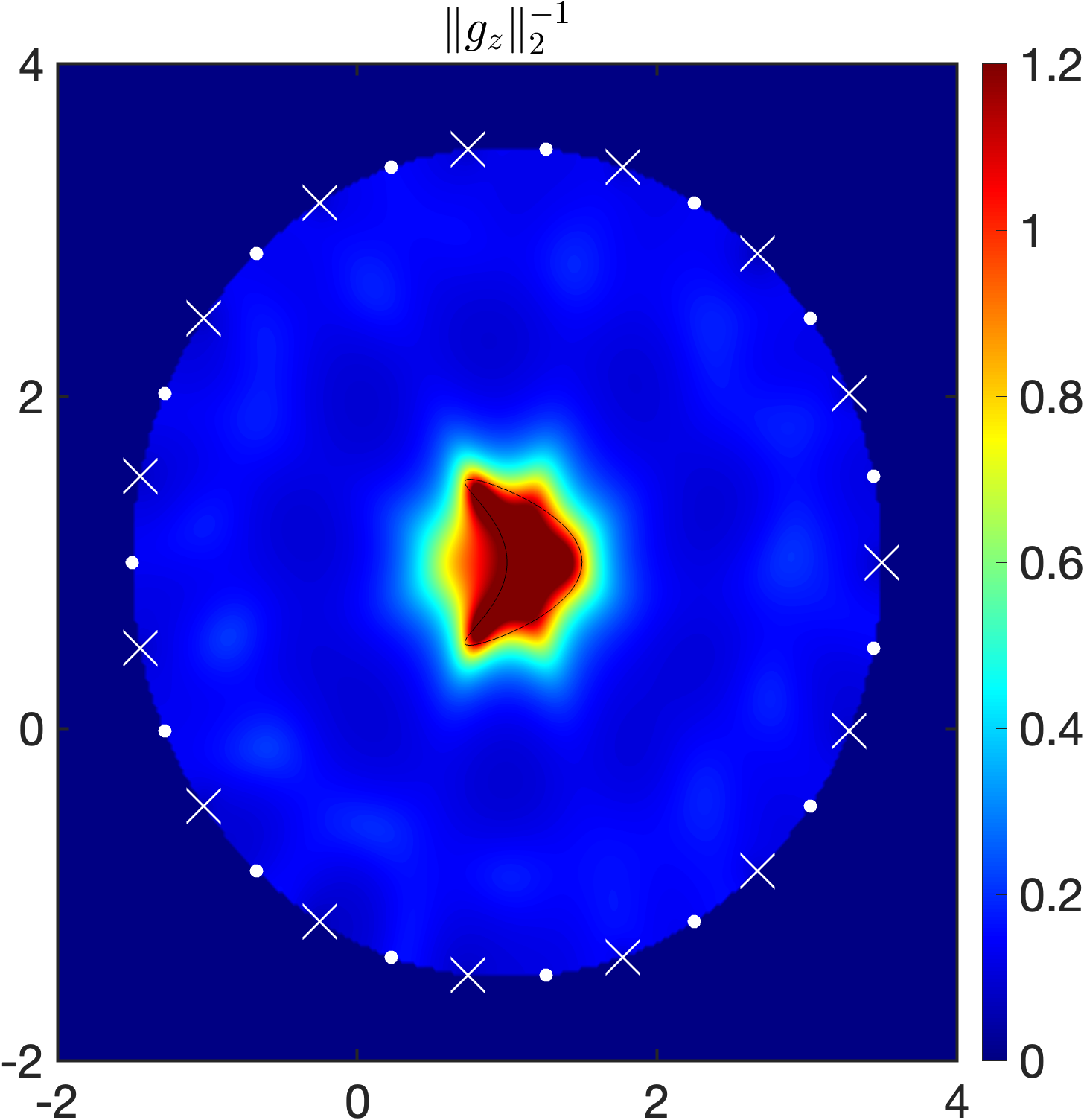}}
    \subfigure[$I_N$, full-aperture]{
    \includegraphics[width=0.3\linewidth, height=3.9cm]{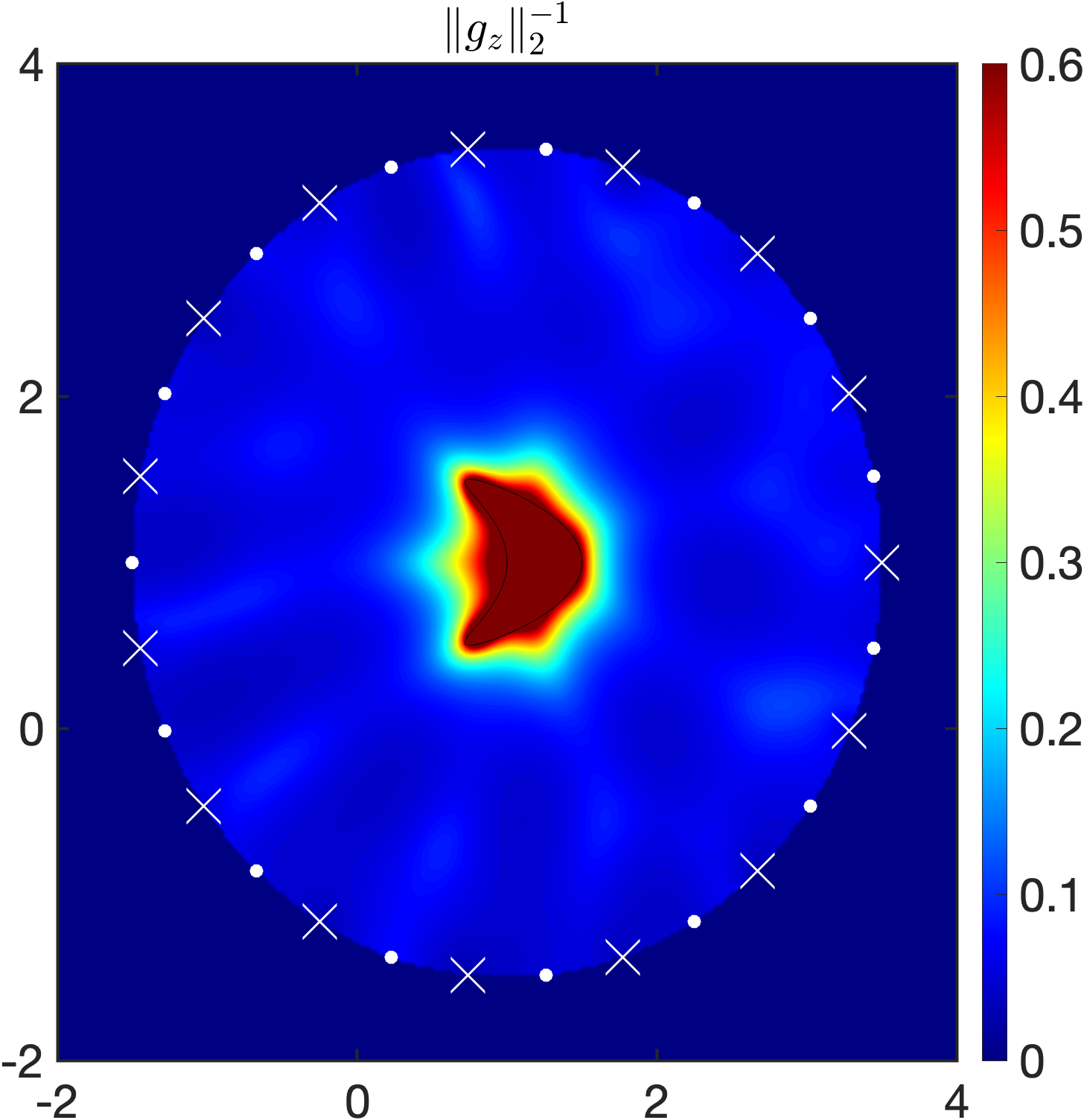}}
    
    \subfigure[$I_C$, $\left(-\frac{2\pi}{3},\frac{2\pi}{3}\right)$]{
    \includegraphics[width=0.3\linewidth, height=3.9cm]{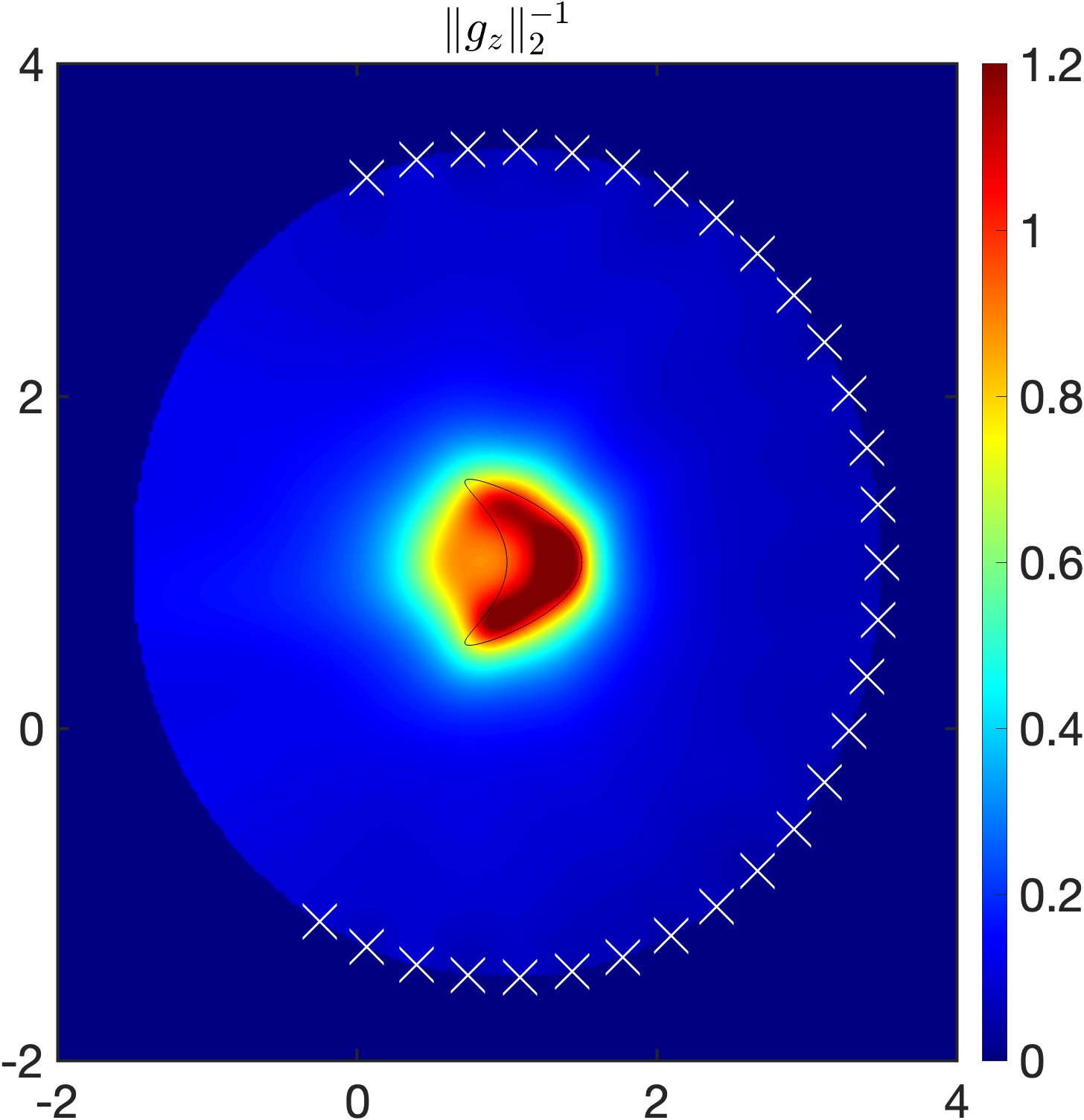}}
    \subfigure[$I_I$, $\left(-\frac{2\pi}{3},\frac{2\pi}{3}\right)$]{
    \includegraphics[width=0.3\linewidth, height=3.9cm]{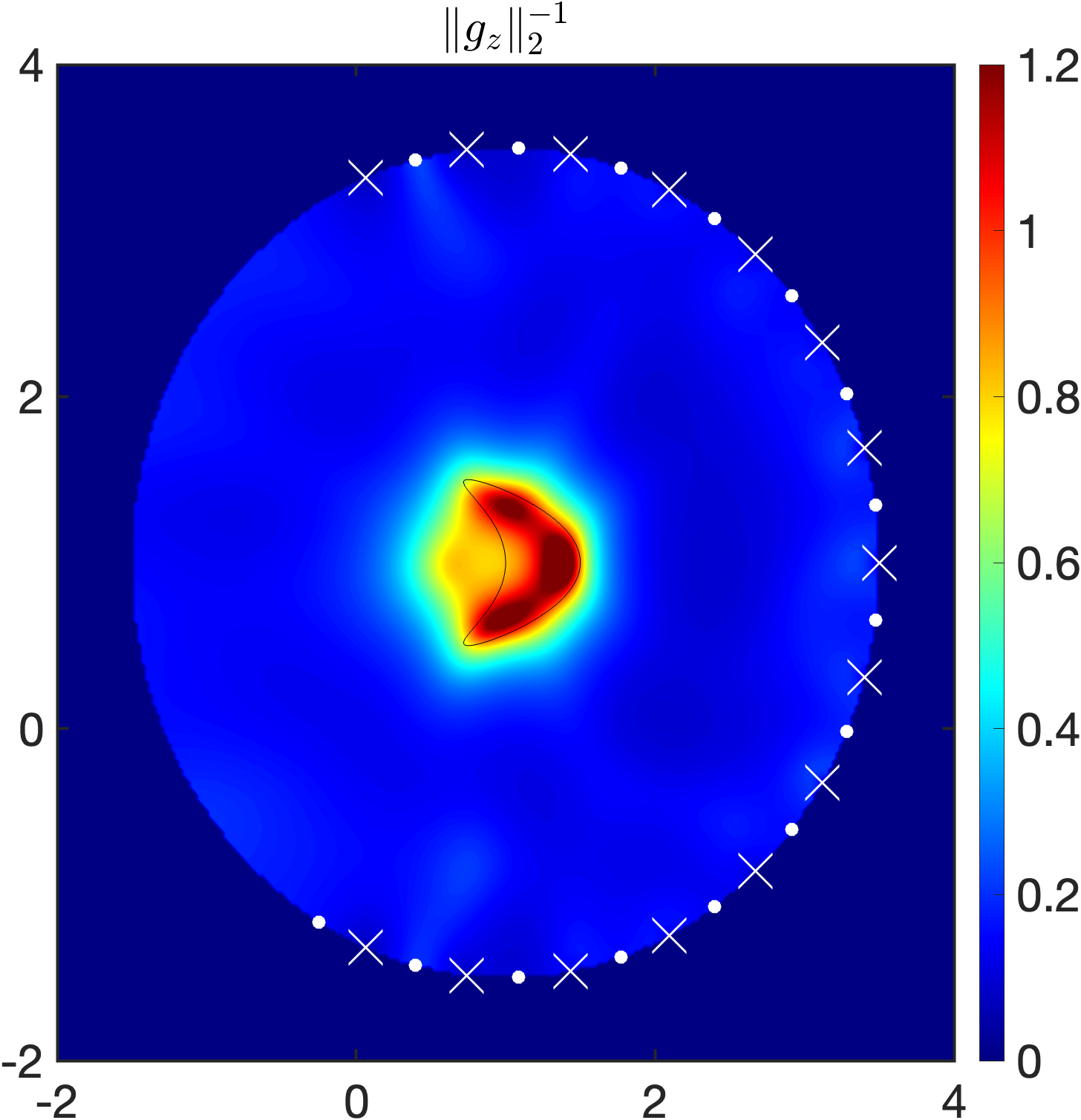}}
    \subfigure[$I_N$, $\left(-\frac{2\pi}{3},\frac{2\pi}{3}\right)$]{
    \includegraphics[width=0.3\linewidth, height=3.9cm]{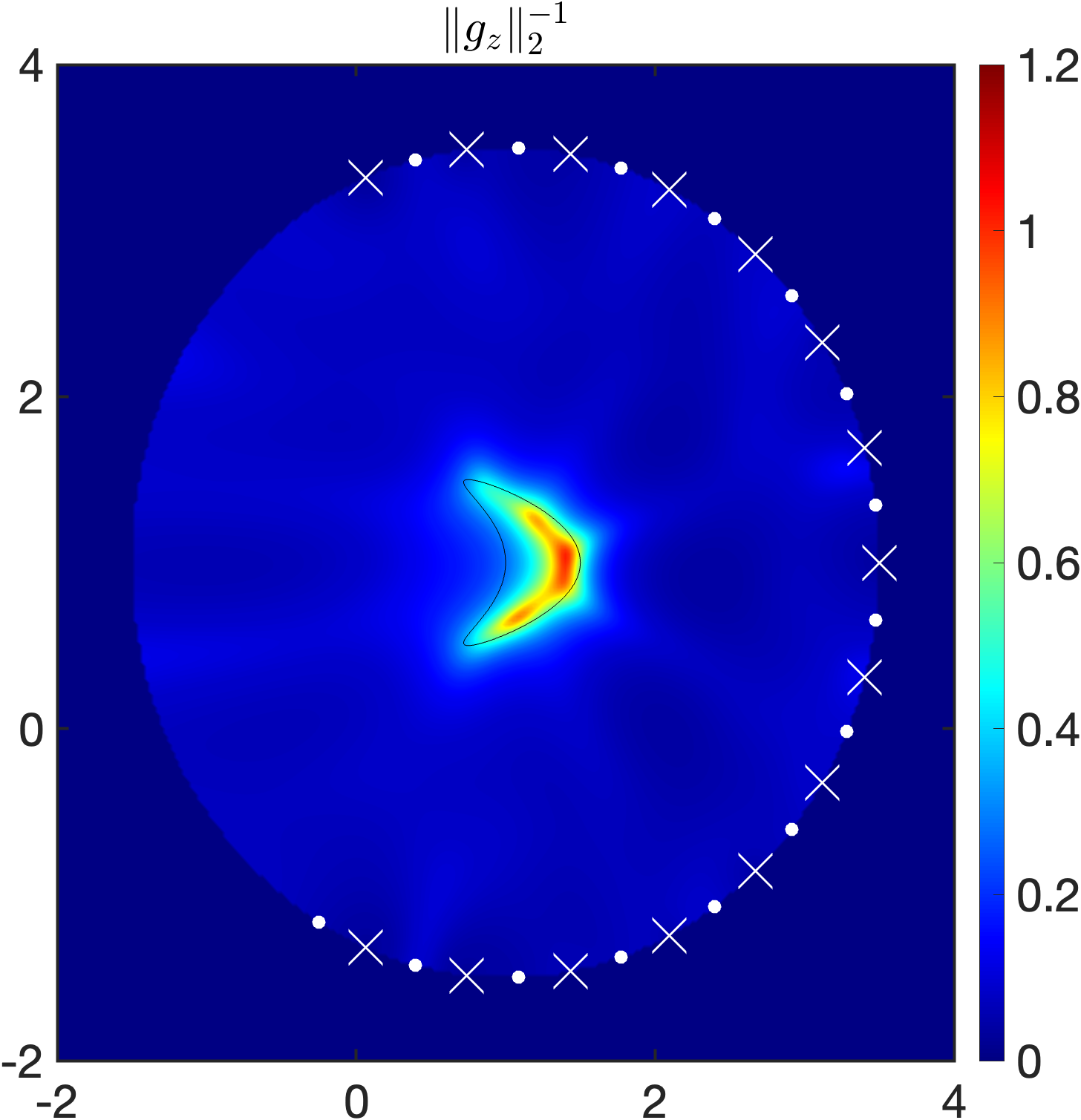}}
    
    \subfigure[$I_C$, $\left(-\frac{\pi}{3},\frac{\pi}{3}\right)$]{
    \includegraphics[width=0.3\linewidth, height=3.9cm]{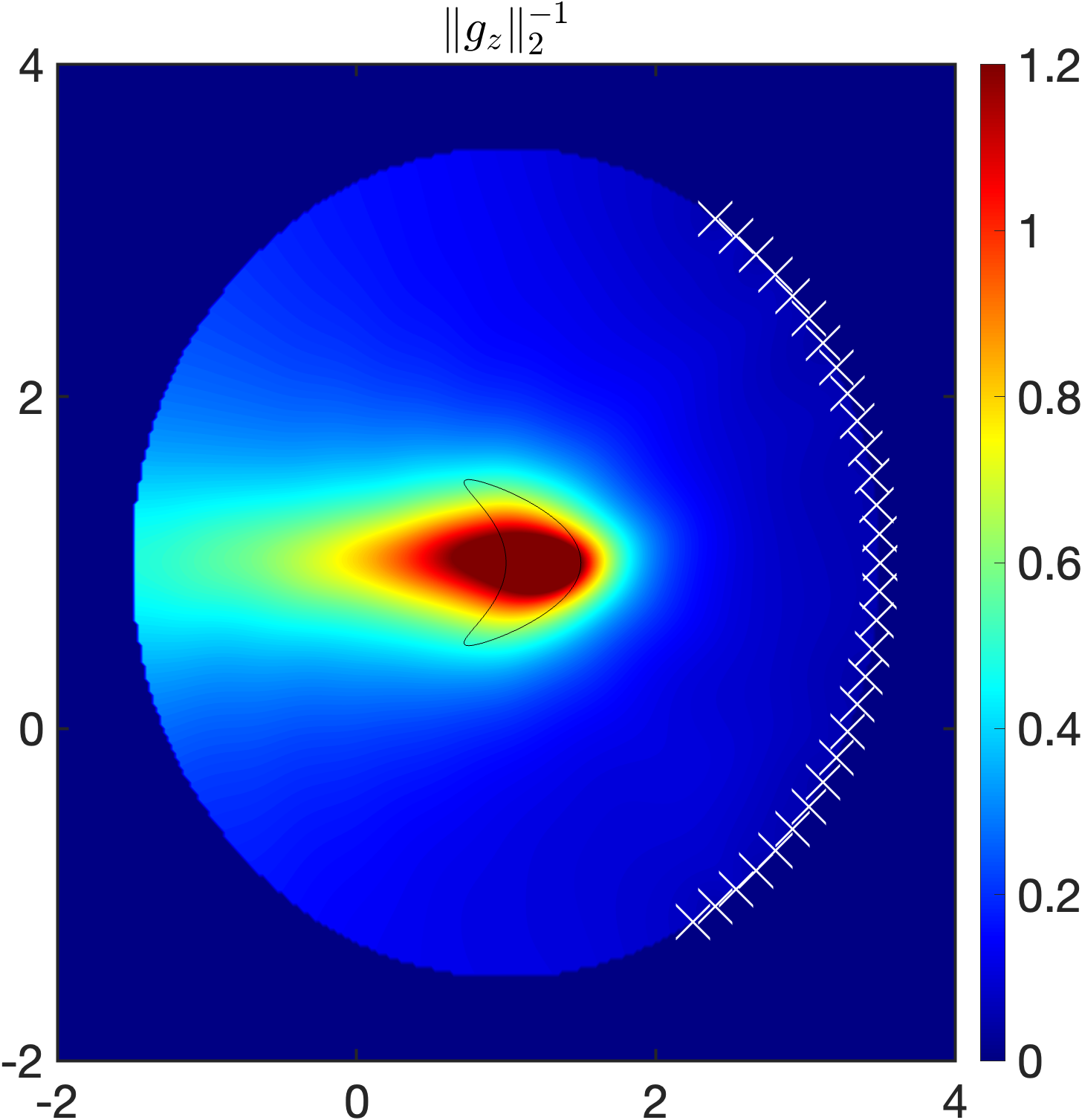}}
    \subfigure[$I_I$, $\left(-\frac{\pi}{3},\frac{\pi}{3}\right)$]{
    \includegraphics[width=0.3\linewidth, height=3.9cm]{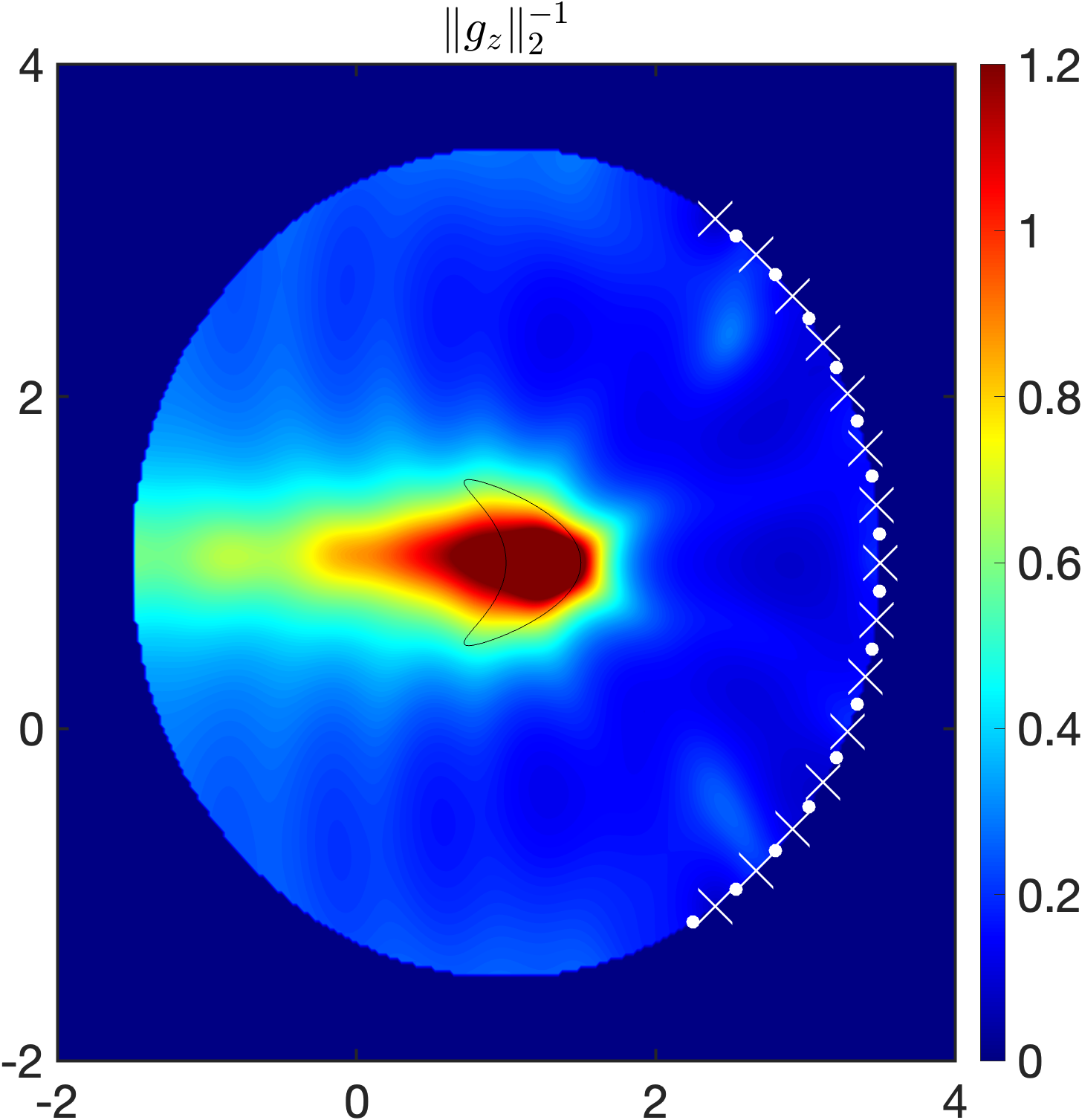}}
    \subfigure[$I_N$, $\left(-\frac{\pi}{3},\frac{\pi}{3}\right)$]{
    \includegraphics[width=0.3\linewidth, height=3.9cm]{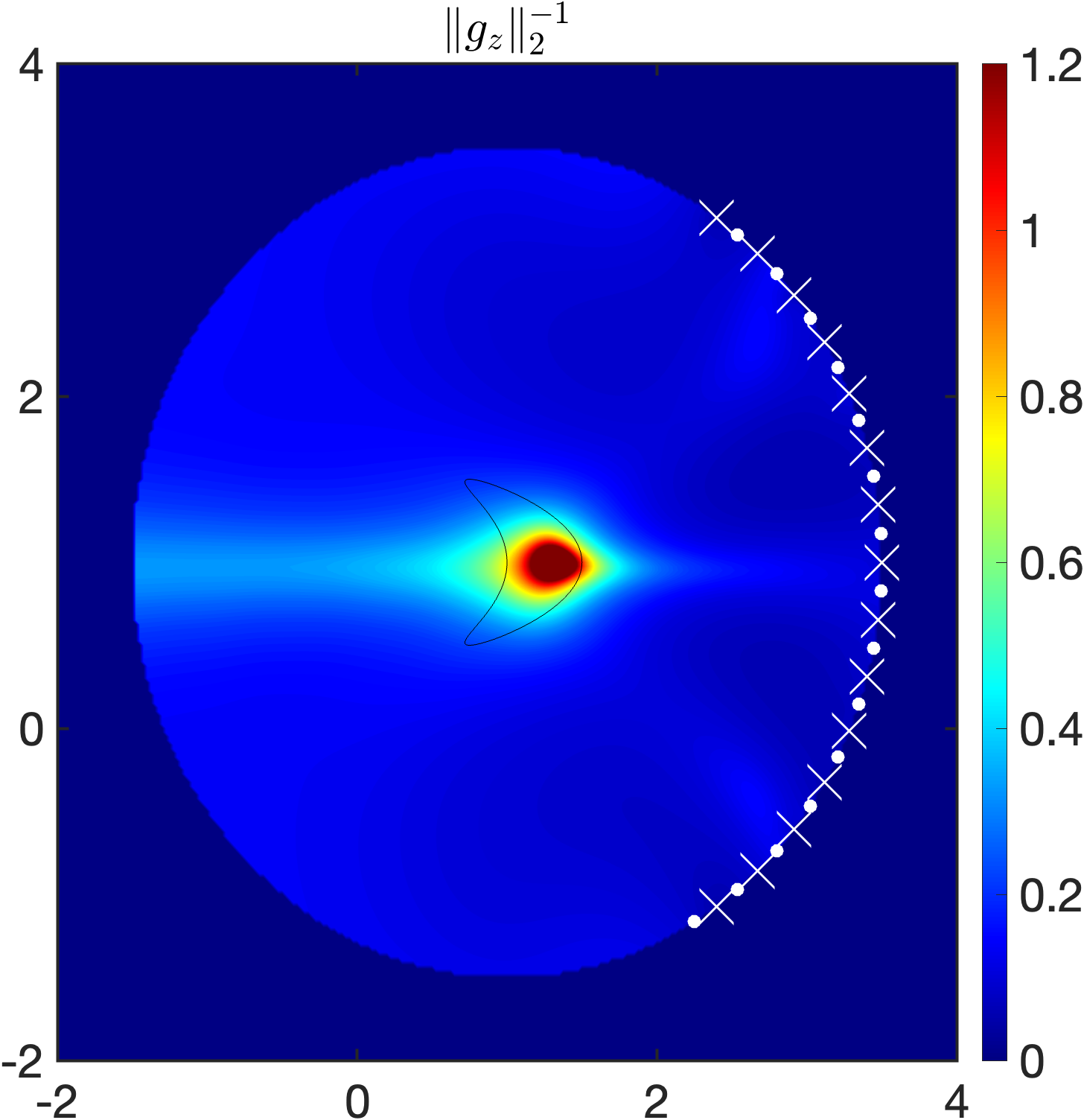}}
    \caption{\color{lxl}Reconstructions for a kite-shaped obstacle using indicators $I_C$ (the first column), $I_I$ (the second row) and $I_N$ (the third row) with limited-aperture measurements.}
\end{figure}}

\section{Conclusion}
In this work we investigate an inverse scattering problem due to random sources in the time domain. To make use of the measurement data set due to random sources, we propose an approximate data set motivated by the the Helmholtz-Kirchhoff identify in the frequency domain \cite{Garnier2023}. Particularly in our time domain case, the approximate data set involves the subtraction of two scattered wave fields, making the analysis much more complicated. Similar to \cite{CakoniHaddarLechleiter2019}, we propose to use several approximate operators to tackle these difficulties and propose a linear sampling method to reconstruct the shape of the scattering object. One key ingredient is an efficient functional framework that relates the mapping properties of Laplace domain factorized operators to their counterparts in the time domain. The proposed imaging method directly makes use of the time domain measurement due to random sources, avoiding certain difficulties in using frequency domain measurements at multiple frequencies \cite{guzina2010a}.  The capability of the proposed method is further demonstrated by numerical examples.

{\color{TJL}There are many ways in which this work could be profitably continued. We expect the similar conclusions are valid for other boundary conditions, such as Neumann and Robin boundary conditions. One could also try and extend our theoretical analysis to two-dimensional case.}
 
\appendix

\section{An approximate identity in the Laplace domain}
Let $B_R$ be a large ball such that $x,y \in B_R$. Note that both $\hat\Phi_s(y;p)$ and $\hat\Phi_s(y;q)$  are fundamental solutions, application of Green's identity yields that 
\begin{equation*}
\begin{aligned}
&\quad\int_{\partial B_R}  \bigg(\frac{\partial \overline{\hat\Phi_s(y;q)}}{\partial \nu(y)}\hat\Phi_s(y;p)   - \frac{\partial \hat\Phi_s(y;p) }{{\color{lxl}\partial\nu(y)}} \overline{\hat\Phi_s(y;q)} \bigg)ds(y)\\
&= -\hat\Phi_s(q;p)+ \overline{\hat\Phi_s(p;q)}+(s^2-\overline{s}^2)\int_{ B_R}\hat\Phi_s(y;p)\overline{\hat\Phi_s(y;q)}dy.
\end{aligned}
\end{equation*}
Note that $B_R$ is a large ball, from the asymptotic behavior of the fundamental solution, similar to \cite{Garnier2023,GarnierPapanicolaou2016} it follows that
\ben
&&-(s+\overline{s})i\int_{\partial B_R}\hat\Phi_s(y;p)\overline{\hat\Phi_s(y;q)}ds(y)\\
&\approx &-\hat\Phi_s(q;p)+ \overline{\hat\Phi_s(p;q)}+(s^2-\overline{s}^2)\int_{B_R}\hat\Phi_s(y;p)\overline{\hat\Phi_s(y;q)}dy.
\enn
Choose a sufficiently large radius $\tilde{R}$ such that  the ball $B_{\tilde{R}}(p)$ contains $B_R$. We can estimate the $L^2$-norm of $\hat\Phi_s(\cdot;p)$ in $B_{\tilde{R}}(p)$ by $\left\Vert\hat\Phi_s(\cdot;p)\right\Vert_{L^2(B_{\tilde{R}}(p))}^2
\le \frac{1}{16\pi^2}\int_{B_{\tilde{R}}(p)}\frac{e^{-2\sigma|p-y|}}{|p-y|^2}dy$
which remains bounded as $\sigma \to 0$.
The $L^2$-norm of $\overline{\hat\Phi_s(\cdot;q)}$ can be estimated similarly. This allows to show that $
(s^2-\overline{s}^2)\int_{B_R}\hat\Phi_s(y;p)\overline{\hat\Phi_s(y;q)}dy$ approaches to $0$ as $\sigma \to 0$. Finally, for small $\sigma$, it allows to conclude that
\ben
-(s+\overline{s})i\int_{\partial B_R}\hat\Phi_s(y;p)\overline{\hat\Phi_s(y;q)}ds(y) \approx -\hat\Phi_s(q;p)+ \overline{\hat\Phi_s(p;q)}.
\enn

\end{document}